\documentclass[11pt]{amsart}
\usepackage{amssymb,amsmath}
\usepackage{graphicx}
\usepackage{subcaption}
\usepackage{bm,bbm}
\usepackage{color}
\usepackage{tikz}
\usepackage{pgfplots}
\usepackage{array,blkarray}
\usepackage[hidelinks]{hyperref}
\usepackage{cleveref}
\usepackage{mathtools}
\usepackage{enumitem}
\usepackage{tabularray}
\usepackage{booktabs}
\theoremstyle{theorem}
\newtheorem{theorem}{Theorem}

\newtheorem{corollary}[theorem]{Corollary}
\theoremstyle{definition}

\newtheorem{lemma}[theorem]{Lemma}
\theoremstyle{remark}
\newtheorem{remark}[theorem]{Remark}

\numberwithin{theorem}{section}
\numberwithin{equation}{section}
\numberwithin{table}{section}
\numberwithin{figure}{section}

\AddToHook{env/corollary/begin}{\crefalias{theorem}{corollary}}
\AddToHook{env/lemma/begin}{\crefalias{theorem}{lemma}}
\AddToHook{env/proposition/begin}{\crefalias{theorem}{proposition}}
\AddToHook{env/remark/begin}{\crefalias{theorem}{remark}}
\AddToHook{env/example/begin}{\crefalias{theorem}{example}}

\newcommand{\V}{\ensuremath{\mathcal{V}}}
\newcommand{\Q}{\ensuremath{\mathcal{Q}} }
\def\R{\mathbb{R}}

\definecolor{myBlue1}{RGB}{101,149,239}  % cornflowerblue
\definecolor{myBlue2}{RGB}{113,104,238} % medium slate blue	
\definecolor{myBlue3}{RGB}{30,144,255} % dodger blue
\definecolor{myGreen1}{RGB}{154,204,50} % yellow green
\definecolor{myGreen2}{RGB}{69,169,0} % chatreuse
\definecolor{myGreen3}{RGB}{154,205,50} % olive Drap
\definecolor{myGreen4}{RGB}{105,139,34} % olive Drap2
\definecolor{myRed1}{RGB}{210,105,30} % chocolate
\definecolor{myRed2}{RGB}{165,42,42} % brown
\definecolor{myRed3}{RGB}{139,26,26} % firebrick
\definecolor{lightgray}{RGB}{175,175,175} % light gray
\definecolor{myLGray}{RGB}{225,225,225} % light light gray
\definecolor{mycolor0}{rgb}{0.66,0.66,0.66}% grey
\definecolor{mycolor4}{rgb}{0.00000,0.44700,0.74100}% blue
\definecolor{mycolor1}{rgb}{0.85000,0.32500,0.09800}% orange
\definecolor{mycolor2}{rgb}{0.92900,0.69400,0.12500}% yellow
\definecolor{mycolor3}{rgb}{0.67000,0.74700,0.14100}% green
\definecolor{mycolor5}{rgb}{0.49400,0.18400,0.55600}% violet
\definecolor{mycolor6}{rgb}{0.85000,0.32500,0.09800}%
\DeclareMathOperator{\tr}{tr}

\DeclareMathOperator{\diag}{diag}

\newcommand{\bdf}{\textnormal{\fontseries{lc}\selectfont\ttfamily BDF}}
\newcommand{\ddf}{\textnormal{\fontseries{lc}\selectfont\ttfamily DDF}}

\newcommand{\calB}{\ensuremath{\mathcal{B}} }

\newcommand{\calD}{\ensuremath{\mathcal{D}} }

\newcommand{\calG}{\ensuremath{\mathcal{G}} }
\newcommand{\calJ}{\ensuremath{\mathcal{J}} }
\newcommand{\calK}{\ensuremath{\mathcal{K}} }

\newcommand{\calR}{\ensuremath{\mathcal{R}} }

\newcommand{\calT}{\ensuremath{\mathcal{T}} }

\def\dr{\,\text{d}r}
\def\ds{\,\text{d}s}
\def\dt{\,\text{d}t}
\def\dx{\,\text{d}x}

\newcommand{\ddt}{\ensuremath{\frac{\text{d}}{\text{d}t}} }

\newcommand{\bmat}[1]{\begin{bmatrix}#1\end{bmatrix}}
\newcommand{\bsmat}[1]{\begin{bsmallmatrix}#1\end{bsmallmatrix}}
\newcommand{\wt}{\widetilde}

\newcommand{\tsum}{\bgroup\textstyle\sum\egroup}

\DeclarePairedDelimiter{\set}{\{}{\}}
\DeclarePairedDelimiter{\pset}{(}{)}

\DeclarePairedDelimiter{\abs}{\lvert}{\rvert}
\DeclarePairedDelimiter{\aset}{\langle}{\rangle}
\DeclarePairedDelimiter{\bigaset}{\big\langle}{\big\rangle}
\DeclarePairedDelimiter{\norm}{\lVert}{\rVert}

\begin{document}
\title[Second-order Bulk--surface Splitting]{Second-order Bulk--surface Splitting for the Wave Equation with Kinetic Boundary Conditions}
\author[]{R.~Altmann$^\dagger$, R.~Morandin$^\dagger$}
\address{${}^{\dagger}$ Institute of Analysis and Numerics, Otto von Guericke University Magdeburg, Universit\"atsplatz 2, 39106 Magdeburg, Germany}
\email{\{robert.altmann,riccardo.morandin\}@ovgu.de}
\thanks{Research funded by the support of the Deutsche Forschungsgemeinschaft (DFG, German Research Foundation) through the project 446856041. }
\date{\today}
%
%
%=============================================================================
%=========  Abstract
%=============================================================================
\begin{abstract}
This paper is devoted to the numerical analysis of a second-order bulk--surface splitting scheme for the semi-linear wave equation with kinetic boundary conditions. 
The construction is based on the interpretation of the equations as coupled system and the implementation of different difference formulae for the discrete states depending on their exact position in the system equations. This results in a $4$-step scheme which decouples bulk and surface dynamics. 
We prove energy stability and second-order convergence under a weak CFL condition and validate these results also numerically. 
\end{abstract}
%
%=============================================================================
%=========  Title / Contents
%=============================================================================
\maketitle
%\setcounter{tocdepth}{2}
%\tableofcontents
%
{\tiny{\bf Key words.} kinetic boundary conditions, semi-linear wave equation, bulk--surface splitting}\\
\indent
{\tiny{\bf AMS subject classifications.} {\bf 65M20}, {\bf 65M12}, {\bf 65L80}} 
% 65J08 - Num Ana in abs Spaces - abstract evolution equations
% 65J10 - Num Ana in abs Spaces - lin Operatrs
% 65M12 - numerics for PDE IBVPs - stability and convergence 
% 65M20 - numerics for PDE IBVPs - method of lines
% 65L80 - numerics of ODEs - DAEs
%
%
%=============================================================================
%=========  Introduction
%=============================================================================
\section{Introduction}
Non-standard boundary conditions arise naturally in models where the interaction between a system and its boundary plays a critical role.
In the case of wave equations, these are typically obtained by taking into account the effect of the momentum on the boundary, which is important for fluid--structure interaction as well as acoustic--elastic couplings~\cite{Hip17}.
%
%If the momentum of a wave on the boundary is taken into account, one obtains so-called non-standard boundary conditions, which 
% capture certain boundary effects
%are important for fluid--structure interaction as well as acoustic--elastic couplings~\cite{Hip17}.
Further applications include the stabilization of wave equations via a boundary feedback law~\cite{KomZ90}, the modeling of membrane vibrations of a bass drum~\cite{Vit15}, or separation processes in mixtures of two materials~\cite{GarK20}. 
% this approach enables to model one part of a coupled problem as a boundary layer, i.e., one wave system may be replaced by a wave-type equation on the boundary~\cite{Lie13,Hip17}. % Lie13 is only parabolic
Comparable boundary conditions also appear for parabolic problems~\cite{VraS13a,KovL17} or separation processes modeled by the Cahn--Hilliard equation~\cite{Gal07,GolMS11,GarK20}. 

%% kinetic bc
In this paper, we focus on a specific class of non-standard boundary conditions for the wave equation, namely {\em kinetic boundary conditions}. For this, the bulk equation is closed with a second wave equation on the surface of the considered domain. Such formulations can be derived in terms of energy balances and constitutive laws and allow a physical interpretation in the one-dimensional setting, cf.~\cite{Gol06}. 
For corresponding existence and well-posedness results, we refer to~\cite{Vit13,GraL14,Hip17}.

%% numerial treatment
In recent years, the numerical approximation of systems with non-standard boundary conditions has also received growing attention. 
% space
For the spatial discretization, one usually applies bulk--surface finite elements as introduced in~\cite{EllR13}. With regard to wave systems, this has been rigorously analyzed in~\cite{HipHS18,HocL20,HipK21},  
% HipHS18: acoustic; HocL20: kinetic; HipK21: L2-estimates 
showing different rates depending on the exact structure of boundary conditions (e.g.~due to advection or strong damping terms).
% time
Regarding the temporal discretization, the Crank--Nicolson method shows second-order convergence. This holds for the original fully implicit version (cf.~\cite{Hip17}) as well as for the implicit--explicit variant introduced in~\cite{HocL21}. 
% General (Gauss--)Runge--Kutta methods were considered in~\cite{HipK20}. 
However, even if the nonlinearity is treated explicitly, a coupled linear system needs to be solved in every time step.  

%% content of paper
To circumvent this issue, we propose a novel bulk--surface splitting scheme of second order for the wave equation with kinetic boundary conditions. In our approach, the equations are interpreted as a coupled system of bulk and surface dynamics and modeled as {\em partial differential--algebraic equation}; see~\cite{LamMT13} for an introduction. Note that such an approach allows different discretizations in the bulk and on the surface, which is beneficial if the dynamics %in the bulk and on the surface
have different characteristic length or time scales, cf.~\cite{AltV21}. 
The design of the proposed scheme is inspired by~\cite{AltKZ22}, where parabolic systems were considered. More precisely, we use function evaluations at former time points in order to decouple the system. The resulting four-step scheme achieves, as we prove in this paper, energy stability and second-order accuracy under a (weak) CFL condition of the form~$\tau \lesssim \sqrt{h}$, where $\tau$ and $h$ denote the temporal and spatial discretization parameters, respectively. 
\smallskip

%% notation
\emph{Notation.} 
Throughout the paper, we write $(u,v)$ for concatinated vectors, i.e., $(u,v) = [u^T, v^T]^T$. Moreover, $\|u\|_{M}$ denotes the classical $M$-norm with $\|u\|_M^2 = u^T Mu$. Based on this, we define 
\[
	\|u\|^p_{L^p(s,t),M}
	= \int_s^t \|u\|_M^p \dt, \qquad 
	\|u\|_{L^p,M}
	= \|u\|_{L^p(0,T),M}
\]
for $p < \infty$ as well as $\|u\|_{L^\infty,M} = \sup_{t} \|u(t)\|_{M}$.
%
%
%=============================================================================
%=========  Models
%=============================================================================
\section{The Wave Equation with Kinetic Boundary Conditions}
\label{sect:formulation}
This section is devoted to the introduction of the system equations which are in focus of this paper. This includes a model problem with kinetic boundary conditions as well as an extended model including possible dissipation terms. 

In preparation for the construction of the splitting scheme in Section~\ref{sect:schemesKinetic}, we formulate the system as a coupled problem. Hence, we consider two dynamic equations -- one in the bulk~$\Omega$ and one on the surface~$\Gamma\coloneqq\partial\Omega$ -- which we couple through a simple constraint, incorporated by the Lagrange multiplier technique. For this, we introduce the function spaces 
\[
	\V 
	\coloneqq H^1(\Omega)\times H^1(\Gamma), \qquad
	\Q 
	\coloneqq H^{-1/2}(\Gamma).
\]
%
%
%======================================================================
\subsection{Model problem with kinetic boundary conditions}
\label{sect:formulation:modelProblem}
On the bounded domain~$\Omega$ with $C^2$ boundary~$\Gamma\coloneqq\partial\Omega$, we consider the system 
\begin{subequations}
	\label{eq:kineticBC}
	\begin{align}
		\ddot u - \Delta u 
		&= f_\Omega(t,u)\qquad \text{in }\Omega, \label{eq:kineticBC:a}\\
		\ddot u - \Delta_\Gamma u + u + \partial_n u 
		&= f_\Gamma(t,u)\qquad \text{on }\Gamma  \label{eq:kineticBC:b}
	\end{align}
with initial conditions
	\begin{align}
		u(0)
		= u^0, \qquad 
		\dot u(0) 
		= \dot u^0.
	\end{align}
\end{subequations}
Therein, %$\beta_\Gamma, \kappa_\Gamma$ are non-negative constants,
$\partial_n$ denotes the derivative in normal direction of the boundary and $\Delta_\Gamma$ is the so-called {\em Laplace--Beltrami operator}, cf.~\cite[Ch.~16.1]{GilT01}. Moreover, $f_\Omega$ and $f_\Gamma$ denote (sufficiently smooth) nonlinearities. For a corresponding abstract formulation, we refer to~\cite{HipK21,HocL20}. 

For the formulation as coupled system, we follow~\cite{Alt23} and introduce $p$ as the trace of $u$. This leads to the equivalent system 
\begin{align*}
	\ddot u - \Delta u 
	&= f_\Omega(t,u)\qquad \text{in }\Omega,\\
	\ddot p - \Delta_\Gamma p + p + \partial_n u 	
	&= f_\Gamma(t,p)\qquad \text{on }\Gamma, \\
	u - p
	&= 0\hspace{1.8cm} \text{on }\Gamma
\end{align*}
with initial conditions
\begin{align*}
	u(0)
	= u^0, \quad 
	\dot u(0) 
	= \dot u^0, \qquad
	p(0)
	= u^0|_\Gamma, \quad 
	\dot p(0) 
	= \dot u^0|_\Gamma.
\end{align*}
Note that these conditions are consistent with the constraint $u = p$ on $\Gamma$. For an abstract formulation, we introduce the operators 
\[
	\calK_\Omega\colon H^1(\Omega)\to[H^1(\Omega)]^*,\qquad
	\calK_\Gamma\colon H^1(\Gamma)\to[H^1(\Gamma)]^*, \qquad 
	\calB\colon \V\to\Q^*=H^{1/2}(\Gamma)
\]
by
\[
	\aset{\calK_\Omega u,v} 
	= \int_\Omega\nabla u\cdot\nabla v\dx, \qquad 
	\aset{\calK_\Gamma p,q} 
	= \int_\Gamma\nabla_\Gamma p\cdot\nabla_\Gamma q\dx + \int_\Gamma pq\dx 
\]
and $\calB(u,p) = p - \tr u$. This then leads to the system 
\begin{subequations}
\label{eq:kinetic:PDAE}
\begin{alignat}{3}
	\bmat{\ddot u \\ \ddot p} 
	+ \bmat{\calK_\Omega \\ & \calK_\Gamma} & \bmat{u \\ p} 
	+ \calB^*\lambda 
	&&= \bmat{f_\Omega(t,u) \\ f_\Gamma(t,p)} \qquad && \text{in }\V^*, \\
	\calB & \bmat{u \\ p} &&= 0 \qquad && \text{in }\Q^*.
\end{alignat}
\end{subequations}
\begin{remark}
Under sufficient regularity assumptions, system~\eqref{eq:kinetic:PDAE} is equivalent to the original PDE~\eqref{eq:kineticBC} and the Lagrange multiplier satisfies $\lambda = \partial_n u$. 
\end{remark}
\begin{remark}
The block operator $\diag(\calK_\Omega, \calK_\Gamma)$ is obviously symmetric and nonnegative. On the constraint manifold, i.e., for $(u,p) \in \V$ with $\tr u = p$, we have in addition 
\[
	\Big\langle \bmat{u \\ p}, \bmat{\calK_\Omega \\ & \calK_\Gamma} \bmat{u \\ p}\Big \rangle
	= \| \nabla u\|^2 + \| p \|_{H^1(\Gamma)}^2
	= \| \nabla u\|^2 + \|\! \tr u \|_{H^1(\Gamma)}^2, 
\]
which equals zero if and only if $u$ and $p$ vanish. 
\end{remark}
%
%
%======================================================================
\subsection{Extensions of the model}
We discuss some possible extensions of the model. In particular, we consider (strong) damping terms which do not change the differential--algebraic structure of the coupling. Following~\cite{HipK21}, we may also consider the system 
\begin{subequations}\label{eq:extPDE}
	\begin{align}
		\ddot u - \delta_\Omega\Delta\dot u - \beta_\Omega\Delta u + (\alpha_\Omega + \gamma_\Omega\cdot\nabla)\dot u + \kappa_\Omega u 
		&= f_\Omega(t,u) &\text{in }\Omega, \\
		\ddot p - \delta_\Gamma\Delta_\Gamma\dot p - \beta_\Gamma\Delta_\Gamma p + (\alpha_\Gamma + \gamma_\Gamma\cdot\nabla_\Gamma)\dot p + \kappa_\Gamma p + \delta_\Omega\partial_n\dot u + \beta_\Omega\partial_n u 
		&= f_\Gamma(t,p) &\text{on }\Gamma, \\
		u - p 
		&= 0 &\text{on }\Gamma,
	\end{align}
\end{subequations}
with coefficients $\alpha_\Omega,\alpha_\Gamma,\beta_\Omega,\beta_\Gamma,\delta_\Omega,\delta_\Gamma,\kappa_\Omega, \geq 0$. 
% maybe \kappa_\Gamma > 0 benificial for stability analysis, since they imply that the corresponding operator is positive instead of just nonnegative.
Moreover, we consider vector fields $\gamma_\Omega\in W^{1,\infty}(\Omega)$ and $\gamma_\Gamma\in W^{1,\infty}(\Gamma)$ satisfying 
\begin{subequations}\label{eq:adv_cond}
	\begin{alignat}{2}
		& \alpha_\Omega - \tfrac{1}{2}\,\nabla\cdot \gamma_\Omega 
		\ge 0 &&\qquad\text{in }\Omega, \label{eq:adv_cond_bulk}\\
		& \alpha_\Gamma + \tfrac{1}{2}\, (\gamma_\Omega\cdot n-\nabla_\Gamma\cdot \gamma_\Gamma) 
		\geq 0 &&\qquad\text{on }\Gamma. \label{eq:adv_cond_surf}
	\end{alignat}
\end{subequations}
The corresponding abstract formulation %is then given by 
%
%\begin{align*}
%	(\ddot u, v)_\Omega 
%	+ (\delta_\Omega\nabla\dot u+\beta_\Omega\nabla u,\nabla v)_\Omega 
%	+ ((\alpha_\Omega + \gamma_\Omega\cdot\nabla)\,\dot u + \kappa_\Omega u,v)_\Omega 
%	- (\delta_\Omega\partial_n\dot u+\beta_\Omega\partial_n u,v)_\Gamma 
%	&= (f_\Omega(t,u),v)_\Omega, \\
%	(\ddot p,q)_\Gamma 
%	+ (\delta_\Gamma\nabla_\Gamma\dot p+\beta_\Gamma\nabla_\Gamma p,\nabla_\Gamma q)_\Gamma 
%	+ ((\alpha_\Gamma + \gamma_\Gamma\cdot\nabla_\Gamma)\dot p +\kappa_\Gamma p+\delta_\Omega\partial_n\dot u+\beta_\Omega\partial_n u,q)_\Gamma 
%	&= (f_\Gamma(t,p),q)_\Gamma.
%\end{align*}
%
in operator form is then given by
\begin{subequations}
\label{eq:fullModel:PDAE}
\begin{alignat}{3}
	\bmat{\ddot u \\ \ddot p} + \bmat{\calD_\Omega \\ & \calD_\Gamma} \bmat{\dot u \\ \dot p} + \bmat{\calK_\Omega \\ & \calK_\Gamma} & \bmat{u \\ p} + \calB^*\lambda 
	&&= \bmat{f_\Omega(t,u) \\ f_\Gamma(t,p)} \qquad && \text{in }\V^*, \\
	\calB & \bmat{u \\ p} 
	&&= 0 \qquad && \text{in }\Q^*,
\end{alignat}
\end{subequations}
where $\calD_\Omega,\calK_\Omega\colon H^1(\Omega)\to[H^1(\Omega)]^*$ and $\calD_\Gamma,\calK_\Gamma\colon H^1(\Gamma)\to[H^1(\Gamma)]^*$ are defined by
\begin{subequations}
	\begin{align}
		\aset{\calD_\Omega\dot u,v} 
		&= \delta_\Omega\int_\Omega\nabla\dot u\cdot\nabla v\dx + \alpha_\Omega\int_\Omega\dot uv\dx + \int_\Omega(\gamma_\Omega\cdot\nabla\dot u)\, v\dx, \\
		\aset{\calK_\Omega u,v} 
		&= \beta_\Omega \int_\Omega\nabla u\cdot\nabla v\dx + \kappa_\Omega\int_\Omega uv\dx, \\
		\aset{\calD_\Gamma\dot p,q} 
		&= \delta_\Gamma\int_\Gamma\nabla_\Gamma\dot p\cdot\nabla_\Gamma q\dx + \alpha_\Gamma\int_\Gamma\dot pq\dx + \int_\Gamma(\gamma_\Gamma\cdot\nabla_\Gamma\dot p)\,q\dx \\
		\aset{\calK_\Gamma p,q} 
		&= \beta_\Gamma\int_\Gamma\nabla_\Gamma p\cdot\nabla_\Gamma q\dx + \kappa_\Gamma\int_\Gamma pq\dx.
	\end{align}
\end{subequations}
Moreover, the constraint operator equals $\calB(u,p) = p-\tr u$ as before.
\begin{remark}
Under sufficient regularity assumptions, the operator equation is equivalent to the original system and $\lambda = \delta_\Omega\partial_n\dot u + \beta_\Omega\partial_n u$.
\end{remark}
%
%Furthermore, by testing with $(\dot u,\dot p)$ we obtain
%%
%\begin{align*}
%	& \frac{1}{2}\ddt \int_\Omega (\dot u^2 + \mu\dot p^2 + \norm{\nabla u}^2 + \beta\norm{\nabla_\Gamma p}^2 + \kappa p^2) \dx \\
%	& \leq \frac{1}{2}\ddt \int_\Omega (\dot u^2 + \mu\dot p^2 + \norm{\nabla u}^2 + \beta\norm{\nabla_\Gamma p}^2 + \kappa p^2) \dx \\
%	&\qquad + \int_\Omega\pset[\big]{ \delta_\Omega\norm{\nabla\dot u}^2 + (\alpha_\Omega-\tfrac{1}{2}\nabla\cdot \gamma_\Omega)\dot u^2 } \dx \\
%	&\qquad + \int_\Gamma\pset[\bigg]{ \delta_\Gamma\norm{\nabla_\Gamma\dot p}^2 + \pset[\big]{ \alpha_\Gamma + \tfrac{1}{2}(\gamma_\Omega\cdot n - \nabla_\Gamma\cdot \gamma_\Gamma) }\dot p^2 } \dx \\
%	&= \int_\Omega f_\Omega(t,u)\dot u\dx + \int_\Gamma f_\Gamma(t,p)\dot p\dx.
%\end{align*}
%
Although the operators $\calK_\Omega$ and $\calK_\Gamma$ are self-adjoint and (semi) positive as in the model problem, the damping operators do not have this nice property. 
% However, while the operator $\calD=\diag(\calD_\Omega,\calD_\Gamma):H^1(\Omega)\times H^1(\Gamma)\to[H^1(\Omega)\times H^1(\Gamma)]^*$ is such that $\calD+\calD^*\geq 0$
In fact, the operators $\calD_\Omega$ and $\calD_\Gamma$ are neither self-adjoint nor non-negative. 

In the following, we propose an alternative operator formulation with better structural properties. For this, we introduce the operators $\calR_\Omega\colon H^1(\Omega)\to[H^1(\Omega)]^*$, $\calR_\Gamma\colon H^1(\Gamma)\to[H^1(\Gamma)]^*$, and $\calG\colon H^1(\Gamma)\to[H^1(\Omega)]^*$, defined as
\begin{subequations}
	\begin{align}
		\aset{\calR_\Omega\dot u,v} 
		&= \delta_\Omega\int_\Omega\nabla\dot u\cdot\nabla v\dx + \int_\Omega \big(\alpha_\Omega-\tfrac{1}{2}\nabla\cdot \gamma_\Omega\big)\,\dot uv\dx, \\
		\aset{\calR_\Gamma\dot p,q} 
		&= \delta_\Gamma\int_\Gamma\nabla_\Gamma\dot p\cdot\nabla_\Gamma q\dx + \int_\Gamma \big(\alpha_\Gamma+\tfrac{1}{2}(\gamma_\Omega\cdot n-\nabla_\Gamma\cdot \gamma_\Gamma)\big)\,\dot pq \dx, \label{eq:bilinear:R_Ga} \\
		\aset{\calG\dot p,v} 
		&= \frac{1}{2}\int_\Gamma(\gamma_\Omega\cdot n)\,\dot p\, \tr(v)\dx.
	\end{align}
\end{subequations}
Note that $\calR_\Omega=\calR_\Omega^*\geq 0$ and $\calR_\Gamma=\calR_\Gamma^*\geq 0$.
As long as $\calB(\dot u,\dot p) = 0$ holds, which is implied by $\calB(u,p) = 0$, we have
\begin{align*}
	\bigaset{ \tfrac{1}{2}(\calD_\Omega+\calD_\Omega^*)\,\dot u, v } 
	&= \aset{\calR_\Omega\dot u,v} + \aset{\calG(\tr \dot u),v}
	= \aset{\calR_\Omega\dot u,v} + \aset{\calG\dot p,v}, \\
	\bigaset{ \tfrac{1}{2}(\calD_\Gamma+\calD_\Gamma^*)\, \dot p, q } 
	&= \aset{\calR_\Gamma\dot p,q} - \aset{\calG q, \dot u}
	= \aset{\calR_\Gamma\dot p,q} - \aset{\calG^*\dot u,q}.
\end{align*}
Then, by defining $\calJ_\Omega\coloneqq\frac{1}{2}(\calD_\Omega-\calD_\Omega^*)$ and $\calJ_\Gamma\coloneqq\frac{1}{2}(\calD_\Gamma-\calD_\Gamma^*)$, we see that system~\eqref{eq:fullModel:PDAE} is equivalent to 
\begin{subequations}
	\begin{alignat}{3}
		\bmat{\ddot u \\ \ddot p} + \bmat{\calJ_\Omega+\calR_\Omega & \calG \\ -\calG^* & \calJ_\Gamma+\calR_\Gamma}\bmat{\dot u \\ \dot p} + \bmat{\calK_\Omega \\ & \calK_\Gamma} & \bmat{u \\ p} + \calB^*\lambda 
		&&= \bmat{f_\Omega(t,u) \\ f_\Gamma(t,p)} \qquad && \text{in }\V^*, \\
		\calB & \bmat{u \\ p} 
		&&= 0 \qquad && \text{in }\Q^*.
	\end{alignat}
\end{subequations}
This formulation has the disadvantage of coupling $u$ and $p$ in the operator applied to $(\dot u,\dot p)$. This coupling, however, only acts on the boundary. Therefore, we incorporate the extra term in the Lagrange multiplier, leading to the operator system
\begin{subequations}\label{eq:extPDE_abstract}
	\begin{alignat}{3}
		\bmat{\ddot u \\ \ddot p} + \bmat{\calJ_\Omega+\calR_\Omega \\ & \calJ_\Gamma+\calR_\Gamma}\bmat{\dot u \\ \dot p} + \bmat{\calK_\Omega \\ & \calK_\Gamma} & \bmat{u \\ p} + \calB^*\nu 
		&&= \bmat{f_\Omega(t,u) \\ f_\Gamma(t,p)} \qquad && \text{in }\V^*, \\
		\calB & \bmat{u \\ p} 
		&&= 0 \qquad && \text{in }\Q^*
	\end{alignat}
\end{subequations}
where the new Lagrange multiplier $\nu$ satisfies
\begin{equation}
	\nu 
	= \lambda - \tfrac{1}{2}(\gamma_\Omega\cdot n)\, \dot p 
	= \delta_\Omega\partial_n\dot u + \beta_\Omega\partial_n u - \tfrac{1}{2}(\gamma_\Omega\cdot n)\, \tr \dot u.
\end{equation}
\begin{remark}
	If we treat the inhomogeneities $f_\Omega$ and $f_\Gamma$ as input variables, the system \eqref{eq:extPDE_abstract} can also be interpreted as the interconnection of two port-Hamiltonian systems (see e.g.~\cite{VilLZM06,RamS04,Mor24phd}). In fact, the bulk PDE can be interpreted as a boundary-controlled port-Hamiltonian system with input $\mathfrak u_b=(\tr\dot u,f_\Omega)$ and output $\mathfrak y_b=(\nu,\dot u)$, while the surface PDE can be interpreted as a distributed-input pH system with input $\mathfrak u_s=(-\nu,f_\Gamma)$ and output $\mathfrak y_s=(\dot p,\dot p)$.
	The interconnection relation $\mathfrak u_{b,1}=\mathfrak y_{s,1},\ \mathfrak u_{s,1}=-\mathfrak y_{b,1}$ defines a Dirac structure and therefore preserves the structure.
	
	As an immediate consequence of this observation, we have the power balance equation and dissipation inequality
	\begin{multline}\label{eq:PBE_ext}
		\ddt\bigg(
		\frac{1}{2}\norm{\dot u}_{L^2}^2
		+ \frac{1}{2}\norm{\nabla u}_{L^2}^2
		+ \frac{1}{2}\norm{\dot p}_{L^2}^2
		+ \frac{1}{2}\norm{\nabla p}_{L^2}^2 \bigg)
		\\
		= -\aset{\calR_\Omega\dot u,\dot u}_{L^2}
		-\aset{\calR_\Gamma\dot p,\dot p}_{L^2}
		+ \int_\Omega f_\Omega \dot u\dx
		+ \int_\Gamma f_\Gamma \dot p\dx
		\leq \int_\Omega f_\Omega \dot u\dx
		+ \int_\Gamma f_\Gamma \dot p\dx,
	\end{multline}
	which can be also easily verified independently.
	Note that, for the model problem \eqref{eq:kinetic:PDAE}, \eqref{eq:PBE_ext} simply yields
	\begin{equation}\label{eq:PBE}
		\ddt\bigg(
		\frac{1}{2}\norm{\dot u}_{L^2}^2
		+ \frac{1}{2}\norm{\nabla u}_{L^2}^2
		+ \frac{1}{2}\norm{\dot p}_{L^2}^2
		+ \frac{1}{2}\norm{\nabla p}_{L^2}^2 \bigg)
		= \int_\Omega f_\Omega \dot u\dx
		+ \int_\Gamma f_\Gamma \dot p\dx,
	\end{equation}
	due to the absence of damping.
\end{remark}
For the sake of simplicity, we mostly focus on the model problem introduced in Section~\ref{sect:formulation:modelProblem}.
Most results also apply to the extended model.
Whenever the differences between models require special attention, we will comment on it.
%
%======================================================================
\subsection{Spatial discretization} 
Since this paper focuses on splitting methods, we only consider standard bulk--surface finite elements for the spatial discretization. We summarize the resulting matrices and refer to~\cite{Dzi88,EllR13,KovL17} for further details. Moreover, we focus on the model problem introduced in Section~\ref{sect:formulation:modelProblem}.
 
The spatial domain $\Omega$ is approximated by a quasi-uniform family of meshes~$\calT_h^\Omega$, where~$h$ is the maximal mesh width and boundary vertices are assumed to be part of $\Gamma$, cf.~\cite{EllR13}. We further denote by $\Omega_h$ the polyhedral domain obtained as the union of all elements of $\calT_h^\Omega$ and by $\calT_h^\Omega|_\Gamma$ the restriction of $\calT_h^\Omega$ to the boundary. 
We introduce a second mesh $\calT_h^\Gamma$ on the boundary, which 
%with applications in mind 
is assumed to be a refinement of $\calT_h^\Omega|_\Gamma$. In particular, its maximal mesh width is bounded by $h$. 
We denote by $\Gamma_h$ the polyhedral surface obtained as the union of all elements of $\calT_h^\Gamma$, noting that $\Gamma_h\neq\partial\Omega_h$ in general.

Moreover, we consider standard $P_k$-finite elements on $\calT_h^\Omega$ for $u$, and their restriction to $\calT_\Omega|_h^\Gamma$ for $\lambda$. The choice of basis functions for $p$ is free, and can therefore focus on approximation properties, cf.~\cite{AltV21,AltZ24}. This leads to a (possibly nonconforming) approximation of $H^1(\Omega)$. 
% suitable lift operator is introduced in~\cite{Dzi88}
%Note that the restriction of the finite element space to the boundary yields a finite element for the approximation of~$p$.  

% FEM matrices
Following this approach yields mass matrices $M_\Omega\in \R^{N_\Omega,N_\Omega}$ and $M_\Gamma\in \R^{N_\Gamma,N_\Gamma}$ as discrete versions of the respective $L^2$-inner products. Here, $N_\Omega$ and $N_\Gamma$ denote the number of degrees of freedom for the $u$ and $p$ variable, respectively. As usual, mass matrices are symmetric and positive definite. 
The corresponding stiffness matrices are denoted by $A_\Omega\in \R^{N_\Omega,N_\Omega}$ and $A_\Gamma\in \R^{N_\Gamma,N_\Gamma}$. As discrete versions of the operators~$\calK_\Omega$ and $\calK_\Gamma$, respectively, these matrices are symmetric and positive semidefinite.  
The constraint operator~$\calB$ results in a rectangular matrix $B=[B_\Omega,B_\Gamma]\in \R^{N_\lambda, N_\Omega+N_\Gamma}$ in which $B_\Omega$ has full row-rank. Assuming a specific ordering of the basis functions, one can bring $B_\Omega$ to the form $B_\Omega = [0,-M_\lambda]$ with $M_\lambda\in\R^{N_\lambda,N_\lambda}$ symmetric and positive definite.

% semi-discrete system
Finally, we state the resulting semi-discrete system. Based on the model problem~\eqref{eq:kineticBC} and its abstract formulation as coupled system in~\eqref{eq:kinetic:PDAE}, we seek~$u\colon[0,T] \to \R^{N_\Omega}$, $p\colon[0,T] \to \R^{N_\Gamma}$, and a Lagrange multiplier~$\lambda\colon[0,T] \to \R^{N_\lambda}$ such that 
\begin{subequations}
	\label{eq:semidisc:kinetic}
	\begin{align}
		\begin{bmatrix} M_\Omega &  \\  & M_\Gamma \end{bmatrix}
		\begin{bmatrix} \ddot u \\ \ddot p  \end{bmatrix}
		+ \begin{bmatrix} A_\Omega &  \\  & A_\Gamma \end{bmatrix}
		\begin{bmatrix} u \\ p  \end{bmatrix}
		+ B^\top \lambda 
		&= \begin{bmatrix} f_\Omega(t,u) \\ f_\Gamma(t,p) \end{bmatrix}, \label{eq:semidisc:kinetic:a} \\%[1mm]
		\begin{bmatrix} 0 & M_\lambda \end{bmatrix} u - B_\Gamma p 
		&= 0, \label{eq:semidisc:kinetic:b}
	\end{align}
\end{subequations}
or, more concisely, 
\begin{subequations}\label{eq:semidiscconcise}
	\begin{align}
		M\ddot z + Az + B^\top\lambda &= f(t,z), \\
		Bz &= 0,
	\end{align}
\end{subequations}
where $z=(u,p)$, $M=\diag(M_\Omega,M_\Gamma)$, $A=\diag(A_\Omega,A_\Gamma)$, and $f=(f_1,f_2)$.
Furthermore, the structure of $B_\Omega$ induces a natural decomposition $u=(u_1,u_2)$. Therefore, we can write equivalently
\begin{subequations}
	\label{eq:semidiscsplit:kinetic}
	\begin{align}
		\begin{bmatrix} M_{11} & M_{12} \\ M_{21} & M_{22} \end{bmatrix}
		\begin{bmatrix} \ddot u_1 \\ \ddot u_2  \end{bmatrix}
		+ \begin{bmatrix} A_{11} & A_{12} \\ A_{21} & A_{22} \end{bmatrix}
		\begin{bmatrix} u_1 \\ u_2  \end{bmatrix}
		+ \bmat{0 \\ -M_\lambda} \lambda 
		&= \begin{bmatrix} f_1(t,u_1,u_2) \\ f_2(t,u_1,u_2) \end{bmatrix}, \label{eq:semidiscsplit:kinetic:bulk} \\%
		M_\Gamma\ddot p + A_\Gamma p + B_\Gamma^\top\lambda &= f_\Gamma(t,p), \label{eq:semidiscsplit:kinetic:surf} \\
		M_\lambda u_2 &= B_\Gamma p. \label{eq:semidiscsplit:kinetic:alg}
	\end{align}
\end{subequations}
In the upcoming section, we introduce a splitting scheme based on this formulation. 
\begin{remark}
Due to the full-rank property of~$B$, system~\eqref{eq:semidisc:kinetic} -- and hence also system~\eqref{eq:semidiscsplit:kinetic} -- equals a differential--algebraic equation (DAE) of index 3, cf.~\cite[Ch.~VII.1]{HaiW96}. 
\end{remark}
\begin{remark}\label{rem:semidiscExtModel}
	The same paradigm can be applied to construct a semi-discretization for the extended model \eqref{eq:extPDE}. In particular, by applying the same bulk--surface finite element method to the abstract formulation \eqref{eq:extPDE_abstract}, we obtain a DAE of the form
	\begin{subequations}\label{eq:semidisc_ext}
		\begin{align}
			\begin{bmatrix} M_\Omega &  \\  & M_\Gamma \end{bmatrix}
			\begin{bmatrix} \ddot u \\ \ddot p  \end{bmatrix}
			+ \bmat{D_\Omega & \\ & D_\Gamma}
			\bmat{\dot u \\ \dot p}
			+ \begin{bmatrix} A_\Omega &  \\  & A_\Gamma \end{bmatrix}
			\begin{bmatrix} u \\ p  \end{bmatrix}
			+ B^\top \nu 
			&= \begin{bmatrix} f_\Omega(t,u) \\ f_\Gamma(t,p) \end{bmatrix}, \\%[1mm]
			\begin{bmatrix} 0 & M_\lambda \end{bmatrix} u - B_\Gamma p 
			&= 0,
		\end{align}
	\end{subequations}
	where $D_\Omega\in\R^{N_\Omega,N_\Omega}$ with $D_\Omega+D_\Omega^\top\geq 0$ and $D_\Gamma\in\R^{N_\Gamma,N_\Gamma}$ with $D_\Gamma=D_\Gamma^\top$.
	Furthermore, we assume $D_\Gamma+D_\Gamma^\top\geq 0$, although its fulfillment might depend on the choice of finite elements for $\Gamma$; see the discussion in \Cref{app:semidiscExtModel}.
	Note that the splitting of the system into bulk and surface variables is retained.
\end{remark}
%
%======================================================================
%=========  Delay, Preliminaries
%======================================================================
\section{Bulk--surface Splitting}\label{sect:schemesKinetic}
This section is devoted to the introduction of bulk--surface
splitting schemes of second order. For this, we first establish some useful difference formulae. 
%
% 
%======================================================================
\subsection{Finite difference formulae}
In the following, we make use of the following second-order \emph{backward difference formulae} ($\bdf$) to approximate derivatives, 
\begin{align*}
	\dot u 
	&\approx \bdf_1^\tau u 
	\coloneqq \tfrac{1}{\tau} \big(E_\tau+\tfrac{1}{2}E_\tau^2 \big) u
	= \tfrac{1}{\tau}\big(\tfrac{3}{2}u-2u_\tau+\tfrac{1}{2}u_{2\tau}\big), \\
	\ddot u 
	&\approx \bdf_2^\tau u 
	\coloneqq \tfrac{1}{\tau^2} \big(E_\tau^2+E_\tau^3\big) u
	= \tfrac{1}{\tau} \big(2u-5u_\tau+4u_{2\tau}-u_{3\tau}\big),
\end{align*}
where $\tau>0$ is a fixed time step, $u_s$ for $s\in\R$ denotes the shifted function $u_s(t)=u(t-s)$, and $E_\tau$ denotes the finite difference operator such that $E_\tau u=u-u_\tau$.
Furthermore, we introduce the following \emph{delay difference formulae} ($\ddf$) to approximate function values and their derivatives, 
\begin{align*}
	u 
	&\approx
	\ddf_0^\tau u 
	\coloneqq \big(I-E_\tau^4\big) u
	= 4u_\tau-6u_{2\tau}+4u_{3\tau}-u_{4\tau}, \\
	\dot u 
	&\approx \ddf_1^\tau u 
	\coloneqq \tfrac{1}{\tau} \big(E_\tau+\tfrac{1}{2}E_\tau^2-\tfrac{3}{2}E_\tau^4\big) u
%	= \frac{1}{\tau}(E+\tfrac{3}{2}E^2+\tfrac{3}{2}E^3)u_\tau
	= \tfrac{1}{\tau} \big(4u_\tau-\tfrac{17}{2}u_{2\tau}+6u_{3\tau}-\tfrac{3}{2}u_{4\tau} \big), \\
	\ddot u
	&\approx \ddf_2^\tau u
	\coloneqq \tfrac{1}{\tau^2} \big(E_\tau^2+E_\tau^3-2E_\tau^4\big) u
%	= \tfrac{1}{\tau^2}(E^2+2E^3)u_\tau
	= \tfrac{1}{\tau^2}\big(3u_\tau-8u_{2\tau}+7u_{3\tau}-2u_{4\tau}\big)
	.
\end{align*}
Note that the $\ddf$ approximate $u,\dot u,\ddot u$ using only past values and satisfy
\begin{equation}\label{eq:bdf-ddf}
	\bdf_k^\tau-\ddf_k^\tau
	= \frac{k+2}{2\tau^k}\, E_\tau^4
\end{equation}
for $k=0,1,2$ by construction (where trivially $\bdf_0^\tau u\coloneqq u$). In particular, the $\ddf$ inherit from the $\bdf$ the property of being second-order approximations (with $\ddf_0$ actually being fourth order); see also \Cref{lem:BDF1_L2_error,lem:BDF2_L2_error} for details.

From now on, we omit the index~$\tau$ from the notation of the difference formulae and of the difference operator $E$, when it is clear from the context.
The same notation will also be applied to discrete sequences in the following sense: given vectors $u^n$ for $n\geq 0$, we denote $Eu^n=u^n-u^{n-1}$ and define the corresponding $\bdf$ and $\ddf$ formulae analogously. 
%
%
%======================================================================
\subsection{A splitting scheme of second order} % implicit--implicit

To define a time-stepping scheme for \eqref{eq:semidiscsplit:kinetic}, let us introduce for a uniform time step $\tau=\tfrac{T}{N}>0$, a time grid $t^n=n\tau$, and discrete states $u^n=(u_1^n,u_2^n)\in\R^{N_\Omega}$, $p^n\in\R^{N_\Gamma}$, and $\lambda^n\in\R^{N_\lambda}$ for $n=0,1,\ldots,N$. 
Furthermore, we introduce possibly different approximations for the discrete states and its derivatives in the first and second block equations of \eqref{eq:semidiscsplit:kinetic:bulk} as well as in the argument of the inhomogeneities. We denote these different approximations by additional indices, namely 
\begin{alignat*}{3}
	u_a^n &= (u_{a1}^n,u_{a2}^n), \qquad&
	\ddot u_a^n &= (\ddot u_{a1}^n,\ddot u_{a2}^n), \qquad&
	u_{fa}^n &= (u_{fa1}^n,u_{fa2}^n), \\
	u_b^n &= (u_{b1}^n,u_{b2}^n), \qquad&
	\ddot u_b^n &= (\ddot u_{b1}^n,\ddot u_{b2}^n), \qquad&
	u_{fb}^n &= (u_{fb1}^n,u_{fb2}^n),
\end{alignat*}
and $p_c^n$, $\ddot p_c^n$, $p_{fc}^n$. 
In particular, we derive from \eqref{eq:semidiscsplit:kinetic:bulk} and \eqref{eq:semidiscsplit:kinetic:surf} the discrete equations
\begin{subequations}\label{eq:discrete_gen}
	\begin{align}
		M_{11}\ddot u_{a1}^n + M_{12}\ddot u_{a2}^n + A_{11}u_{a1}^n + A_{12}u_{a2}^n &= f_1(t^n,u_{fa}^n), \label{eq:discrete_gen:bulk} \\
		M_{21}\ddot u_{b1}^n + M_{22}\ddot u_{b2}^n + A_{21}u_{b1}^n + A_{22}u_{b2}^n - M_\lambda\lambda^n &= f_2(t^n,u_{fb}^n), \label{eq:discrete_gen:lagrange} \\
%		\bmat{M_{11} & M_{12} \\ M_{21} & M_{22}}\bmat{\ddot u_1^n \\ \ddot u_2^n}
%		+ \bmat{A_{11} & A_{12} \\ A_{21}& A_{22}}\bmat{u_1^n \\ u_2^n}
%		+ \bmat{0 \\ -M_\lambda}\lambda^n &= \bmat{f_1^n \\ f_2^n}, \label{eq:discrete_gen:bulk} \\
		%
		M_\Gamma\ddot p_c^n + A_\Gamma p_c^n + B_\Gamma^\top\lambda^n &= f_\Gamma(t^n,p_c^n). \label{eq:discrete_gen:surf}
	\end{align}
\end{subequations}
The algebraic constraint \eqref{eq:semidiscsplit:kinetic:alg} can be interpreted as $u_2=Pp$, where $P=M_\lambda^{-1}B_\Gamma$. From this, also $\ddot u_2=P\ddot p$ follows. To enforce the decoupling of the bulk and surface variables, we approximate $u_2$ and $\ddot u_2$ using only past values of the system, more precisely  
\begin{equation}
	u_2^n = u_{a2}^n = u_{b2}^n = u_{fa2}^n = u_{fb2}^n 
	\coloneqq P \ddf_0 p^n, \qquad
	\ddot u_{a2}^n = \ddot u_{b2}^n 
	\coloneqq P\ddf_2 p^n.
\end{equation}
To allow the system to be solved uniquely for $u_1^n$, $p^n$ and have it implicit in these two variables, we set
\begin{equation}
	u_{a1}^n = u_1^n, \qquad p_c^n = p^n, \qquad \ddot u_{a1}^n = \bdf_2u_1^n, \qquad \ddot p_c^n = \bdf_2p^n.
\end{equation}
To further decouple the system, we also define $u_{b1}^n$, $\ddot u_{b1}^n$, and $u_{fb1}^n$ using past values, i.e.,
\begin{equation}
	u_{b1}^n = u_{fb1}^n 
	\coloneqq \ddf_0 u_1^n, \qquad
	\ddot u_{b1}^n 
	\coloneqq \ddf_2 u_1^n.
\end{equation}
Finally, to avoid having to solve a nonlinear system at every step, we approximate the remaining arguments of $f$ using again only past values, i.e.,
\begin{equation}
	u_{fa1}^n 
	\coloneqq \ddf_0 u_1^n, \qquad 
	p_{fc}^n 
	\coloneqq \ddf_0 p^n.
\end{equation}

Since the resulting discrete system involves values of the discrete states up to four steps in the past, we write it more precisely as 
\begin{subequations}\label{eq:multistep}
	\begin{align}\setlength\arraycolsep{2pt}
		(M_{11}\bdf_2+A_{11})u_1^{n+4} + (M_{12}\ddf_2+A_{12}\ddf_0)Pp^{n+4} &= f_1^{n+4}, \label{eq:multistep:bulk} \\
		u_2^{n+4} - \ddf_0 Pp^{n+4} &= 0, \label{eq:multistep:alg} \\
		(M_{21}\ddf_2+A_{21}\ddf_0)u_1^{n+4} + (M_{22}\ddf_2+A_{22}\ddf_0)Pp^{n+4} - M_\lambda\lambda^{n+4} &= f_2^{n+4}, \label{eq:multistep:lagrange} \\
		(M_\Gamma\bdf_2 + A_\Gamma)p^{n+4} + B_\Gamma^\top\lambda^{n+4} &= f_\Gamma^{n+4}, \label{eq:multistep:surf}
	\end{align}
\end{subequations}
for all $n\geq 0$, where we introduce the concise notation
\begin{subequations}\label{eq:f_eval}
\begin{align}
	\bsmat{f_1^n \\ f_2^n} = f_\Omega^n &\coloneqq f_\Omega(t^n, \ddf_0 u_1^n, P\ddf_0 p^n), \\
	f_\Gamma^n &\coloneqq f_\Gamma(t^n, \ddf_0p^n).
\end{align}
\end{subequations}

In particular, given previous values $u_1^n,\ldots,u_1^{n+3}$ and $p^n,\ldots,p^{n+3}$, equations \eqref{eq:multistep:bulk} and \eqref{eq:multistep:alg}, and the pair of equations \eqref{eq:multistep:lagrange}, \eqref{eq:multistep:surf} can be solved independently for $(u_1^{n+4}, u_2^{n+4})$ and $(\lambda^{n+4},p^{n+4})$, respectively.
In fact, the Lagrange multiplier $\lambda^{n+4}$ can be obtained in \eqref{eq:multistep:lagrange} and then inserted in \eqref{eq:multistep:surf}, yielding
\begin{multline}\label{eq:multistep:surf_pure}
	M_\Gamma\bdf_2p^{n+4} + A_\Gamma p^{n+4} + P^\top\big( (M_{21}\ddf_2+A_{21}\ddf_0)u_1^{n+4} \\ + (M_{22}\ddf_2+A_{22}\ddf_0)Pp^{n+4} - f_2^{n+4} \big) = f_\Gamma^{n+4},
\end{multline}
which can be solved for $p^{n+4}$ using only past values of $u_1$ and $p$.
In particular, the equations \eqref{eq:multistep:bulk}, \eqref{eq:multistep:alg}, and \eqref{eq:multistep:surf_pure} can be solved in parallel for $u_1^{n+4}$, $u_2^{n+4}$, and $p^{n+4}$, respectively, obtaining a fully decoupled method.

For the analysis of the time-stepping method, it is useful to introduce the auxiliary state $\hat u_2^n \coloneqq Pp^n$ for all $n\geq 0$. This state satisfies, in particular, 
\[
	M_\lambda\hat u_2^n=B_\Gamma p^n,\qquad 
	u_2^{n+4}=\ddf_0\hat u_2^{n+4},\qquad
	\ddot u_2^{n+4}=\ddf_2\hat u_2^{n+4}
\]
for all $n\geq 0$.
By defining the auxiliary bulk and bulk--surface states 
\begin{align}
	\label{eq:defAuxStates}
	\hat u^n=(u_1^n,\hat u_2^n), \qquad
	\hat z^n=(\hat u^n,p^n)
\end{align}
and the combined inhomogeneity $f^{n+4}=(f_\Omega^{n+4},f_\Gamma^{n+4})$, an application of formula~\eqref{eq:bdf-ddf} shows that we can rewrite the discrete system \eqref{eq:multistep} equivalently as
\begin{subequations}\label{eq:multistep_aux}
	\begin{align}
		M\bdf_2\hat z^{n+4} + A\hat z^{n+4} + B^\top\lambda^{n+4} &= f^{n+4} + E^4\pset*{ \tfrac{2}{\tau^2}\wt M + \wt A }\hat z^{n+4}, \label{eq:multistep_aux:1} \\
		B\hat z^{n+4} &= 0,
	\end{align}
\end{subequations}
where
\begin{subequations}
	\begin{alignat}{2}
		\wt M &\coloneqq \bmat{\wt M_\Omega & 0 \\ 0 & 0}, &\qquad
		\wt M_\Omega &\coloneqq \bmat{0 & M_{12} \\ M_{21} & M_{22}}, \\
		\wt A &\coloneqq \bmat{\wt A_\Omega & 0 \\ 0 & 0}, &\qquad
		\wt A_\Omega &\coloneqq \bmat{0 & A_{12} \\ A_{21} & A_{22}}
	\end{alignat}
\end{subequations}
are symmetric (and possibly indefinite) matrices.
In particular, \eqref{eq:multistep_aux} can be interpreted as a perturbation of the standard $\bdf$-2 method applied to~\eqref{eq:semidiscconcise}.  

As we will prove in the next section, this splitting method has good stability and convergence properties. In particular, up to the discretization parameters $\tau,h$ satisfying a weak CFL condition of the form~$\tau\lesssim\sqrt{h}$, the scheme has convergence order two.
%
%
%======================================================================
\subsection{Alternative splitting schemes}\label{sec:alternativeSplitting}
We shortly discuss alternative splitting approaches. However, since corresponding convergence results can be proven with similar techniques but the individual proofs are quite long, we will focus on the analysis of \eqref{eq:multistep} in the remainder of the paper.

%\begin{remark}
Generally speaking, one can construct alternative splitting schemes by making different choices for the discrete approximations in \eqref{eq:discrete_gen}.
For example, choosing $u_{b1}^n=u_1^n$ instead of using past values would lead to a similar method with comparable stability and convergence properties. This scheme, however is not parallelizable. 
Another possibility would be to replace $\ddf_0$ with a delay formula which uses less terms, e.g.,~$I-E^3$ or $I-E^2$. However, in both cases the resulting stability results are weaker, while second-order convergence is retained.
%	
%	Furthermore, by eliminating the Lagrange multiplier in \eqref{eq:discrete_gen} by writing the subsystem
%	%
%	\begin{subequations}
%		\begin{align}
%			& M_{11}\ddot u_{a1}^n + M_{12}\ddot u_{a2}^n + A_{11}u_{a1}^n + A_{12}u_{a2}^n = f_1(t^n,u_{fa}^n), \\
%			& \begin{multlined}[t]PM_{21}\ddot u_{b1}^n + PM_{22}\ddot u_{b2}^n + M_\Gamma\ddot p_c^n \\ + PA_{21}u_{b1}^n + PA_{22}u_{b2}^n + A_{22}p_c^n = Pf_2(t^n,u_{fb}^n) + f_\Gamma(t^n,p_c^n),\end{multlined}
%%			PM_{21}\ddot u_{b1}^n + PM_{22}\ddot u_{b2}^n + M_\Gamma\ddot p_c^n + PA_{21}u_{b1}^n + PA_{22}u_{b2}^n + A_{22}p_c^n &= Pf_2(t^n,u_{fb}^n) + f_\Gamma(t^n,p_c^n),
%		\end{align}
%	\end{subequations}
%	%
%	we observe that choosing $\ddot u_{b2}^n=P\bdf_2 p^n$ and $u_{b2}^n=Pp^n$, while still using delay formulae for $\ddot u_{a2}^n,u_{a2}^n,\ddot u_{b1}^n,u_{b2}^n$, also allows us to decouple $u_1^n$ and $p^n$. This method is still parallelizable and has analogous stability and convergence properties as the presented one.
%\end{remark}

Finally, we remark that the presented method can also be applied to the extended model \eqref{eq:semidisc_ext}. 
%\begin{remark}
%The same approach used to discretize \eqref{eq:semidisc:kinetic} can also be applied to the extended model \eqref{eq:semidisc_ext}. 
In particular, the occurrences of $\dot u_1,\dot u_2,\dot p$ in the equations will be approximated with $\bdf_1$ and $\ddf_1$, with possibly different choices in different parts of the equation. In fact, if we denote analogously as before by $\dot u_a^n = (\dot u_{a1}^n,\dot u_{a2}^n)$, $\dot u_b^n = (\dot u_{b1}^n,\dot u_{b2}^n)$, and $\dot p_c^n$ the terms appearing in the three block equations, it is natural to choose
\begin{equation}
	\dot u_{a1}^n = \bdf_1u_1^n, \quad
	\dot u_{a2}^n = \dot u_{b2}^n = \ddf_1u_2^n, \quad
	\dot u_{b1}^n = \ddf_1u_1^n, \quad
	\dot p_c^n = \bdf_1p^n.
\end{equation}
When $\delta_\Omega,\gamma_\Omega=0$, the same stability and convergence properties hold, with minimal alterations to the following proofs.
For $\delta_\Omega>0$, the stricter CFL condition $\tau\lesssim h$ is required.
When $\gamma_\Omega\neq 0$, a longer formula for $\ddf_1$ (e.g.~$E+\tfrac{3}{2}E^2-\tfrac{3}{2}E^5$) is necessary to adapt the stability and convergence proofs, although numerical simulations seem to show that the provided formula for $\ddf_1$ still leads to comparable results.
\section{Stability and Convergence Analysis}

In this section, we prove stability and convergence of the splitting scheme~\eqref{eq:multistep}. Due to the length of the complete proof, we split it in multiple intermediate steps. 
As before, the time interval $[0,T]$ with $T>0$ is decomposed into $N$ subintervals with equal length $\tau=\tfrac{T}{N}$, so that $z^n\approx z(t^n)$ with $t^n=n\tau$, $n=0,1,\ldots,N$. 

Suppose we want to compute an approximate solution for \eqref{eq:semidisc:kinetic} using the $4$-step splitting scheme~\eqref{eq:multistep}. We assume consistent initial data $u^n$, $p^n$ for $n=0,\ldots,3$, i.e., values which satisfy the algebraic constraint $M_\lambda u_2^n=B_\Gamma^\top p^n$ (or equivalently $\hat z^n=z^n$).  
Furthermore, we assume that the nonlinearity $f$ is split into
\begin{equation}
	f = Mf_M + \wt Mf_{\wt M} + \wt Af_{\wt A},
	\qquad
	f_{\wt M} = \bmat{f_{\wt M,\Omega} \\ f_{\wt M,\Gamma}}, \qquad
	f_{\wt A} = \bmat{f_{\wt A,\Omega} \\ f_{\wt A,\Gamma}},
%	f_A = \bmat{f_{A,\Omega} \\ f_{A,\Gamma}},
\end{equation}
where the bulk components themselves are decomposed into $f_{\wt M,\Omega}=(f_{\wt M,1},f_{\wt M,2})$ and $f_{\wt A,\Omega}=(f_{\wt A,1},f_{\wt A,2})$ with $f_{\wt M,2}=Pf_{\wt M,\Gamma}$ and $f_{\wt A,2}=Pf_{\wt A,\Gamma}$. 
Note that the split is artificial in the sense that one may choose~$f_M=M^{-1}f$ and $f_{\wt M}, f_{\wt A}=0$. It will, however, be helpful for the upcoming convergence analysis of the method.
% While the term $f_M$ naturally derives from the semi-discretization of the inhomogeneity, the terms $f_{\wt M}$ and $f_{\wt A}$ are artificial 
Finally, we denote sequences of function evaluations by $f_M^n$ and so on as in~\eqref{eq:f_eval}.

In the following proofs, we denote by $C$ a generic positive real constant, independent of the discretization parameters $h$, $\tau$ and the iteration counter (usually denoted by $n$ or $m$).
Hence, $C$ may denote different constants with the understanding that we replace them with a larger one whenever necessary.
Alternatively, we sometimes use the notion $\mathrm{LHS}\lesssim\mathrm{RHS}$ instead of $\mathrm{LHS}\leq C\,\mathrm{RHS}$.

Furthermore, we will apply the following inequalities and equations without referring to them explicitly:
\begin{enumerate}[itemsep=5pt,topsep=5pt]
%	\item $\norm{E^k\hat u_2^{n}}_{M_{22}}^2 \lesssim h\norm{E^kp^{n}}_{M_\Gamma}^2$ and
%	$\norm{E^k\hat u_2^{n}}_{A_{22}}^2 \lesssim \tfrac{1}{h}\norm{E^kp^{n}}_{M_\Gamma}^2$
%	for all $0\leq k\leq n$, due to \eqref{eq:matrixAssumptions}.
	%
	\item $\aset{x,y} \leq \lambda_1\norm{x}^2 + \lambda_2\norm{y}^2$ for all vectors $x,y$ and $\lambda_1,\lambda_2>0$ such $\lambda_1\lambda_2\geq\tfrac{1}{4}$, due to Young's inequality.
	In particular, $\aset{x,y}\lesssim f(\tau,h)\norm{x}^2+f(\tau,h)^{-1}\norm{y}^2$ for any positive function $f$ of $\tau$ and $h$.
	\item $\sqrt{\norm{x^1}^2+\ldots+\norm{x^k}^2}\leq\norm{x^1}+\ldots+\norm{x^k} \leq \sqrt{k}\, (\norm{x^1}^2+\ldots+\norm{x^k}^2)^{\frac{1}{2}}$
	for all vectors $x^1,\ldots,x^k$, where the second bound follows from the Cauchy--Schwarz inequality. In particular,
	$\norm{x^1+\ldots+x^k}\lesssim (\norm{x^1}^2+\ldots+\norm{x^k}^2)^{\frac{1}{2}}$ whenever $k$ is an integer independent of the discretization parameters and the iteration counter.
	%
%	\item $\norm{\tfrac{1}{\tau}E^k\hat z^n}_M^2
%	=\norm{\tfrac{1}{\tau}\tsum_{\ell=0}^{k-1}\binom{k-1}{\ell}(-1)^\ell E\hat z^{n-\ell}}_M^2
%	\lesssim H_0 + \norm{\tfrac{1}{\tau}E\hat z^4}_M^2 + \ldots + \norm{\tfrac{1}{\tau}E\hat z^n}_M^2$ for all $n\geq 1$ and $1\leq k\leq n$. Analogous bounds hold for $\norm{\tfrac{1}{\tau}E^k\hat u^n}_{M_\Omega}^2$ and $\norm{\tfrac{1}{\tau}E^k p^n}_{M_\Gamma}^2$, $\norm{E^{k-1}\hat z^{n-1}}_A^2$, $\norm{E^{k-1}\hat u^{n-1}}_{A_\Omega}^2$, and $\norm{E^{k-1}p^{n-1}}_{A_\Gamma}^2$.
	%
	\item The discrete integration by parts formulae
	\begin{subequations}\label{eq:discreteIntegrationByParts}
	\begin{align}
		\aset{u^{n+1}, Ev^{n+1}}
		&= E\aset{u^{n+1},v^{n+1}} - \aset{Eu^{n+1}, v^n} \label{eq:discreteIntegrationByParts:a}\\
		&= E\aset{u^{n+1},v^{n+1}} - \aset{Eu^{n+1}, v^{n+1}} + \aset{Eu^{n+1},Ev^{n+1}}, \label{eq:discreteIntegrationByParts:b}\\
		\aset{u^{n+1} , Eu^{n+1}} 
		&= \tfrac{1}{2}E(\norm{u^{n+1}}^2) + \tfrac{1}{2}\norm{Eu^{n+1}}^2 \label{eq:discreteIntegrationByParts:c}
	\end{align}
	\end{subequations} 
	for all sequences of vectors $u^n,v^n$.
	\item We assume sufficiently small mesh and step sizes. In particular, we have $h,\tau \lesssim 1$. 
\end{enumerate}
We will need the following technical lemma, which holds true for the choice of $P_k$-finite elements for $u$ and their restriction to the boundary for $\lambda$, while it is independent of the basis choice for $p$.
\begin{lemma}[see~{\cite[Lem.~4.3]{AltZ24}}]
	\label{lem:Christoph}
	There exist constants $c_M,c_A>0$, only depending on the system parameters and the uniformity parameter of the underlying mesh $\calT_h^\Omega$, but not on $h$, such that
	\begin{equation*}
		\norm{u_2}_{M_{22}}^2 
		\leq c_Mh\, \norm{p}_{M_\Gamma}^2, \qquad
		\norm{u_2}_{A_{22}}^2 
		\leq c_Ah^{-1}\norm{p}_{M_\Gamma}^2,
	\end{equation*}
	for all $u=(u_1,u_2)\in\R^{N_\Omega}$ and $p\in\R^{N_\Gamma}$ satisfying $M_\lambda u_2=B_\Gamma p$.
\end{lemma}
We further observe that the discrete integration by parts formulae~\eqref{eq:discreteIntegrationByParts} hold more general for symmetric bilinear forms (not necessarily definite or semidefinite), up to replacing $\norm{\,\cdot\,}^2$ with $\aset{\,\cdot\,,\cdot\,}$. With that in mind, we deduce the following bounds.
\begin{lemma}\label{lem:fakescalar_bound}
	Let $z=(u,p),\ y=(v,q)\in\R^{N_\Omega+N_\Gamma}$ with $u=(u_1,u_2),\ v=(v_1,v_2)$ and $M_\lambda u_2=B_\Gamma p$, $M_\lambda v_2=B_\Gamma q$. Further assume $\tau \lesssim \sqrt{h}$. Then we have the estimates 
	\begin{subequations}
		\begin{align}
			\abs[\big]{\aset{y,z}_{\wt M}}
			&= \abs[\big]{\aset{v,u}_{\wt M_\Omega}}
			\lesssim \sqrt{h}\, \norm{y}_M\norm{z}_M, \\
			\abs[\big]{\aset{y,z}_{\wt A}}
			&= \abs[\big]{\aset{v,u}_{\wt A_\Omega}}
			\lesssim \tfrac{\tau}{\sqrt h}\, \big(\norm{v}_{A_\Omega}+\norm{\tfrac{1}{\tau}q}_{M_\Gamma}\big)
			\big(\norm{u}_{A_\Omega} + \norm{\tfrac{1}{\tau}p}_{M_\Gamma}\big).
		\end{align}
	\end{subequations}
\end{lemma}
\begin{proof}
	For the first estimate, we apply \Cref{lem:Christoph} and obtain 
	\begin{align*}
		\abs[\big]{\aset{y,z}_{\wt M}}
		&= \abs[\big]{\aset{v,u}_{\wt M_\Omega}}
		= \abs[\big]{ \aset{\bsmat{0 \\ v_2},u}_{M_\Omega} + \aset{v,\bsmat{0 \\ u_2}}_{M_\Omega} - \aset{v_2,u_2}_{M_{22}} } \\
		&\leq \norm{v_2}_{M_{22}}\norm{u}_{M_\Omega} + \norm{v}_{M_\Omega}\norm{u_2}_{M_{22}} + \norm{v_2}_{M_{22}}\norm{u_2}_{M_{22}} \\
		&\lesssim \sqrt{h}\, \norm{q}_{M_\Gamma}\norm{u}_{M_\Omega} + \sqrt{h}\, \norm{v}_{M_\Omega}\norm{p}_{M_\Gamma} + h\, \norm{q}_{M_\Gamma}\norm{p}_{M_\Gamma}
		\lesssim \sqrt{h}\, \norm{y}_M\norm{z}_M.
	\end{align*}
	For the second estimate, we apply \Cref{lem:Christoph} once more, leading to 
	\begin{align*}
		\abs[\big]{\aset{y,z}_{\wt A}}
		&= \abs[\big]{\aset{v,u}_{\wt A_\Omega}}
		= \abs[\big]{ \aset{\bsmat{0 \\ v_2},u}_{A_\Omega} + \aset{v,\bsmat{0 \\ u_2}}_{A_\Omega} - \aset{v_2,u_2}_{A_{22}} } \\
		&\leq \norm{v_2}_{A_{22}}\norm{u}_{A_\Omega} + \norm{v}_{A_\Omega}\norm{u_2}_{A_{22}} + \norm{v_2}_{A_{22}}\norm{u_2}_{A_{22}} \\
		&\lesssim \tfrac{\tau}{\sqrt h}\norm{\tfrac{1}{\tau}q}_{M_\Gamma}\norm{u}_{A_\Omega} + \tfrac{\tau}{\sqrt h}\norm{v}_{A_\Omega}\norm{\tfrac{1}{\tau}p}_{M_\Gamma} + \tfrac{\tau^2}{h}\norm{\tfrac{1}{\tau}q}_{M_\Gamma}\norm{\tfrac{1}{\tau}p}_{M_\Gamma} \\
		&\lesssim \tfrac{\tau}{\sqrt h} \big(\norm{v}_{A_\Omega}+\norm{\tfrac{1}{\tau}q}_{M_\Gamma}\big)
		\big(\norm{u}_{A_\Omega} + \norm{\tfrac{1}{\tau}p}_{M_\Gamma}\big),
	\end{align*}
	where the second-order term disappears due to $\tau \lesssim \sqrt{h}$.
\end{proof}
\begin{lemma}\label{lem:fakenorm_bound2}
	Let $z^n=(u^n,p^n)\in\R^{N_\Omega+N_\Gamma}$ for $n\geq 0$ with $u^n=(u_1^n,u_2^n)$ and $M_\lambda u_2^n=B_\Gamma p^n$. Under the assumption $\tau \lesssim \sqrt{h}$, we obtain the inequalities
	\begin{align}
		\abs*{\sum_{n=0}^m\aset{ u^{n+1} , Eu^{n+1} }_{\wt M_\Omega}} &\lesssim \sqrt{h}\, \pset*{ \norm{z^0}_M^2 + \norm{z^{m+1}}_M^2 + \sum_{n=0}^m\norm{Ez^n}_M^2 }, \\
		\abs*{\sum_{n=0}^m\aset{ u^{n+1} , Eu^{n+1} }_{\wt A_\Omega}} &\lesssim \begin{multlined}[t]
			\frac{\tau}{\sqrt{h}}\Bigg( \norm{\tfrac{1}{\tau}p^0}_{M_\Gamma}^2 + \norm{u^0}_{A_\Omega}^2 + \norm{\tfrac{1}{\tau}p^{m+1}}_{M_\Gamma}^2 \\ + \norm{u^{m+1}}_{A_\Omega}^2
			+ \sum_{n=0}^m\big( \norm{\tfrac{1}{\tau}Ep^n}_{M_\Gamma}^2 + \norm{Eu^n}_{A_\Omega}^2\big) \Bigg),
		\end{multlined}
	\end{align}
	which hold for all $m\geq 0$.
\end{lemma}
\begin{proof}
	Applying the discrete integration by parts formula~\eqref{eq:discreteIntegrationByParts:c} %, we obtain
%	\[
%	\aset{u^{n+1},Eu^{n+1}}_{\wt M_\Omega}
%	= \tfrac{1}{2}E\aset{u^{n+1},u^{n+1}}_{\wt M_\Omega} + \tfrac{1}{2}\aset{Eu^{n+1},Eu^{n+1}}_{\wt M_\Omega}
%	\]
%	and, therefore,
	for $\aset{\,\cdot\,, \cdot\,}_{\wt M_\Omega}$ yields
	\begin{align*}
		\abs*{\sum_{n=0}^m \aset{u^{n+1},Eu^{n+1}}_{\wt M_\Omega}}
		\leq \frac{1}{2}\,\Big(\abs[\big]{\aset{u^{0},u^{0}}_{\wt M_\Omega}}
		&+ \abs[\big]{\aset{u^{m+1},u^{m+1}}_{\wt M_\Omega}}
		\\
		&\qquad\quad+ \sum_{n=0}^m\abs[\big]{\aset{Eu^{n+1},Eu^{n+1}}_{\wt M_\Omega}}\Big).
	\end{align*}
	The first inequality then immediately follows by \Cref{lem:fakescalar_bound}.
	The second inequality is proven analogously.
\end{proof}

Let us now define the quantity
\begin{equation}
	H_0 
	\coloneqq \sum_{n=1}^3\, \norm[\big]{\tfrac{1}{\tau}Ez^n}_M^2 + \sum_{n=0}^3\, \norm{z^n}_A^2,
\end{equation}
which represents the energy of the initial values (see also \eqref{eq:PBE}).
Note that this holds for both the original and auxiliary discrete states $z^n$ and $\hat z^n$, since by construction $\hat z^n=z^n$ for $n\leq 3$.
Furthermore, we define the quantities
\begin{multline}
	\varphi^n 
	\coloneqq \tfrac{1}{6}\, \norm[\big]{\tfrac{1}{\tau}E\hat z^{n+4}}_{M}^2
	+ \tfrac{1}{3}\, \norm[\big]{\tfrac{1}{\tau}(E+\tfrac{3}{2}E^2)\hat z^{n+4}}_{M}^2
	+ \tfrac{1}{4}\, \norm{\hat z^{n+4}}_{A}^2
	+ \tfrac{1}{4}\, \norm{(I+E)\hat z^{n+4}}_{A}^2 \\
	+ \tsum_{k=0}^n \pset[\big]{ \tfrac{3}{4}\,\norm[\big]{\tfrac{1}{\tau}E^3\hat z^{k+4}}_{M_\Omega}^2 + \tfrac{1}{4}\,\norm{E^2\hat z^{k+4}}_{A}^2 }
\end{multline}
and
\begin{equation}
	b^n 
	\coloneqq \tau\,\norm[\big]{f_M^{n+4}}_M
	+ \sqrt{h}\tau\, \norm[\big]{f_{\wt M}^{n+4}}_M
	+ \tfrac{\tau}{\sqrt{h}}\, \norm[\big]{f_{\wt A,\Omega}^{n+4}}_{A_\Omega} 
	+ \tfrac{1}{\sqrt{h}}\, \norm[\big]{f_{\wt A,\Gamma}^{n+4}}_{M_\Gamma}
\end{equation}
for all $n\geq 0$.  
The following lemma relates these quantities to the normal and auxiliary states introduced in~\eqref{eq:defAuxStates}.

\begin{lemma}\label{lem:aux_vs_orig_state}
	Let $S\in\R[E]$ be a polynomial in $E$, where its degree $\ell\geq 0$ is independent of the discretization parameters.
	Then for every $n\geq\ell-4$ the following inequalities hold: 
	\begin{enumerate}
		\item $\norm{S(E)\hat z^{n+4}}_A \lesssim \begin{cases}\sqrt{H_0} & \text{if }n<0, \\ \sqrt{H_0} + \sqrt{\varphi^n} & \text{if }0\leq n<\ell-2 \\ \sqrt{\varphi^n} & \text{otherwise}.\end{cases}$
		\item If $E\mid S$, i.e., if the polynomial $S$ has no constant coefficient, then\\ 
		$\norm[\big]{\tfrac{1}{\tau}S(E)\hat z^{n+4}}_M \lesssim \begin{cases}\sqrt{H_0} & \text{if }n<0, \\ \sqrt{H_0} + \sqrt{\varphi^n} & \text{if }0\leq n<\ell-3 \\ \sqrt{\varphi^n} & \text{otherwise}.\end{cases}$
		%
%		\item $\norm{\tfrac{1}{\tau}E^{\ell+1}\hat z^{n+4}}_M^2+\norm{E^\ell \hat z^{n+4}}_A^2\lesssim\varphi^n$,
		%
%		\item $\norm{\tfrac{1}{\tau}E^\ell\hat z^{n+4} - \tfrac{1}{\tau}E^\ell z^{n+4}}_M^2 \lesssim h\varphi^n$,
		%
		\item $\norm{S(E)\hat z^{n+4} - S(E)z^{n+4}}_A \lesssim \begin{cases}0 & \text{if }n<0, \\ \frac{\tau}{\sqrt{h}} \big(\sqrt{H_0} + \sqrt{\varphi^n}\big) & \text{if }0\leq n<\ell+1 \\ \frac{\tau}{\sqrt{h}}\sqrt{\varphi^n} & \text{otherwise}.\end{cases}$
		\item $\norm[\big]{\tfrac{1}{\tau}S(E)\hat z^{n+4} - \tfrac{1}{\tau}S(E)z^{n+4}}_M \lesssim \begin{cases}0 & \text{if }n<0, \\ \sqrt{h}\,\big(\sqrt{H_0} + \sqrt{\varphi^n}\big) & \text{if }0\leq n<\ell+1 \\ \sqrt{h}\sqrt{\varphi^n} & \text{otherwise}.\end{cases}$
%		\item $\norm{E^\ell\hat z^{n+4} - E^\ell z^{n+4}}_A^2 \lesssim \tfrac{\tau^2}{h}\varphi^n$,
		%
%		\item $\norm{\tfrac{1}{\tau}ES(E)z^{n+4}}_M + \norm{S(E)z^{n+4}}_A^2\lesssim(1+\sqrt{h}+\frac{\tau}{\sqrt{h}})\sqrt{\varphi^n}$.
		\item $\norm{S(E)z^{n+4}}_A \lesssim \begin{cases} \sqrt{H_0} & \text{if }n<0, \\ (1+\tfrac{\tau}{\sqrt{h}})\big(\sqrt{H_0} + \sqrt{\varphi^n}\big) & \text{if }0\leq n<\ell+1 \\ (1+\tfrac{\tau}{\sqrt{h}})\sqrt{\varphi^n} & \text{otherwise}.\end{cases}$
		\item If $E\mid S(E)$, then $\norm[\big]{\tfrac{1}{\tau}S(E)z^{n+4}}_M \lesssim \begin{cases}\sqrt{H_0} & \text{if }n<0, \\ (1+\sqrt{h})\big(\sqrt{H_0} + \sqrt{\varphi^n}\big) & \text{if }0\leq n<\ell+1 \\ (1+\sqrt{h})\sqrt{\varphi^n} & \text{otherwise}.\end{cases}$
	\end{enumerate}
\end{lemma}

\begin{proof}
	Note that, since $\ell$ is independent of the discretization parameters, it is sufficient to prove the statements for $S(E)=E^\ell$ and $\ell\geq 0$ (or $\ell\geq 1$ for the second and last inequalities). The general result then follows by the triangle inequality.
	
	Furthermore, since the second, fourth, and sixth inequalities can be proven analogously to the first, third, and fifth, respectively, we will only show the latter.
	
	\begin{enumerate}
		\item For $\ell=0$ or $n<0$, the first inequality holds trivially by the definitions of $H_0$ and $\varphi^n$. Let us now assume $n\geq 0$.
		For $\ell=1$, we observe that
		\begin{align*}
			\norm{E\hat z^{n+4}}_A
			&= \norm{(I+E)\hat z^{n+4} - \hat z^{n+4}}_A \\
			&\lesssim \big(\norm{\hat z^{n+4}}_A^2 + \norm{(I+E)\hat z^{n+4}}_A^2\big)^{\frac{1}{2}}
			\lesssim \sqrt{\varphi^n}.
		\end{align*}
		For $\ell\geq 2$, we note that
		\begin{align*}
			\norm[\big]{E^\ell\hat z^{n+4}}_A
%			&\leq \tsum_{k=n-\ell+2}^n\norm{E^2\hat z^{k+4}}_A \\
			&\lesssim \big(\tsum_{k=n-\ell+2}^n\norm{E^2\hat z^{k+4}}_A^2\big)^{\frac{1}{2}}
			\\
			&\leq \big(\tsum_{k=n-\ell+2}^{-1}\norm{E^2\hat z^{k+4}}_A^2\big)^{\frac{1}{2}}
			+ \big(\tsum_{k=0}^n\norm{E^2\hat z^{k+4}}_A^2\big)^{\frac{1}{2}}
			\\
			&\lesssim \big(\tsum_{k=n-\ell}^{-1}\norm{\hat z^{k+4}}_A^2\big)^{\frac{1}{2}}
			+ \sqrt{\varphi^n}
			\lesssim \sqrt{H_0} + \sqrt{\varphi^n},
		\end{align*}
		where the first term in the last line actually vanishes for $n\geq\ell-2$.
		\addtocounter{enumi}{1}
		%
%		\item The second inequality is mostly analogous to the first one, shifting $\ell$ by one. Note that the case $\ell=0$ can be excluded due to the condition $E\mid S(E)$.
%		We remark that 
%		\begin{align*}
%			\norm{\tfrac{1}{\tau}E^{\ell+1}\hat z^{n+4}}_M
%			&\leq \tsum_{k=n-\ell+2}^n\norm{\tfrac{1}{\tau}E^3\hat z^{k+4}}_M \\
%			&\leq \sqrt{\ell-1}\big(\tsum_{k=n-\ell+2}^n\norm{\tfrac{1}{\tau}E^3\hat z^{k+4}}_M^2\big)^{\frac{1}{2}}
%			\lesssim \sqrt{\varphi^n}
%		\end{align*}
%		and analogously
%		\[
%			\norm{E^\ell\hat z^{n+4}}_A
%			\leq \sqrt{\ell-1}\big(\tsum_{k=n-\ell+2}^n\norm{E^2\hat z^{k+4}}_A^2\big)^{\frac{1}{2}}
%			\lesssim \sqrt{\varphi^n}.
%		\]
		%
		\item For the third inequality, we observe due to the consistency of the initial condition and the definition of $\hat z^n$ that
%		\[
%		\hat z^k - z^k =
%		\begin{cases}
%			\bsmat{0 \\ E^4\hat z_2^k \\ 0} & \text{for }k\geq 4, \\
%			0 & \text{otherwise}.
%		\end{cases}
%		\]
%		We deduce that
		\begin{align*}
			\norm[\big]{E^\ell\hat z^{n+4}-E^\ell z^{n+4}}_A
			&= \norm[\big]{E^\ell(\hat u_2^{n+4}-u_2^{n+4})}_{A_{22}}
			\leq \tsum_{k=n-\ell}^n\norm[\big]{\hat u_2^{k+4}-u_2^{k+4}}_{A_{22}}
			\\
			&\leq \tsum_{k=\max(0,n-\ell)}^n\norm[\big]{E^4\hat u_2^{k+4}}_{A_{22}}
			\lesssim \tfrac{\tau}{\sqrt{h}}\,\tsum_{k=\max(-1,n-\ell-1)}^n\norm[\big]{\tfrac{1}{\tau}E^3p^{k+4}}_{M_\Gamma},
		\end{align*}
		where the second to last sum actually vanishes for $n<0$.
		Let us now assume $n\geq 0$.
		Similarly to before, we split
		\begin{align*}
			\tsum_{k=\max(-1,n-\ell-1)}^n\norm[\big]{\tfrac{1}{\tau}E^3p^{k+4}}_{M_\Gamma}
			&\lesssim \norm[\big]{\tfrac{1}{\tau}E^3p^3}_{M_\Gamma}
			+ \tsum_{k=0}^n \big(\norm[\big]{\tfrac{1}{\tau}E^3p^{k+4}}_{M_\Gamma}^2\big)^{\frac{1}{2}}
			\\
			&\lesssim \big(\tsum_{k=1}^{3}\norm[\big]{\tfrac{1}{\tau}Ep^{k}}_{M_\Gamma}^2\big)^{\frac{1}{2}}
			+ \sqrt{\varphi^n}
			\lesssim \sqrt{H_0} + \sqrt{\varphi^n}
		\end{align*}
		and note that the first term is actually only there when $n\geq\ell+1$.
		\addtocounter{enumi}{1}
		%
%		\item The fourth inequality is analogous to the third one.
%		\begin{multline*}
%			\norm{\tfrac{1}{\tau}E^\ell\hat z^{n+4} - \tfrac{1}{\tau}E^\ell z^{n+4}}_M
%			= \norm{\tfrac{1}{\tau}E^\ell(\hat z^{n+4}-z^{n+4})}_M
%			\leq \tsum_{k=n-\ell}^{n} \norm{\tfrac{1}{\tau}(\hat z^{k+4}-z^{k+4})}_M
%			\\
%			\lesssim \tfrac{1}{\tau^2} \tsum_{k=\max(4,n+4-\ell)}^{n+4} \norm{E^4\hat z_2^k}_{M_{22}}^2
%			\lesssim \tfrac{h}{\tau^2} \tsum_{k=\max(4,n+4-\ell)}^{n+4} \norm{E^4p^k}_{M_\Gamma}^2
%			\\
%			\lesssim h\tsum_{k=4}^{n+4}\norm{\tfrac{1}{\tau}E^3p^k}_{M_\Gamma}^2 \leq h\varphi^n.
%		\end{multline*}
%		%
%		\item The third inequality follows analogously to the second one, since
%		%
%		\begin{multline*}
%			\norm{E^\ell\hat z^{n+4}-E^\ell\hat z^{n+4}}_A^2
%			\lesssim \tsum_{k=\max(4,n+4-\ell)}^{n+4} \norm{E^4\hat z_2^k}_{A_{22}}^2
%			\\
%			\lesssim \tfrac{\tau^2}{h}\tsum_{k=\max(4,n+4-\ell)}^{n+4} \norm{\tfrac{1}{\tau}E^4p^k}_{M_\Gamma}^2
%			\lesssim \tfrac{\tau^2}{h}\tsum_{k=4}^{n+4}\norm{\tfrac{1}{\tau}E^3p^k}_{M_\Gamma}^2 \leq \tfrac{\tau^2}{h}\varphi^n.
%		\end{multline*}
%		%
		\item The fifth inequality immediately follows from the first and third, since
		\[
		\norm{E^\ell z^{n+4}}_A 
		\leq \norm{E^\ell\hat z^{n+4}}_A + \norm{E^\ell\hat z^{n+4}-E^\ell z^{n+4}}_A. \qedhere
		\]
	\end{enumerate}
\end{proof}
%
%
%======================================================================
\subsection{Energy stability}

We are now ready to prove the following lemma, which is central to the stability of the method.

\begin{lemma}\label{lem:discreteEnergyStability}
	If $\tfrac{\tau}{\sqrt{h}}$ is sufficiently small, then it holds for all $m\geq 0$, 
	\begin{equation}\label{eq:varphiStability}
		\sqrt{\varphi^m} 
		\lesssim \sqrt{H_0} +  \sum_{n=0}^mb^n.
	\end{equation}
\end{lemma}

\begin{proof}
	Consider~\eqref{eq:multistep_aux:1} with test function $\tau\bdf_1\hat z^{n+4}$.
	In the following inequalities, we apply \Cref{lem:aux_vs_orig_state} several times.
	By applying \Cref{lem:bdf-bdf_formula} for $k=0,1$, we get  
	\begin{align*}
		\aset{\tau\bdf_1\hat z^{n+4} , \bdf_2\hat z^{n+4}}_M
		&= \tfrac{3}{4}\, \norm{\tfrac{1}{\tau}E^3\hat z^{n+4}}_M^2
		+ E\big(\tfrac{1}{6}\norm{\tfrac{1}{\tau}E\hat z^{n+4}}_M^2
		+ \tfrac{1}{3}\norm{\tfrac{1}{\tau}(E+\tfrac{3}{2}E^2)\hat z^{n+4}}_M^2\big), \\
		\aset{\tau\bdf_1\hat z^{n+4} , \hat z^{n+4}}_A
		&= \tfrac{1}{4}\, \norm{E^2\hat z^{n+4}}_A^2
		+ E\big(\tfrac{1}{4}\norm{\hat z^{n+4}}_A^2
		+ \tfrac{1}{4}\norm{(I+E)\hat z^{n+4}}_A^2\big)
	\end{align*}
	and, therefore,  
	\begin{align*}
		&\sum_{k=0}^n \Big(
		\aset{\tau\bdf_1\hat z^{k+4} , \bdf_2\hat z^{k+4}}_M
		+ \aset{\tau\bdf_1\hat z^{k+4} , \hat z^{k+4}}_A
		\Big)
		\\
		&\quad= \varphi^n
		- \tfrac{1}{6}\norm{\tfrac{1}{\tau}E\hat z^3}_M^2
		- \tfrac{1}{3}\norm{\tfrac{1}{\tau}(E+\tfrac{3}{2}E^2)\hat z^3}_M^2
		- \tfrac{1}{4}\norm{\hat z^3}_A^2
		- \tfrac{1}{4}\norm{(I+E)\hat z^3}_A^2
		\geq \varphi^n - CH_0.
	\end{align*}
	For the perturbation terms, we obtain by applying discrete integration by parts, 
	\begin{align*}
		\aset{(E+\tfrac{1}{2}&E^2)\hat u^{k+4},E^4\hat u^{k+4}}
		\\
		&= E\aset{E\hat u^{k+4},E^3\hat u^{k+4}}
		- \aset{E^2\hat u^{k+4},E^3\hat u^{k+4}}
		+ \tfrac{3}{2}\aset{E^2\hat u^{k+4},E^4\hat u^{k+4}}
		\\
		&= E\aset{(E+\tfrac{3}{2}E^2)\hat u^{k+4},E^3\hat u^{k+4}}
		- \aset{E^2\hat u^{k+4}, E^3\hat u^{k+4}}
		- \tfrac{3}{2}\aset{E^3\hat u^{k+4},E^3\hat u^{k+3}}
	\end{align*}
	for all $k\geq 0$. Therefore, using \Cref{lem:fakescalar_bound,lem:fakenorm_bound2},
	\begin{align*}
		\abs[\big]{\tsum_{k=0}^n &\aset{(E+\tfrac{1}{2}E^2)\hat u^{k+4},\tfrac{1}{\tau^2}E^4\hat u^{k+4}}_{\wt M_\Omega}}
		\\
		&\leq \abs[\big]{\aset[\big]{\tfrac{1}{\tau}(E+\tfrac{3}{2}E^2)\hat u^{3},\tfrac{1}{\tau}E^3\hat u^{3}}_{\wt M_\Omega}}
		+ \abs[\big]{\aset[\big]{\tfrac{1}{\tau}(E+\tfrac{3}{2}E^2)\hat u^{n+4},\tfrac{1}{\tau}E^3\hat u^{n+4}}_{\wt M_\Omega}}
		\\
		&\qquad+ \tsum_{k=0}^n\abs[\big]{\aset[\big]{\tfrac{1}{\tau}E^2\hat u^{k+4}, \tfrac{1}{\tau}E^3\hat u^{k+4}}_{\wt M_\Omega}}
		+ \tsum_{k=0}^n\abs[\big]{\aset[\big]{\tfrac{1}{\tau}E^3\hat u^{k+4},\tfrac{1}{\tau}E^3\hat u^{k+3}}_{\wt M_\Omega}}
		\\
		&\lesssim \sqrt{h}\,\Big(
			\norm[\big]{\tfrac{1}{\tau}(E+\tfrac{3}{2}E^2)\hat z^3}_M \norm[\big]{\tfrac{1}{\tau}E^3\hat z^3}_M
			+ \norm[\big]{\tfrac{1}{\tau}(E+\tfrac{3}{2}E^2)\hat z^{n+4}}_M \norm[\big]{\tfrac{1}{\tau}E^3\hat z^{n+4}}_M \\
			&\qquad+ \norm[\big]{\tfrac{1}{\tau}E^2\hat z^3}_M^2
			+ \norm[\big]{\tfrac{1}{\tau}E^2\hat z^{n+4}}_M^2
			+ \tsum_{k=0}^n\norm[\big]{\tfrac{1}{\tau}E^3\hat z^{k+4}}_M^2
		\Big)
		\lesssim \sqrt{h}\, \big(H_0 + \varphi^n\big)
	\end{align*}
	and
	\begin{align*}
		\abs[\big]{&\tsum_{k=0}^n \aset{(E+\tfrac{1}{2}E^2)\hat u^{k+4},E^4\hat u^{k+4}}_{\wt A_\Omega}}
		\\
		&\leq \abs[\big]{\aset{(E+\tfrac{3}{2}E^2)\hat u^{3},E^3\hat u^{3}}_{\wt A_\Omega}}
		+ \abs[\big]{\aset{(E+\tfrac{3}{2}E^2)\hat u^{n+4},E^3\hat u^{n+4}}_{\wt A_\Omega}}
		\\
		&\qquad+ \tsum_{k=0}^n\abs[\big]{\aset{E^2\hat u^{k+4}, E^3\hat u^{k+4}}_{\wt A_\Omega}}
		+ \tsum_{k=0}^n\abs[\big]{\aset{E^3\hat u^{k+4},E^3\hat u^{k+3}}_{\wt A_\Omega}}
		\\
		&\lesssim \tfrac{\tau}{\sqrt{h}}\,\Big(
			\big(\norm[\big]{(E+\tfrac{3}{2}E^2)\hat u^{3}}_{A_\Omega}
			+ \norm[\big]{\tfrac{1}{\tau}(E+\tfrac{3}{2}E^2)p^3}_{M_\Gamma}\big)
			\big(\norm{E^3\hat u^{3}}_{A_\Omega}
			+ \norm[\big]{\tfrac{1}{\tau}E^3p^3}_{M_\Gamma}\big) \\
			&\qquad+ \big(\norm{(E+\tfrac{3}{2}E^2)\hat u^{n+4}}_{A_\Omega}
			+ \norm[\big]{\tfrac{1}{\tau}(E+\tfrac{3}{2}E^2)p^{n+4}}_{M_\Gamma}\big)
			\big(\norm{E^3\hat u^{n+4}}_{A_\Omega}
			+ \norm[\big]{\tfrac{1}{\tau}E^3p^{n+4}}_{M_\Gamma}\big) \\
			&\qquad+ \norm{E^2\hat u^{3}}_{A_\Omega}^2
			+ \norm[\big]{\tfrac{1}{\tau}E^2p^{3}}_{M_\Gamma}^2
			+ \norm{E^2\hat u^{n+4}}_{A_\Omega}^2
			+ \norm[\big]{\tfrac{1}{\tau}E^2p^{n+4}}_{M_\Gamma}^2 \\
			&\qquad+ \tsum_{k=0}^n\big( \norm{E^3\hat u^{k+4}}_{A_\Omega}^2
			+ \norm[\big]{\tfrac{1}{\tau}E^3p^{k+4}}_{M_\Gamma}^2\big)
		\Big)
		\lesssim \tfrac{\tau}{\sqrt{h}}\, \big(H_0 + \varphi^n\big)
	\end{align*}
	for all $n\geq 0$.
	For the nonlinear terms, we observe 
	\begin{align*}
		\abs[\big]{\tsum_{k=0}^n\aset{(E+\tfrac{1}{2}E^2)\hat z^{k+4},f_M^{k+4}}}_M
		&\lesssim \tau\, \tsum_{k=0}^n\sqrt{\varphi^k}\, \norm{f_M^{k+4}}_M, \\
		%
%		\abs[\big]{\tsum_{k=0}^n\aset{(E+\tfrac{1}{2}E^2)\hat z^{k+4},f_A^{k+4}}}_A
%		&\lesssim \tsum_{k=0}^n\sqrt{\varphi^k}\norm{f_A^{k+4}}_A, \\
		%
		\abs[\big]{\tsum_{k=0}^n\aset{(E+\tfrac{1}{2}E^2)\hat z^{k+4},f_{\wt M}^{k+4}}}_{\wt M}
		&\lesssim \tau\sqrt{h}\,\tsum_{k=0}^n\sqrt{\varphi^k}\,\norm{f_{\wt M}^{k+4}}_M, \\
		\abs[\big]{\tsum_{k=0}^n\aset{(E+\tfrac{1}{2}E^2)\hat z^{k+4},f_{\wt A}^{k+4}}}_{\wt A}
		&\lesssim \tfrac{\tau}{\sqrt{h}}\, \tsum_{k=0}^n\sqrt{\varphi^k}\, \big(\norm{f_{A,\Omega}^{k+4}}_{A_\Omega} + \tfrac{1}{\tau}\norm{f_{A,\Gamma}^{n+4}}_{M_\Gamma}\big),
	\end{align*}
	where we applied \Cref{lem:fakescalar_bound} for the last two inequalities.
	Putting everything together, we obtain
	\[
	\pset[\big]{1-C(h+\tfrac{\tau}{\sqrt{h}})} \varphi^n
	\lesssim \big(1+h+\tfrac{\tau}{\sqrt{h}}\big) H_0
	+ \tsum_{k=0}^n\sqrt{\varphi^k}b^k
	\]
	for all $n\geq 0$.
	Then, as long as $h$ and $\tfrac{\tau}{\sqrt{h}}$ are sufficiently small, we deduce that
	\begin{equation}
		\varphi^n \lesssim H_0 + \tsum_{k=0}^n\sqrt{\varphi^k}\, b^k
	\end{equation}
	for all $n\geq 0$.
	By applying a discrete Grönwall lemma (see e.g.~\cite[Lem.~8.13]{Zim21phd}), we finally obtain the claimed stability bound. 
	%
%	\begin{equation}
%		\sqrt{\varphi^m} 
%		\lesssim \sqrt{H_0} + \tsum_{n=0}^mb^n.
%		\qedhere
%	\end{equation}
	%
\end{proof}

By combining \Cref{lem:discreteEnergyStability,lem:aux_vs_orig_state}, we immediately deduce the following energy stability bound.

\begin{theorem}\label{thm:discreteEnergyStability}
	If $\tfrac{\tau}{\sqrt{h}}$ is sufficiently small and $f_{\wt M}=f_{\wt A}=0$, then
	\begin{equation}\label{eq:discreteEnergyStability}
		\norm[\big]{\tfrac{1}{\tau}E\hat z^{m+4}}_M^2 + \norm[\big]{\hat z^{m+4}}_A^2 
		\lesssim H_0 + \pset*{ \tau \sum_{n=0}^m \norm{f_M^{n+4}}_M }^2
	\end{equation}
	for all $m\geq 0$. In particular, if $f=0$, then the energy of the discrete system is bounded independently from the discretization parameters. The same holds true by replacing the auxiliary state $\hat z^n$ with the original state $z^n$.
\end{theorem}

\begin{remark}
	By interrupting the proof of \Cref{lem:discreteEnergyStability} before applying the discrete Grönwall lemma, we can replace \eqref{eq:discreteEnergyStability} by
	\begin{equation*}
		\norm[\big]{\tfrac{1}{\tau}E\hat z^{m+4}}_M^2 + \norm[\big]{\hat z^{m+4}}_A^2 
		\lesssim H_0 + \tau \sum_{n=0}^m\norm{\bdf_1\hat z^{n+4}}_M\norm{f_M^{n+4}}_M.
	\end{equation*}
	While this inequality is less expressive in terms of the energy stability, it is evidently closely related to \eqref{eq:PBE}. %The dissipation inequality of the system is then in some sense preserved.
\end{remark}

%
%======================================================================
\subsection{State stability}

While it might seem like \Cref{lem:discreteEnergyStability} and \Cref{thm:discreteEnergyStability} only provide the stability of the energy of the system, we can show that for the model problem the stability transfers to the state.
\begin{lemma}\label{lem:M_to_A_bound}
	If $z=(u,p)\in\R^{N_\Omega+N_\Gamma}$ is such that $M_\lambda u_2=B_\Gamma p$, then $\norm{z}_M\lesssim\norm{z}_A$.
\end{lemma}
\begin{proof}
	By construction of $M_\Gamma$ and $A_\Gamma$, it holds that $M_\Gamma\leq A_\Gamma$ in the sense of the Loewner partial order, i.e., $A_\Gamma-M_\Gamma\geq 0$.
	Thus, we deduce that
	\[
	\norm{p}_{M_\Gamma}^2
	\leq \norm{p}_{A_\Gamma}^2
	\leq \norm{u}_{A_\Omega}^2 + \norm{p}_{A_\Gamma}^2
	= \norm{z}_{A}^2,
	\]
	in particular $\norm{p}_{M_\Gamma}\lesssim\norm{z}_A$.
	Let us now recall that $z=(u,p)$ are the coordinates of some functions $z_h=(u_h,p_h)$ with $u_h\in H^1(\Omega_h)$ and $p_h\in H^1(\Gamma_h)$ in terms of the basis functions used in the finite element construction.
	Then, by using a variant of the Poincaré inequality \cite[Lem.~3.30]{ErnG21}, and the fact that $\tr(u_h)$ is due to $M_\lambda u_2=B_\Gamma p$ the $L^2(\Gamma_h)$-best approximation of $p_h$ in its corresponding space (see~\cite[Lem.~4.3]{AltZ24}), we obtain
	\begin{align*}
		\norm{u}_{M_\Omega}
		&= \norm{u_h}_{L^2(\Omega_h)}
		\lesssim \norm{\nabla u_h}_{L^2(\Omega_h)} + \norm{\tr(u_h)}_{L^2(\Gamma_h)}
		\\
		&\leq \norm{\nabla u_h}_{L^2(\Omega_h)} + \norm{p_h}_{L^2(\Gamma_h)}
		= \norm{u}_{A_\Omega} + \norm{p}_{M_\Gamma}
		\leq \norm{z}_A.
	\end{align*}
	We conclude that $\norm{z}_M\lesssim\norm{u}_{M_\Omega}+\norm{p}_{M_\Gamma}\lesssim\norm{z}_A$.
\end{proof}

From \Cref{lem:M_to_A_bound}, we deduce that $\norm{S(E)\hat z^{n+4}}_M\lesssim\norm{S(E)\hat z^{n+4}}_A$ and
\[
	\norm{S(E)z^{n+4}}_M 
	\lesssim \norm{S(E)\hat z^{n+4}}_A + \tau\, \norm{\tfrac{1}{\tau}S(E)\hat z^{n+4} - \tfrac{1}{\tau}S(E)z^{n+4}}_M
\]
for all $S\in\R[E]$ of degree $\ell$ and $n\geq\ell-4$, which allows us to provide bounds with \Cref{lem:aux_vs_orig_state} without requiring $E\mid S$. In particular, we have $\norm{\hat z^{n+4}}_M\lesssim\norm{\hat z^{n+4}}_A\lesssim\sqrt{\varphi^n}$ and
\[
	\norm{z^{n+4}}_M 
	\lesssim \norm{\hat z^{n+4}}_A + \tau\,\norm{\tfrac{1}{\tau}\hat z^{n+4}-\tfrac{1}{\tau}z^{n+4}}_M
	\lesssim \big(1+\tau\sqrt{h}\big)\big(\sqrt{H_0} + \sqrt{\varphi^n}\big)
\]
for all $n\geq 0$. Thus, \Cref{lem:discreteEnergyStability} yields bounds for $\norm{\hat z^{n+4}}_M$ and $\norm{z^{n+4}}_M$. The resulting estimates for $\norm{\hat z^{n+4}}_{M+A}$ and $\norm{z^{n+4}}_{M+A}$ are summarized as follows.

\begin{theorem}\label{thm:discreteStateStability}
	If $\tfrac{\tau}{\sqrt{h}}$ is sufficiently small, then
	\begin{equation}
		\norm{\hat z^{m+4}}_M \leq \norm{\hat z^{m+4}}_{M+A} \lesssim \sqrt{H_0} + \sum_{n=0}^mb^n
	\end{equation}
	for all $m\geq 0$.
	In particular, for $f_{\wt M}=f_{\wt A}=0$, we have
	\begin{equation}
		\norm{\hat z^{m+4}}_M \leq \norm{\hat z^{m+4}}_{M+A} \lesssim \sqrt{H_0} + \tau\sum_{n=0}^m\, \norm{f_M^{n+4}}_M.
	\end{equation} 
	The same inequalities hold by replacing $\hat z^{m+4}$ with $z^{m+4}$.
\end{theorem}
If $f$ depends on time only, then \Cref{thm:discreteEnergyStability,thm:discreteStateStability} give a satisfactory characterization of the stability of the energy and of the state of the discrete system. When $f$ also depends on the state, the inequalities in their present form become less informative. However, when $f$ can be bounded in terms of the norm of the state, we can still deduce weaker stability inequalities.
Let us first define
\begin{equation}
	\tilde b^n 
	\coloneqq b^n - \tau\,\norm[\big]{f_M^{n+4}}_M 
	= \sqrt{h}\tau\,\norm[\big]{f_{\wt M}^{n+4}}_M
	+ \tfrac{\tau}{\sqrt{h}}\, \norm[\big]{f_{\wt A,\Omega}^{n+4}}_{A_\Omega} + \tfrac{1}{\sqrt{h}}\, \norm[\big]{f_{\wt A,\Gamma}^{n+4}}_{M_\Gamma}
\end{equation}
for all $n\geq 0$. Then we have the following result.

\begin{lemma}\label{lem:LipschitzStability}
	Suppose that $f_M$ satisfies
	\begin{equation}\label{eq:f_bounded_condition}
		\norm{f_M(t,z)}_M
		\leq L\, \norm{z}_{M+A} + b_M(t)
	\end{equation}
	for all relevant $t$ and $z$, for some fixed $L\geq 0$ and $b_M\colon [0,T]\to\R$ with $b_M\geq 0$ pointwise, where $L$ and $b_M$ are independent of the discretization parameters.
	Then we obtain for sufficiently small $\tfrac{\tau}{\sqrt{h}}$ the estimate 
	\begin{equation}
		\label{eq:Lem49_estimate}
		\sqrt{\varphi^m} 
		\lesssim e^{CL\tau(m+1)}\pset*{ \sqrt{H_0} + \sum_{n=0}^m\pset[\big]{ \tau b_M(t^{n+4}) + \tilde b^n } }
	\end{equation}
	for all $m\geq 0$.
\end{lemma}

\begin{proof}
	We have
	\[
	\norm{f_M^{n+4}}_M
	= \norm{f_M(t^{n+4},\ddf_0\hat z^{n+4})}_M
	\leq L\,\norm{\ddf_0\hat z^{n+4}}_{M+A} + b_M(t^{n+4})
	\]
	for all $n\geq 0$. We observe that 
	\[
	\sum_{k=0}^n\, \norm{\ddf_0\hat z^{k+4}}_{M+A}
	\lesssim \sum_{k=0}^n\, \norm{(I-E^4)\hat z^{k+4}}_{A}
	\lesssim \sqrt{H_0} + \sum_{k=0}^n\, \sqrt{\varphi^k}
	\]
%	\begin{align*}
%		\tsum_{k=0}^n\norm{\ddf_0\hat z^{k+4}}_{M+A}
%		&\leq \tsum_{k=0}^n\big(\norm{\hat z^{k+4}}_{M+A} + \norm{E^4\hat z^{k+4}}_{M+A} + \norm{(I-E^4)\hat z^{k+4}}_{M+A}\big) \\
%		&\lesssim \sqrt{H_0} + \tsum_{k=0}^n\sqrt{\varphi^k}
%	\end{align*}
	%
	and, therefore,  
	\[
	\sqrt{\varphi^n} 
	\lesssim \sqrt{H_0} + \sum_{k=0}^nb^k
	\lesssim \sqrt{H_0} + L\tau\, \sum_{k=0}^n\sqrt{\varphi^k}
	+ \sum_{k=0}^n\big(\tau b_M(t^{k+4}) + \tilde b^k\big)
	\]
%	or, equivalently,
%	\[
%	\sqrt{\varphi^n} 
%	\leq CL\,\tau\,\tsum_{k=0}^{n}\sqrt{\varphi^k} + C\Big(\sqrt{H_0} + \tsum_{k=0}^n\big(\tau b_M(t^{k+4}) + \tilde b^k\big)\Big)
%	\]
	due to \Cref{lem:discreteEnergyStability}, for all $n\geq 0$.  
	Then, assuming $\tau$ to be sufficiently small, we can apply \Cref{lem:discreteGronwallMinimal} to obtain estimate~\eqref{eq:Lem49_estimate}. 
%	\[
%	\sqrt{\varphi^n} 
%	\lesssim e^{CL\tau(m+1)}\Big(\sqrt{H_0}+\tsum_{k=0}^m\big(\tau b_M(t^{k+4})+\tilde b^k\big)\Big),
%	\]
%	allowing us to conclude.
\end{proof}

\begin{remark}
	As an example, condition~\eqref{eq:f_bounded_condition} is satisfied for Lipschitz continuous functions. In fact, if $f_M$ satisfies
	\begin{equation}\label{eq:Lipschitz_semidisc}
		\norm{f_M(t,z_1)-f_M(t,z_2)}_M 
		\leq L\, \norm{z_1-z_2}_{M+A}
	\end{equation}
	for all relevant $t,z_1,z_2$, then in particular
	\[
	\norm{f_M(t,z)}_M 
	\leq \norm{f_M(t,z)-f_M(t,0)}_M + \norm{f_M(t,0)}_M 
	\leq L\, \norm{z}_{M+A} + \norm{f_M(t,0)}_M
	\]
	for all $z$. Thus, condition~\eqref{eq:f_bounded_condition} holds with $b_M(t)=\norm{f_M(t,0)}_M$.
	Note that, for the aforementioned default construction with $f_M=M^{-1}f$, the Lipschitz condition \eqref{eq:Lipschitz_semidisc} for the semi-discretized inhomogeneity can be written equivalently as
	\begin{equation}\label{eq:Lipschitz_invNorm}
		\norm{f(t,z_1) - f(t,z_2)}_{M^{-1}} 
		\leq L\, \norm{z_1-z_2}_{M+A}
	\end{equation}
	and is implied by the Lipschitz continuity of the original PDE in the sense that
	% Why corresponds the $M^{-1}$-norm to the $L^2$-norm?
	% Answer: let $f\in L^2(\Omega)$, let $u_1,\ldots,u_n\in L^2(\Omega)$ be the chosen basis for $\mathcal V_h\subseteq L^2(\Omega)$, let $M=[\aset{u_i,u_j}]_{i,j}$ be the mass matrix, let $f_h$ be the semidiscretized version of $f$, let $\tilde f=\sum\tilde f_iu_i$ be the orthogonal projection of $f$ onto $\mathcal V_h$, and let $\tilde f_h=(\tilde f_i)_i\in\R^n$. Then
%		\[
%		f_h
%		= [\aset{u_i,f}_{L^2}]_i
%		= [\aset{u_i,\tilde f}_{L^2}]_i
%		= [\aset{u_i,\tsum_j\tilde f_ju_j}_{L^2}]_i
%		= \tsum_j[\aset{u_i,u_j}_{L^2}]_i\tilde f_j
%		= M\tilde f_h.
%		\]
%		We deduce that
%		\[
%		\norm{f_h}_{M^{-1}}^2
%		= \norm{\tilde f_h}_M^2
%		= \norm{\tsum_i \tilde f_iu_i}_{L^2}^2
%		= \norm{\tilde f}_{L^2}^2
%		\leq \norm{f}_{L^2}^2.
%		\]
%		Do you have a shorter way to explain this?
	%
	\begin{align}
		\label{eq:Lipschitz_PDE}
		\norm{f_\Omega(t,u) - f_\Omega(t,\tilde u)}_{L^2(\Omega)}^2 + \norm{f_\Gamma(t,&p) - f_\Gamma(t,\tilde p)}_{L^2(\Gamma)}^2 \\
		&\leq L^2\big( \norm{u-\tilde u}_{H^1(\Omega)}^2 + \norm{p-\tilde p}_{H^1(\Gamma)}^2 \big) \notag
	\end{align}
	for all relevant $t,u,\tilde u,p,\tilde p$.
\end{remark}

We can then update our energy and state stability results as follows.

\begin{theorem}\label{thm:LipschitzStability}
	Assume $\tfrac{\tau}{\sqrt{h}}$ is sufficiently small, $f_M$ satisfies condition \eqref{eq:f_bounded_condition}, and $f_{\wt M}=f_{\wt A}=0$. Then
	\begin{align}
		& \norm[\big]{\tfrac{1}{\tau}E\hat z^{m+4}}_M^2 + \norm{\hat z^{m+4}}_A^2 
		\lesssim e^{CL\tau(m+1)}\pset*{H_0 + \pset*{ \tau \sum_{n=0}^m b_M(t^{n+4}) }^2}, \\
		& \norm{\hat z^{m+4}}_M \lesssim e^{CL\tau(m+1)}\pset*{\sqrt{H_0} + \tau \sum_{n=0}^m b_M(t^{n+4})}
	\end{align}
	hold for all $m\geq 0$.
	The same inequalities hold by replacing $\hat z^{m+4}$ with $z^{m+4}$.
\end{theorem}

\begin{remark}\label{rem:discreteStateStability}
	The results in this subsection are all based on \Cref{lem:M_to_A_bound}, which does not in general hold for the extended model \eqref{eq:semidisc_ext}.
	In fact, we only have $M_\Gamma\leq A_\Gamma$ because the term $-\Delta_\Gamma p+p$ is involved in the construction of $A_\Gamma$.
	
	In the context of the extended model, the inequality $M_\Gamma\lesssim A_\Gamma$ corresponds to $\kappa_\Gamma>0$, while $\beta_\Omega>0$ is necessary to deduce $\norm{\nabla u_h}_{L^2(\Omega_h)}\lesssim\norm{u}_{A_\Omega}$.
	A similar reasoning would also allow us to prove \Cref{lem:M_to_A_bound} if we replaced $\beta_\Omega>0$ with $\kappa_\Omega>0$.
	In the other cases, a different argument has to be used to deduce appropriate state stability bounds. 
	
	Concerning the energy stability, by observing that
	\[
		\norm{\hat z^{m+4}}_M 
		\leq \norm{z^3}_M + \tau\,\sum_{n=0}^m\,\norm[\big]{\tfrac{1}{\tau}E\hat z^{n+4}}_M
		\lesssim \norm{z^3}_M + \tau\,\sum_{n=0}^m\sqrt{\varphi^n},
	\]
	we conclude that
	\begin{equation}
		\norm{\hat z^{m+4}}_M 
		\lesssim \norm{z^3}_M + (m+1)\, \tau\, \big( \sqrt{H_0} + \tsum_{n=0}^m\norm{f_M^{n+4}} \big)
	\end{equation}
	for all $m\geq 0$. Hence, $\hat z^{m+4}$ is controlled by a similar term as in the model problem, but in a way that scales linearly with the length of the time interval.
	By applying \Cref{lem:aux_vs_orig_state}, the same inequality can be extended to $z^{m+4}$.
	
	To adapt \Cref{lem:LipschitzStability} and \Cref{thm:LipschitzStability} to the extended model, one can introduce modified quantities
	\begin{equation}
		H_1 
		\coloneqq \tsum_{n=1}^3\norm[\big]{\tfrac{1}{\tau}Ez^n}_M^2 + \tsum_{n=0}^4\norm{z^n}_{M+A}^2 \geq H_0
	\end{equation}
	and
	\begin{multline}
		\psi^n 
		\coloneqq \tfrac{1}{6}\,\norm[\big]{\tfrac{1}{\tau}E\hat z^{n+4}}_{M}^2
		+ \tfrac{1}{3}\,\norm[\big]{\tfrac{1}{\tau}(E+\tfrac{3}{2}E^2)\hat z^{n+4}}_{M}^2
		+ \tfrac{1}{4}\,\norm{\hat z^{n+4}}_{M+A}^2 \\
		+ \tfrac{1}{4}\,\norm{(I+E)\hat z^{n+4}}_{M+A}^2
		+ \tsum_{k=0}^n \pset[\big]{ \tfrac{3}{4}\,\norm[\big]{\tfrac{1}{\tau}E^3\hat z^{k+4}}_{M}^2 + \tfrac{1}{4}\,\norm{E^2\hat z^{k+4}}_{M+A}^2 } \geq \varphi^n.
	\end{multline}
	Then, under condition \eqref{eq:f_bounded_condition}, it can be analogously proven that
	\begin{equation}
		\sqrt{\psi^m} \lesssim e^{C(L+1)\tau(m+1)}\pset*{ \sqrt{H_1} + \sum_{n=0}^m\pset[\big]{ \tau b_M(t^{n+4}) + \tilde b^n } }.
	\end{equation}
	From this, analogous bounds as in \Cref{thm:LipschitzStability} follow with the important difference that the exponential term does not vanish for $L=0$.
\end{remark}

\subsection{Convergence}

We start by observing that the error of the auxiliary state can be interpreted as the solution of the same discrete system with a modified inhomogeneity.

\begin{lemma}\label{lem:error_solutions}
	Let $(\hat e_z^n,e_\lambda^n)=(\hat z^n-z(t^n),\lambda^n-\lambda(t^n))$ for $n\geq 0$ denote the error of the auxiliary discrete solution. Then $(\hat e_z^n,e_\lambda^n)$ is the auxiliary discrete solution of the discrete system \eqref{eq:multistep_aux}, where the inhomogeneity $f$ is replaced by $Mf_M+\wt Mf_{\wt M}+\wt Af_{\wt A}$ with $f_{\wt M}(t,e_z) \coloneqq \tfrac{2}{\tau^2}E^4z(t)$, $f_{\wt A}(t,e_z) \coloneqq E^4z(t)$, and 
	\begin{equation}
		f_M(t,e_z) \coloneqq M^{-1}\Big( f\pset[\big]{t,e_z + \ddf_0z(t)} - f\pset[\big]{t,z(t)} \Big) + \ddot z(t) - \bdf_2z(t).
	\end{equation}
	Therein, $z(t)=0$ for $t<0$ and the initial condition is replaced by $\hat e_z^n=e_z^n=z^n-z(t^n)$ for $n=0,\ldots,3$.
	Furthermore, the corresponding non-auxiliary solution $e_z^n$ satisfies
	\begin{equation}\label{eq:error_discrepancy}
		e_z^{n+4} = z^{n+4} - z(t^{n+4}) + \bsmat{0 \\ E^4u_2(t^{n+4}) \\ 0}
	\end{equation}
	for all $n\geq 0$.
\end{lemma}

\begin{proof}
	By combining \eqref{eq:semidiscconcise} and \eqref{eq:multistep_aux}, we obtain
	\begin{multline*}
		M\bdf_2\big(\hat z^{n+4}-z(t^{n+4})\big) + A\big(\hat z^{n+4}-z(t^{n+4})\big) + B^\top\big(\lambda^{n+4}-\lambda(t^{n+4})\big)
		\\
		= f\pset{ t^{n+4} , \ddf_0\hat z^{n+4} }
		- f\pset[\big]{ t^{n+4} , z(t^{n+4}) }
		+ M\big( \ddot z(t^{n+4}) - \bdf_2z(t^{n+4}) \big)
		\\
		+ E^4\pset*{ \tfrac{2}{\tau^2}\wt M + \wt A }z(t^{n+4})
		+ E^4\pset*{ \tfrac{2}{\tau^2}\wt M + \wt A }\big(\hat z^{n+4}-z(t^{n+4})\big)
	\end{multline*}
	and $B(\hat z^{n+4}-z(t^{n+4}))=0$ for all $n\geq 0$, which is equivalent to \eqref{eq:multistep_aux} where $\hat z$ is replaced by $\hat e_z$ and $f$ by $Mf_M+\wt Mf_{\wt M}+\wt Af_{\wt A}$ as defined in the statement of this lemma.
	
	Since for the initial condition of the original system we have $\hat z^n=z^n$ and $M_\lambda u_2^n=B_\Gamma p^n$ for $n=0,\ldots,3$, for the error we have
	\[
	e_z^n = \hat e_z^n = \hat z^n - z(t^n) = z^n - z(t^n)
	\]
	for the same values of $n$ and, therefore,
	\[
	M_\lambda e_{u,2}^n = M_\lambda u_2^n - M_\lambda u_2(t^n) = B_\Gamma p^n - B_\Gamma p(t^n)
	= B_\Gamma e_p^n,
	\]
	where $e_z^n=(e_{u,1}^n,e_{u,2}^n,e_p^n)$, so the assumptions for the initial condition also hold in this case.
	Finally, we observe that
	\[
	e_z^{n+4}
	= \bsmat{e_{u,1}^{n+4} \\ \ddf_0\hat e_{u,2}^{n+4} \\ e_p^{n+4}}
	= \bsmat{u_1^{n+4} \\ \ddf_0\hat u_2^{n+4} \\ p^{n+4}}
	- \bsmat{u_1(t^{n+4}) \\ (I-E^4)u_2(t^{n+4}) \\ p(t^{n+4})}
	= z^{n+4} - z(t^{n+4}) + \bsmat{0 \\ E^4u_2(t^{n+4}) \\ 0}
	\]
	for all $n\geq 0$.
\end{proof}

Note that \eqref{eq:error_discrepancy} shows that, while in general $e_z^{n+4}\neq z^{n+4}-z(t^{n+4})$ for $n\geq 0$, these two quantities are quite close. In fact,
\begin{multline}
	\norm[\big]{e_z^{n+4} - \big(z^{n+4}-z(t^{n+4})\big)}_{M+A}
	= \norm[\big]{E^4u_2(t^{n+4})}_{M_{22}+A_{22}}
	\\
	\lesssim \big(\sqrt{h}+\tfrac{1}{\sqrt{h}}\big) \norm[\big]{E^4p(t^{n+4})}_{M_\Gamma}
	\leq \tfrac{\tau}{\sqrt{h}}(1+h)\,\norm[\big]{\tfrac{1}{\tau}E^4p(t^{n+4})}_{M_\Gamma}.
\end{multline}
Hence, we easily deduce the following statement.

\begin{corollary}\label{cor:error_of_error}
	Let $S\in\R[E]$ be a polynomial whose degree $\ell\geq 0$ is independent of the discretization parameters.
	If $\tfrac{\tau}{\sqrt{h}}$ is sufficiently small and the solution of the semi-discretized system \eqref{eq:semidisc:kinetic} satisfies $p\in W^{3,\infty}([0,T],\R^{N_\Gamma})$ with respect to the $M_\Gamma$-norm, then
	\begin{subequations}
	\begin{align}
		\norm[\big]{\tfrac{1}{\tau}S(E)\big(e_z^{n+4} - (z^{n+4}-z(t^{n+4}))\big)}_{M} 
		&\lesssim \tau^2\, \norm{p^{(3)}}_{L^\infty(t^{n-\ell},t^{n+4}),M_\Gamma}, \\
		\norm[\big]{S(E)\big(e_z^{n+4} - (z^{n+4}-z(t^{n+4}))\big)}_{M+A} &\lesssim \tau^2\, \norm{p^{(3)}}_{L^\infty(t^{n-\ell},t^{n+4}),M_\Gamma}
	\end{align}
	\end{subequations}
	for all $n\geq\ell-4$, where we set for simplicity $p^{(3)}(t)=0$ for $t<0$.
\end{corollary}

\begin{proof}
	Due to the triangle inequality, it is sufficient to prove the statement for $S(E)=E^\ell$ with $\ell\geq 0$. Then we have
	\begin{align*}
		\norm[\big]{\tfrac{1}{\tau}E^\ell&\big(e_z^{n+4} - (z^{n+4}-z(t^{n+4}))\big)}_{M}
		\lesssim \tfrac{1}{\tau}\, \tsum_{k=n-\ell}^n\norm[\big]{e_z^{k+4} - z^{k+4} + z(t^{k+4})}_{M}
		\\
		&\lesssim \tfrac{\sqrt{h}}{\tau}\, \tsum_{k=\max(0,n-\ell)}^n\norm{E^4p(t^{k+4})}_{M_\Gamma}
		\lesssim \tfrac{\sqrt{h}}{\tau}\,\tsum_{k=\max(-1,n-\ell-1)}^n\norm{E^3p(t^{k+4})}_{M_\Gamma}
		\\
		&\lesssim \tau^2\sqrt{h}\,\tsum_{k=\max(-1,n-\ell-1)}^n\norm{p^{(3)}}_{L^\infty(t^{k+1},t^{k+4}),M_\Gamma}
		\lesssim \tau^2\,\norm{p^{(3)}}_{L^\infty(t^{n-\ell},t^{n+4}),M_\Gamma}
	\end{align*}
	and, analogously,
	\begin{align*}
		\norm[\big]{E^\ell\big(e_z^{n+4}& - (z^{n+4}-z(t^{n+4}))\big)}_{M+A}
		\lesssim \tsum_{k=n-\ell}^n\norm[\big]{e_z^{k+4} - z^{k+4} + z(t^{k+4})}_{A}
		\\
		&\lesssim \tfrac{1}{\sqrt{h}}\,\tsum_{k=\max(0,n-\ell)}^n\norm{E^4p(t^{k+4})}_{M_\Gamma}
		\lesssim \tfrac{1}{\tau}\tfrac{\tau}{\sqrt{h}}\,\tsum_{k=\max(-1,n-\ell-1)}^n\norm{E^3p(t^{k+4})}_{M_\Gamma}
		\\
		&\lesssim \tau^2\,\tsum_{k=\max(-1,n-\ell-1)}^n\norm{p^{(3)}}_{L^\infty(t^{k+1},t^{k+4}),M_\Gamma}
		\lesssim \tau^2\,\norm{p^{(3)}}_{L^\infty(t^{n-\ell},t^{n+4}),M_\Gamma}
	\end{align*}
	for all $n\geq 0$,
	where we applied \Cref{lem:Ek_Lp_error}
\end{proof}

We observe that, as long as the solution $z=(u,p)$ of \eqref{eq:semidisc:kinetic} is sufficiently smooth, the terms on the inhomogeneity defined in \Cref{lem:error_solutions} satisfy
\begin{subequations}\label{eq:f_error_bound}
	\begin{align}
		\norm{\ddot z(t^{n+4})-\bdf_2z(t^{n+4})}_M
		&\lesssim \tau\,\norm{z^{(4)}}_{L^1(t^{n+1},t^{n+4}),M}, \\
		\norm{f_{\wt M}(t^{n+4},e_z)}_M = \tfrac{2}{\tau^2}\norm{E^4z(t^{n+4})}_M
		&\lesssim \tau\,\norm{z^{(4)}}_{L^1(t^n,t^{n+4}),M}, \\
		\norm{f_{\wt A,\Omega}(t^{n+4},e_z)}_{A_\Omega} = \norm{E^4u(t^{n+4})}_{A_\Omega}
		&\lesssim \tau^3\,\norm{u^{(4)}}_{L^1(t^n,t^{n+4}),A_\Omega}, \\
		\norm{f_{\wt A,\Gamma}(t^{n+4},e_z)}_{M_\Gamma} = \norm{E^4z(t^{n+4})}_{M_\Gamma}
		&\lesssim \tau^3\,\norm{p^{(4)}}_{L^1(t^n,t^{n+4}),M_\Gamma},
	\end{align}
\end{subequations}
again due to \Cref{lem:Ek_Lp_error}.
We finally present a convergence result for the state of the discrete system.
\begin{theorem}\label{thm:discreteStateConvergence}
	Suppose that $f$ satisfies the Lipschitz condition \eqref{eq:Lipschitz_invNorm} and that the solution of \eqref{eq:semidisc:kinetic} satisfies $z\in W^{4,1}([0,T],\R^{N_\Omega+N_\Gamma})$ with respect to the $M+A$-norm.
	Furthermore, assume that $\tfrac{\tau}{\sqrt{h}}$ is sufficiently small and that the initial energy of the error satisfies
	\begin{equation}
	e_{H_0} 
	\coloneqq \tsum_{n=1}^3 \norm[\big]{\tfrac{1}{\tau}Ee_z^n}_M^2 + \tsum_{n=0}^3\norm{e_z^n}_A^2 \lesssim \tau^4.
	\end{equation}
	Then we have 
	\begin{equation}
		\norm[\big]{\bdf_1\hat z^{m+4}-\dot z(t^{m+4})}_M + \norm[\big]{\hat z^{m+4}-z(t^{m+4})}_{M+A} 
		\lesssim \tau^2e^{CL\tau(m+1)} \big(1+\norm{z}_{W^{4,1},M+A}\big)
	\end{equation}
	for all $m\ge 0$, where we use the norm $\norm{z}_{W^{4,1},M+A} = \sum_{j=0}^4 \norm{z^{(j)}}_{L^1,M+A}$.  
\end{theorem}

\begin{proof}
	Let $f_M,f_{\wt M},f_{\wt A}$ be defined as in \Cref{lem:error_solutions}.
	In particular,
	\begin{align*}
		\norm{f_M(t,e_z)}_M &\leq \norm{f(t,e_z+\ddf_0z(t))-f(t,z(t))}_{M^{-1}} + \norm{\ddot z(t)-\bdf_2z(t)}_M
		\\
		&\leq L\,\norm{e_z-E^4z(t)}_{M+A} + \norm{\ddot z(t)-\bdf_2z(t)}_M
		\\
		&\leq L\,\norm{e_z}_{M+A} + L\,\norm{E^4z(t)}_{M+A} + \norm{\ddot z(t)-\bdf_2z(t)}_M
	\end{align*}
	for all relevant $t,e_z$. Thus $f_M$ satisfies condition \eqref{eq:f_bounded_condition} with
	\[
		b_M(t) 
		= L\,\norm{E^4z(t)}_{M+A} + \norm{\ddot z(t)-\bdf_2z(t)}_M.
	\]
	Due to \eqref{eq:f_error_bound} and \Cref{lem:Ek_Lp_error}, we have then
	\[
	\sum_{n=0}^m\big(\tau b_M(t^{n+4}) + \tilde b^n\big)
	\lesssim \tau^2 \sum_{n=0}^m\, \norm{z^{(4)}}_{L^1(t^n,t^{n+4}),M+A}
	\lesssim \tau^2\, \norm{z^{(4)}}_{L^1(0,T),M+A}.
	\]
%	\[
%	\sum_{n=0}^mb^n
%	\lesssim \tau^2 \sum_{n=0}^m \norm{z^{(4)}}_{L^1(t^n,t^{n+4}),M+A}
%	\lesssim \tau^2\norm{z^{(4)}}_{L^1(0,T),M+A}.
%	\]
	Then, by applying \Cref{lem:aux_vs_orig_state,lem:M_to_A_bound,lem:LipschitzStability}, we obtain that
	\begin{align*}
		\norm{\bdf_1\hat e_z^{m+4}}_M + \norm{\hat e_z^{m+4}}_{M+A}
		&\lesssim e^{CL\tau(m+1)}\big(\sqrt{e_{H_0}} + \tau^2\norm{z^{(4)}}_{L^1,M+A}\big) \\
		&\lesssim \tau^2e^{CL\tau(m+1)} \big(1+\norm{z^{(4)}}_{L^1,M+A}\big).
	\end{align*}
	Thus,
	\begin{align*}
		&\norm{\bdf_1\hat z^{m+4}-\dot z(t^{m+4})}_M + \norm{\hat z^{m+4}-z(t^{m+4})}_{M+A}
		\\
		&\qquad \leq \norm{\bdf_1z(t^{m+4})-\dot z(t^{m+4})}_M + \norm{\bdf_1\hat e_z^{m+4}}_M + \norm{\hat e_z^{m+4}}_{M+A}
		\\
		&\qquad\lesssim \tau^2\norm{\ddot z}_{L^\infty,M} + \tau^2e^{CL\tau(m+1)} \big(1+\norm{z^{(4)}}_{L^1,M+A}\big)
		\\
		&\qquad\lesssim \tau^2e^{CL\tau(m+1)} \big(1+\norm{z}_{W^{4,1},M+A}\big).
	\end{align*}
	To prove the inequality for $z^{m+4}$, we observe that we analogously have
	\[
	\norm{\bdf_1e_z^{m+4}}_M + \norm{e_z^{m+4}}_{M+A} 
	\lesssim \tau^2e^{CL\tau(m+1)} \big(1+\norm{z^{(4)}}_{L^1,M+A}\big),
	\]
	which together with \Cref{cor:error_of_error} yields 
	\begin{align*}
		& \norm{\bdf_1z^{m+4}-\dot z(t^{m+4})}_M + \norm{z^{m+4}-z(t^{m+4})}_{M+A}
		\\
		&\qquad\leq \norm{\bdf_1z(t^{m+4})-\dot z(t^{m+4})}_M
		+ \norm{\bdf_1(e_z^{m+4}-z^{m+4}+z(t^{m+4}))}_M \\
		&\qquad\qquad+ \norm{\bdf_1e_z^{m+4}}_M
		+ \norm{e_z^{m+4}-z^{m+4}+z(t^{m+4})}_{M+A}
		+ \norm{e_z^{m+4}}_{M+A}
		\\
		&\qquad\lesssim \tau^2\,\norm{\ddot z}_{L^\infty,M} + \tau^2\,\norm{p^{(3)}}_{L^\infty,M_\Gamma} + \tau^2e^{CL\tau(m+1)} \big(1+\norm{z^{(4)}}_{L^1,M+A}\big)
		\\
		&\qquad\lesssim \tau^2e^{CL\tau(m+1)} \big(1+\norm{z}_{W^{4,1},M+A}\big),
	\end{align*}
	concluding the proof.
\end{proof}

When the inhomogeneity is independent of the state, it is in particular Lipschitz continuous with $L=0$ and the exponential term vanishes, leading to the following result.

\begin{corollary}\label{cor:discreteStateConvergence}
	If, additionally to the hypotheses of \Cref{thm:discreteStateConvergence}, the inhomogeneity $f$ is independent of the state, then
	\begin{equation}
		\norm{\bdf_1\hat z^{m+4}-\dot z(t^{m+4})}_M + \norm{\hat z^{m+4}-z(t^{m+4})}_{M+A} 
		\lesssim \tau^2 \big(1+\norm{z}_{W^{4,1},M+A}\big)
	\end{equation}
	holds of all $m\geq 0$.
\end{corollary}
\begin{remark}
If a second-order method such as Crank--Nicolson is used to construct the extended initial values $z^1,z^2,z^3$, starting from an initial condition of the form $u(0)=u_0$, $\dot u(0)=w_0$, one typically obtains $e_z^0=0$ and $\norm{e_z^n}_{M+A}\lesssim\tau^3$ for $n=1,2,3$, since the local truncation error of a discretization method is typically higher than its global truncation error.
The required bound on the initial energy of the error is then fulfilled, since
\[
	\sqrt{e_{H_0}} 
	\leq \sum_{n=1}^3\, \norm[\big]{\tfrac{1}{\tau}Ee_z^n}_M + \sum_{n=0}^3\, \norm{e_z^n}_A \lesssim \frac1\tau\,\sum_{n=0}^4\, \norm{e_z^n}_M + \sum_{n=0}^4\, \norm{e_z^n}_A
	\lesssim \tau^2.
\]
\end{remark}
%
%
%=============================================================================
%=========  Numerics
%=============================================================================
\section{Numerical Experiments}\label{sect:numerics}

In this section, we present numerical experiments in three different scenarios to verify our theoretical results empirically.
In all three settings the domain is the unit circle $\Omega=\set{x\in\R^2\mid x_1^2+x_2^2\leq 1}$ and the meshes $\calT_h^\Omega$ are generated by \textit{DistMesh}; see~\cite{PerS04}. For simplicity, the mesh at the boundary $\calT_h^\Gamma$ is taken to be the restriction of $\calT_h^\Omega$ to the boundary.
Moreover, we consider standard $P_1$-finite elements on both $\calT_h^\Omega$ and $\calT_h^\Gamma$.

To compute the needed initial data $u^n$, $p^n$, $n=0,\ldots,3$, for the proposed splitting scheme, we apply the Crank--Nicolson scheme with the same time step size.

%
%
%======================================================================
\subsection{Time convergence for the model problem}\label{sec:firstExperiment}

For the first experiment, we consider the model problem \eqref{eq:kineticBC} with initial condition
\begin{equation}
	u(0,x) = e^{-20((x_1-1)^2+x_2^2)}, \qquad
	\dot u(0,x) = 0,
\end{equation}
analogously to \cite[Sect.~8.1]{HipK21}. Differently from there, however, we choose a time-varying inhomogeneity of the form
\begin{equation}
	f_\Omega(t,u) = \sin(t), \qquad f_\Gamma(t,p) = \cos(t)
\end{equation}
and we simulate the system up to time $T=2.3$.

For the space discretization, we consider a single mesh with maximal width $h\approx 0.0672$.
For the time discretization, we test the splitting method \eqref{eq:multistep} applied with time step $\tau=2^{-k}$ for $k=8,\ldots,14$. The reference solution is computed by the fully implicit Crank--Nicolson scheme with step size $2^{-16}$. 

In \Cref{tab:1}, we display the error in the $L^\infty(H^1)$-norm in both the bulk and the boundary, which can be stated in the discrete setting in terms of the $M_\Omega+A_\Omega$- and $M_\Gamma+A_\Gamma$-norms.
Furthermore, we show the $\log_2$ of the ratio between the errors of pairs of consecutive chosen step sizes.
Since the consecutive step sizes are always reduced by half, this value is expected to converge towards the convergence rate of the method.
Indeed, we can see that these rates are apparently converging to 2, in accordance with \Cref{cor:discreteStateConvergence}.
Analogous results are obtained when comparing the errors in the $L^\infty(L^2)$-norm and in the energy norm, but we omit the details for the sake of brevity.

\begin{table}[ht]
	\begin{tblr}{
			colspec={c|c||c|c},
%			hlines,
			row{1-Z} = {font=\small},
			rowsep = 0.2pt,
		}
		$\norm{u^N-u(t^N)}_{M_\Omega+A_\Omega}$ & 
		rate$(u)$ &
		$\norm{p^N-p(t^N)}_{M_\Omega+A_\Omega}$ & 
		rate$(p)$ \\ \hline\hline
%		0.050305 &  & 0.013828 & \\ \hline
%		0.015252 & 1.721746 & 0.003625 & 1.931470 \\ \hline
%		0.003904 & 1.965990 & 0.000919 & 1.980172 \\ \hline
%		0.000978 & 1.996908 & 0.000230 & 1.996398 \\ \hline
%		0.000244 & 2.002086 & 0.000058 & 1.997467 \\ \hline
%		0.000061 & 2.010126 & 0.000015  & 1.982435 \\
		0.050305 &  & 0.013828 & \\ \hline
		0.015252 & 1.72 & 0.003625 & 1.93 \\ \hline
		0.003904 & 1.97 & 0.000919 & 1.98 \\ \hline
		0.000978 & 2.00 & 0.000230 & 2.00 \\ \hline
		0.000244 & 2.00 & 0.000058 & 2.00 \\ \hline
		0.000061 & 2.01 & 0.000015  & 1.98 \\
	\end{tblr}
	\caption{Discrete $L^\infty(H^1)$ errors and rates with respect to the solution of the semidiscretized system for time step sizes $\tau=2^{-k}$, $k=8,\ldots,14$. %The values in the second and fourth column indicate the $\log_2$-ratios between two consecutive errors. As expected from the convergence analysis, both rates are approximately 2.
	}
	\label{tab:1}
\end{table}
%
%
%======================================================================
\subsection{Total convergence for the model problem}
\label{sec:secondExperiment}

In the second experiment, we consider the model problem \eqref{eq:kineticBC} with given exact solution
\begin{equation}\label{eq:enforcedSolution}
	u(t,x) = \cos(2\pi t)(x_1+x_2)^2,
\end{equation}
fixing the inhomogeneity to
\begin{subequations}
\begin{align}
	f_\Omega(t,x) &= -4\cos(2\pi t)\big(1+\pi^2(x_1+x_2)^2\big), \\
%	-4\pi^2\cos(2\pi t)(x_1+x_2)^2 - 4\cos(2\pi t)(x_1+x_2)^2
	f_\Gamma(t,x) &= \cos(2\pi t)\big( 8x_1x_2 - (4-3\pi^2)(x_1+x_2)^2 \big).
%	-4\pi^2\cos(2\pi t)(x_1+x_2)^2 + 8\cos(2\pi t)x_1x_2 + \cos(2\pi t)(x_1+x_1)^2 + 2\cos(2\pi t)(x_1+x_2)^2
\end{align}
\end{subequations}
For the spatial discretization, we consider $8$ different meshes, starting from maximal width $h\approx 0.6038$ and reducing the width every time by a factor of approximately $\sqrt{2}$.
For simplicity, the semidiscretization of the inhomogeneity is obtained by first approximating $f_\Omega$ and $f_\Gamma$ by piecewise linear interpolations at the vertices of $\Omega_h$ and $\Gamma_h$, respectively.
For the temporal discretization, we apply the splitting method \eqref{eq:multistep} for different step sizes~$\tau$ and compare the result with the exact solution.

In \Cref{fig:2}, we plot the $L^\infty(H^1)$ error of the combined bulk--surface discrete state $z^n=(u^n,p^n)$ as well as the $L^\infty(L^2)$ error of its time derivative $\dot z^n=(\dot u^n,\dot p^n)=\bdf_1z^n$ for different values of $h$ and $\tau$.
The errors are computed in the fully continuous setting by means of linear interpolation and integration.
Analyzing the plots, we make the following observations:
\begin{itemize}
	\item For every mesh, the error decreases with $\tau$ up to reaching asymptotically a fixed positive value, which becomes smaller when a finer mesh is selected. This is due to the fact that, for any fixed mesh, the discrete solution will converge to the solution of the semidiscretized system \eqref{eq:semidisc:kinetic} for $\tau\to 0$. The total discretization error meets then the barrier given by the semidiscretization error, which depends on the resolution of the mesh.
	\item For large values of $\tau$ --  compared to the corresponding values of $h$ -- the method shows some instability, clearly visible in the right half of both plots. This is due to the violation of the CFL condition $\tau\lesssim\sqrt{h}$.
	\item For the values of $\tau$ and $h$ which present neither instability nor horizontal asymptotic behavior, i.e., when $\tau\lesssim\sqrt{h}$ and $h$ is sufficiently small, we observe the expected second-order convergence in $\tau$.
\end{itemize}
Further simulations, which we omit to show here, with a Lipschitz continuous inhomogeneity $f$ depending in a nonlinear fashion on $z$, verify that the convergence order is retained also in the case of semilinear equations.
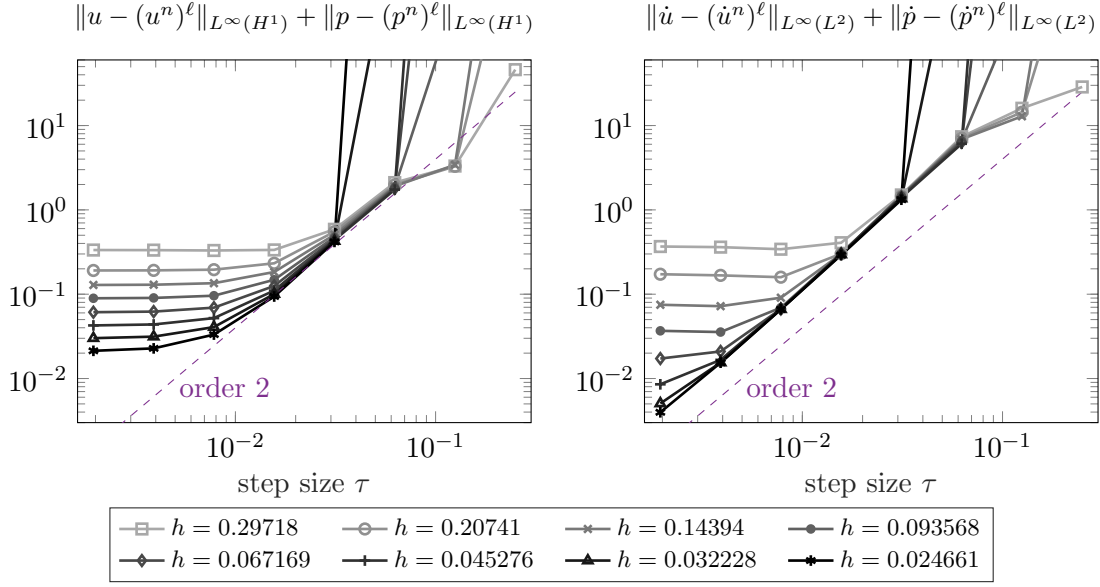
\begin{figure}%[ht]
	% This file was created by matlab2tikz.
%
\definecolor{mycolor1}{rgb}{0.56571,0.56571,0.56571}%
\definecolor{mycolor2}{rgb}{0.47143,0.47143,0.47143}%
%\definecolor{mycolor3}{rgb}{1.00000,0.00000,0.60000}%
\definecolor{mycolor3}{rgb}{0.49400,0.18400,0.55600}% violet
\begin{tikzpicture}

\begin{axis}[%
width=0.4\textwidth,
height=0.32\textwidth,
scale only axis,
xmode=log,
xmin=0.00162760416666667,
xmax=0.3,
xminorticks=true,
xlabel style={font=\color{white!15!black}},
xlabel={step size $\tau$},
ymode=log,
ymin=0.003,
ymax=60,
yminorticks=true,
ylabel style={font=\color{white!15!black}},
axis background/.style={fill=white},
title style={font=\bfseries},
title={\small $\|u-(u^n)^\ell\|_{L^\infty(H^1)}+ \|p-(p^n)^\ell\|_{L^\infty(H^1)}$},
legend style={legend cell align=left, align=left, draw=white!15!black},
legend columns=4,
legend style={at={(2.02,-0.44)}, anchor=south east, font=\footnotesize}
]
\addplot [color=white!66!black, line width=1.0pt, mark=square, mark options={solid, white!66!black}]
  table[row sep=crcr]{%
0.25	46.0014683205986\\
0.125	3.30983647721316\\
0.0625	2.10576451842732\\
0.03125	0.588340524897257\\
0.015625	0.335065882148786\\
0.0078125	0.32998928375549\\
0.00390625	0.333132993418736\\
0.001953125	0.334012651192061\\
};
\addlegendentry{$h = 0.29718$\quad}

\addplot [color=mycolor1, line width=1.0pt, mark=o, mark options={solid, mycolor1}]
  table[row sep=crcr]{%
0.25	2718.08438271384\\
0.125	3.28487089977927\\
0.0625	1.9753047508811\\
0.03125	0.550535200604696\\
0.015625	0.232615365467944\\
0.0078125	0.195393354991515\\
0.00390625	0.191663491321974\\
0.001953125	0.191166229065653\\
};
\addlegendentry{$h = 0.20741$\quad}

\addplot [color=mycolor2, line width=1.0pt, mark=x, mark options={solid, mycolor2}]
  table[row sep=crcr]{%
0.25	152771.618432347\\
0.125	3.42720467779469\\
0.0625	1.90519518480657\\
0.03125	0.507800710472664\\
0.015625	0.18332323987688\\
0.0078125	0.135039024930562\\
0.00390625	0.129545411513913\\
0.001953125	0.128652052272492\\
};
\addlegendentry{$h = 0.14394$\quad}

\addplot [color=black!20!mycolor2, line width=1.0pt, mark size=1.3pt, mark=*, mark options={solid, black!20!mycolor2}]
  table[row sep=crcr]{%
0.25	8467719.66591026\\
0.125	287.505937235062\\
0.0625	1.85583421112984\\
0.03125	0.475171650418256\\
0.015625	0.148275263292337\\
0.0078125	0.0959492397832061\\
0.00390625	0.0903166894938563\\
0.001953125	0.0894612823074638\\
};
\addlegendentry{$h = 0.093568$}

\addplot [color=black!50!mycolor1, line width=1.0pt, mark=diamond, mark options={solid, black!50!mycolor1}]
  table[row sep=crcr]{%
0.25	300765989.972614\\
0.125	4470849.22194939\\
0.0625	1.79832953047221\\
0.03125	0.453103182155189\\
0.015625	0.125240574236399\\
0.0078125	0.0692413156984859\\
0.00390625	0.0621804583867801\\
0.001953125	0.0611611082389716\\
};
\addlegendentry{$h = 0.067169$\quad}

\addplot [color=black!60!mycolor2, line width=1.0pt, mark=+, mark options={solid, black!60!mycolor2}]
  table[row sep=crcr]{%
0.25	8403933631.22051\\
0.125	14094922576.0994\\
0.0625	1.73816096077722\\
0.03125	0.438579510231993\\
0.015625	0.110402823865203\\
0.0078125	0.0520501766753891\\
0.00390625	0.0437718621480082\\
0.001953125	0.0426198234764422\\
};
\addlegendentry{$h = 0.045276$\quad}

\addplot [color=black!80!mycolor2, line width=1.0pt, mark=triangle, mark options={solid, black!80!mycolor2}]
  table[row sep=crcr]{%
0.25	279725734357.444\\
0.125	36563886957325.2\\
0.0625	2258.36142123016\\
0.03125	0.42723711002698\\
0.015625	0.100561109669728\\
0.0078125	0.0407328502713664\\
0.00390625	0.0313822913885991\\
0.001953125	0.0301003419752768\\
};
\addlegendentry{$h = 0.032228$\quad}

\addplot [color=black, line width=1.0pt, mark=asterisk, mark options={solid, black}]
  table[row sep=crcr]{%
0.25	8035831427559.98\\
0.125	4.46962171190652e+16\\
0.0625	27765066257.7548\\
0.03125	0.416869996725518\\
0.015625	0.093886282960465\\
0.0078125	0.0331981397442592\\
0.00390625	0.0227750344751071\\
0.001953125	0.0213058654153312\\
};
\addlegendentry{$h = 0.024661$}

\addplot [color=mycolor3, dashed, forget plot]
  table[row sep=crcr]{%
0.25	25\\
0.125	6.25\\
0.0625	1.5625\\
0.03125	0.390625\\
0.015625	0.09765625\\
0.0078125	0.0244140625\\
0.00390625	0.006103515625\\
0.001953125	0.00152587890625\\
};
\node[left, align=right, inner sep=0, font=\color{mycolor3}]
at (axis cs:0.015,0.008) {order 2};
%\node[right, align=left, inner sep=0, font=\color{mycolor3}]
%at (axis cs:0.25,15) {2};
\end{axis}

\begin{axis}[%
width=0.4\textwidth,
height=0.32\textwidth,
at={(0.5\textwidth,0\textwidth)},
scale only axis,
xmode=log,
xmin=0.00162760416666667,
xmax=0.3,
xminorticks=true,
xlabel style={font=\color{white!15!black}},
xlabel={step size $\tau$},
ymode=log,
ymin=0.003,
ymax=60,
yminorticks=true,
ylabel style={font=\color{white!15!black}},
axis background/.style={fill=white},
title style={font=\bfseries},
title={\small $\|\dot u-(\dot u^n)^\ell\|_{L^\infty(L^2)}+ \|\dot p-(\dot p^n)^\ell\|_{L^\infty(L^2)}$},
legend style={legend cell align=left, align=left, draw=white!15!black},
legend style={at={(0.95,0.05)}, anchor=south east, font=\footnotesize}
]
\addplot [color=white!66!black, line width=1.0pt, mark=square, mark options={solid, white!66!black}]
  table[row sep=crcr]{%
0.25	28.6955568460893\\
0.125	16.0374079124053\\
0.0625	7.36710209498332\\
0.03125	1.50273803585315\\
0.015625	0.407837054462483\\
0.0078125	0.341488145859703\\
0.00390625	0.361579843568436\\
0.001953125	0.36776184581524\\
};
%\addlegendentry{$h = 0.29718$}

\addplot [color=mycolor1, line width=1.0pt, mark=o, mark options={solid, mycolor1}]
  table[row sep=crcr]{%
0.25	3102.79518602558\\
0.125	14.3924337587651\\
0.0625	7.07155330233559\\
0.03125	1.45316933405722\\
0.015625	0.301915355353481\\
0.0078125	0.159074168985841\\
0.00390625	0.167230352179324\\
0.001953125	0.172538826036358\\
};
%\addlegendentry{$h = 0.20741$}

\addplot [color=mycolor2, line width=1.0pt, mark=x, mark options={solid, mycolor2}]
  table[row sep=crcr]{%
0.25	239088.243433468\\
0.125	12.8833131516903\\
0.0625	6.88315390688775\\
0.03125	1.47695935993368\\
0.015625	0.29943388113926\\
0.0078125	0.0908455972271824\\
0.00390625	0.0721103205365191\\
0.001953125	0.0749244782067184\\
};
%\addlegendentry{$h = 0.14394$}

\addplot [color=black!20!mycolor2, line width=1.0pt, mark size=1.3pt, mark=*, mark options={solid, black!20!mycolor2}]
  table[row sep=crcr]{%
0.25	16904748.706565\\
0.125	333.945707534673\\
0.0625	6.62562043434867\\
0.03125	1.46202504014293\\
0.015625	0.301511332206946\\
0.0078125	0.069233746648564\\
0.00390625	0.0355728313446702\\
0.001953125	0.0367847614783188\\
};
%\addlegendentry{$h = 0.093568$}

\addplot [color=black!50!mycolor1, line width=1.0pt, mark=diamond, mark options={solid, black!50!mycolor1}]
  table[row sep=crcr]{%
0.25	696334601.173545\\
0.125	7023831.99288149\\
0.0625	6.36636687279894\\
0.03125	1.4348576800986\\
0.015625	0.302873560332193\\
0.0078125	0.0667319305005844\\
0.00390625	0.0209522511554872\\
0.001953125	0.0172507130115537\\
};
%\addlegendentry{$h = 0.067169$}

\addplot [color=black!60!mycolor2, line width=1.0pt, mark=+, mark options={solid, black!60!mycolor2}]
  table[row sep=crcr]{%
0.25	22005812536.3831\\
0.125	28733923991.4499\\
0.0625	6.08139186839235\\
0.03125	1.40152204285493\\
0.015625	0.30037686738599\\
0.0078125	0.0666694948649458\\
0.00390625	0.0166517050781035\\
0.001953125	0.00853884855130118\\
};
%\addlegendentry{$h = 0.045276$}

\addplot [color=black!80!mycolor2, line width=1.0pt, mark=triangle, mark options={solid, black!80!mycolor2}]
  table[row sep=crcr]{%
0.25	802942461299.653\\
0.125	90180173118959.1\\
0.0625	3313.23340235377\\
0.03125	1.36185992479122\\
0.015625	0.295800087535902\\
0.0078125	0.066317078162583\\
0.00390625	0.01546025395404\\
0.001953125	0.00506029333715312\\
};
%\addlegendentry{$h = 0.032228$}

\addplot [color=black, line width=1.0pt, mark=asterisk, mark options={solid, black}]
  table[row sep=crcr]{%
0.25	25383710961266.3\\
0.125	1.28969917054157e+17\\
0.0625	54499414232.9543\\
0.03125	1.31954818232714\\
0.015625	0.291132970067091\\
0.0078125	0.0660135498936948\\
0.00390625	0.0154035958375186\\
0.001953125	0.00397800796256698\\
};
%\addlegendentry{$h = 0.024661$}

\addplot [color=mycolor3, dashed, forget plot]
  table[row sep=crcr]{%
0.25	25\\
0.125	6.25\\
0.0625	1.5625\\
0.03125	0.390625\\
0.015625	0.09765625\\
0.0078125	0.0244140625\\
0.00390625	0.006103515625\\
0.001953125	0.00152587890625\\
};
\node[left, align=right, inner sep=0, font=\color{mycolor3}]
at (axis cs:0.015,0.008) {order 2};
%\node[right, align=left, inner sep=0, font=\color{mycolor3}]
%at (axis cs:0.25,15) {2};
\end{axis}

\end{tikzpicture}%
	\caption{Errors with respect to the exact solution of the model problem~\eqref{eq:kineticBC} for different values of $h$ and $\tau$. The curves are compared with a dashed line representing second-order convergence.}
	\label{fig:2}
\end{figure}
%
%
%======================================================================
\subsection{Strong damping}

In this final experiment, we consider the extended model \eqref{eq:extPDE} with parameters $\beta_\Omega=\beta_\Gamma=\kappa_\Gamma=1$, $\delta_\Omega=0.1$, and $\delta_\Gamma=0.2$, while all the others are set to zero.
This corresponds to the wave equation with strongly damped dynamic boundary condition studied in~\cite[Sect.~8.3]{HipK21}.
Differently from their case, however, we enforce the same exact solution \eqref{eq:enforcedSolution} as in \Cref{sec:secondExperiment}, leading to the inhomogeneities
\begin{subequations}
	\begin{align}
		f_\Omega(t,x) &= -4\cos(2\pi t)\big( 1 + \pi^2(x_1+x_2)^2 \big) + 0.8\pi\sin(2\pi t), \\
		f_\Gamma(t,x) &=
		\begin{multlined}[t]
			\cos(2\pi t)\big( 8x_1x_2 - (4-3\pi^2)(x_1+x_2)^2 \big) \\ - 4\pi\sin(2\pi t)\big( 0.8x_1x_2 + 0.1(x_1+x_2)^2 \big).
		\end{multlined}
	\end{align}
\end{subequations}
We repeat then the same numerical experiment as in the previous subsection, using the same meshes but finer time steps.

In \Cref{fig:3}, we observe similar phenomena as in \Cref{fig:2}. Due to the stricter CFL condition $\tau\lesssim h$ holding for the strongly damped wave equation (as discussed in \Cref{sec:alternativeSplitting}), the plots exhibit more instability. In particular, for the $L^\infty(H^1)$-norm, the second-order convergence in $\tau$ is not visible. 
It is, however, clearly visible if the $L^\infty(L^2)$-error of $z$ is considered instead. 
Moreover, simulations on the same model for a fixed mesh, when compared to the solution of the semidiscretized system analogously as in \Cref{sec:firstExperiment}, still display second-order convergence, as expected. 
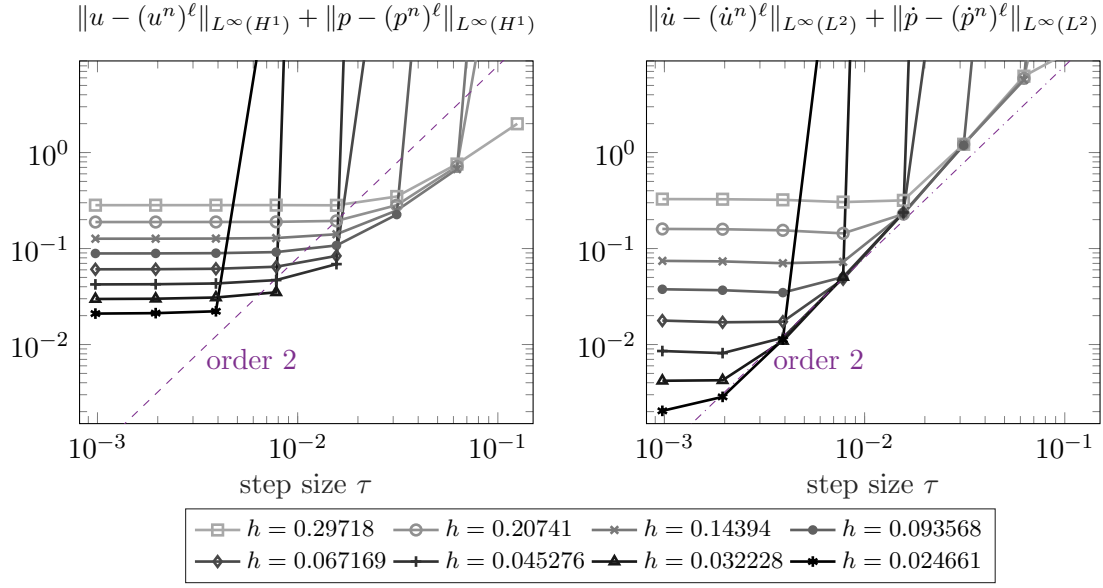
\begin{figure}%[ht]
	% This file was created by matlab2tikz.
%
\definecolor{mycolor1}{rgb}{0.56571,0.56571,0.56571}%
\definecolor{mycolor2}{rgb}{0.47143,0.47143,0.47143}%
%\definecolor{mycolor3}{rgb}{1.00000,0.00000,0.60000}%
\definecolor{mycolor3}{rgb}{0.49400,0.18400,0.55600}% violet
\begin{tikzpicture}

\begin{axis}[%
width=0.4\textwidth,
height=0.32\textwidth,
scale only axis,
xmode=log,
xmin=0.000813802083333333,
xmax=0.15,
xminorticks=true,
xlabel style={font=\color{white!15!black}},
xlabel={step size $\tau$},
ymode=log,
ymin=0.0015,
ymax=9,
yminorticks=true,
ylabel style={font=\color{white!15!black}},
axis background/.style={fill=white},
title style={font=\bfseries},
title={\small $\|u-(u^n)^\ell\|_{L^\infty(H^1)}+ \|p-(p^n)^\ell\|_{L^\infty(H^1)}$},
legend style={legend cell align=left, align=left, draw=white!15!black},
legend columns=4,
legend style={at={(2.02,-0.44)}, anchor=south east, font=\footnotesize}
]
\addplot [color=white!66!black, line width=1.0pt, mark=square, mark options={solid, white!66!black}]
  table[row sep=crcr]{%
0.125	1.99387153842582\\
0.0625	0.759382229069193\\
0.03125	0.348235460067506\\
0.015625	0.282027333501994\\
0.0078125	0.283345232057774\\
0.00390625	0.283179289835194\\
0.001953125	0.283003860368102\\
0.0009765625	0.282939247310227\\
};
\addlegendentry{$h = 0.29718$}

\addplot [color=mycolor1, line width=1.0pt, mark=o, mark options={solid, mycolor1}]
  table[row sep=crcr]{%
0.125	2533.16561371551\\
0.0625	0.717504776737073\\
0.03125	0.282945250764987\\
0.015625	0.194055471953235\\
0.0078125	0.189697969006701\\
0.00390625	0.189101861342637\\
0.001953125	0.188863848525282\\
0.0009765625	0.188788080152595\\
};
\addlegendentry{$h = 0.20741$}

\addplot [color=mycolor2, line width=1.0pt, mark=x, mark options={solid, mycolor2}]
  table[row sep=crcr]{%
0.125	12209033.5166262\\
0.0625	0.672344136185911\\
0.03125	0.249229200105254\\
0.015625	0.140828033732496\\
0.0078125	0.128642116121058\\
0.00390625	0.127054462778029\\
0.001953125	0.126674211878097\\
0.0009765625	0.126574833235989\\
};
\addlegendentry{$h = 0.14394$}

\addplot [color=black!20!mycolor2, line width=1.0pt, mark size=1.3pt, mark=*, mark options={solid, black!20!mycolor2}]
  table[row sep=crcr]{%
0.125	58672488068.3861\\
0.0625	1269115.96066099\\
0.03125	0.225824083682722\\
0.015625	0.10766571911896\\
0.0078125	0.0917930007290756\\
0.00390625	0.089459761208142\\
0.001953125	0.0889686297082938\\
0.0009765625	0.0888554049997348\\
};
\addlegendentry{$h = 0.093568$}

\addplot [color=black!50!mycolor1, line width=1.0pt, mark=diamond, mark options={solid, black!50!mycolor1}]
  table[row sep=crcr]{%
0.125	120684422122856\\
0.0625	15933399633269.1\\
0.03125	2451.53295431138\\
0.015625	0.0837843790172323\\
0.0078125	0.064612282154908\\
0.00390625	0.0615352881314106\\
0.001953125	0.0608942549239558\\
0.0009765625	0.0607602671546958\\
};
\addlegendentry{$h = 0.067169$}

\addplot [color=black!60!mycolor2, line width=1.0pt, mark=+, mark options={solid, black!60!mycolor2}]
  table[row sep=crcr]{%
0.125	1.55252948586208e+17\\
0.0625	4.26227939921104e+19\\
0.03125	5.6604437326934e+16\\
0.015625	0.0687874695699415\\
0.0078125	0.0468801946085359\\
0.00390625	0.04325159528139\\
0.001953125	0.0424575622068905\\
0.0009765625	0.0422948104197276\\
};
\addlegendentry{$h = 0.045276$}

\addplot [color=black!80!mycolor2, line width=1.0pt, mark=triangle, mark options={solid, black!80!mycolor2}]
  table[row sep=crcr]{%
0.125	2.78178227121638e+20\\
0.0625	1.29995605241377e+26\\
0.03125	1.21403208237612e+30\\
0.015625	4.22765847440849e+17\\
0.0078125	0.0349771592766291\\
0.00390625	0.0308930184457724\\
0.001953125	0.0299995442208531\\
0.0009765625	0.0298048602770797\\
};
\addlegendentry{$h = 0.032228$}

\addplot [color=black, line width=1.0pt, mark=asterisk, mark options={solid, black}]
  table[row sep=crcr]{%
0.125	3.80395586392145e+23\\
0.0625	1.13738841180575e+32\\
0.03125	7.49436691985207e+41\\
0.015625	1.53324132099413e+41\\
0.0078125	174.862130410213\\
0.00390625	0.0221942766687191\\
0.001953125	0.0212264243113757\\
0.0009765625	0.0210011041736329\\
};
\addlegendentry{$h = 0.024661$}

\addplot [color=mycolor3, dashed, forget plot]
  table[row sep=crcr]{%
0.125	12.5\\
0.0625	3.125\\
0.03125	0.78125\\
0.015625	0.1953125\\
0.0078125	0.048828125\\
0.00390625	0.01220703125\\
0.001953125	0.0030517578125\\
0.0009765625	0.000762939453125\\
};
\node[left, align=right, inner sep=0, font=\color{mycolor3}]
at (axis cs:0.01,0.007) {order 2};
\end{axis}

\begin{axis}[%
width=0.4\textwidth,
height=0.32\textwidth,
at={(0.5\textwidth,0\textwidth)},
scale only axis,
xmode=log,
xmin=0.000813802083333333,
xmax=0.15,
xminorticks=true,
xlabel style={font=\color{white!15!black}},
xlabel={step size $\tau$},
ymode=log,
ymin=0.0015,
ymax=9,
yminorticks=true,
ylabel style={font=\color{white!15!black}},
axis background/.style={fill=white},
title style={font=\bfseries},
title={\small $\|\dot u-(\dot u^n)^\ell\|_{L^\infty(L^2)}+ \|\dot p-(\dot p^n)^\ell\|_{L^\infty(L^2)}$},
legend style={legend cell align=left, align=left, draw=white!15!black},
legend style={at={(0.95,0.05)}, anchor=south east, font=\footnotesize}
]
\addplot [color=white!66!black, line width=1.0pt, mark=square, mark options={solid, white!66!black}]
  table[row sep=crcr]{%
0.125	13.7123476257188\\
0.0625	6.24579313732881\\
0.03125	1.21967245127124\\
0.015625	0.317500053184976\\
0.0078125	0.304729460822813\\
0.00390625	0.321616322521093\\
0.001953125	0.326244465984856\\
0.0009765625	0.327356769562353\\
};
%\addlegendentry{$h = 0.29718$}

\addplot [color=mycolor1, line width=1.0pt, mark=o, mark options={solid, mycolor1}]
  table[row sep=crcr]{%
0.125	4302.46716328698\\
0.0625	5.92888169823011\\
0.03125	1.21963076088553\\
0.015625	0.228434902263195\\
0.0078125	0.143889360608452\\
0.00390625	0.154772912224226\\
0.001953125	0.158941595645431\\
0.0009765625	0.159990731620793\\
};
%\addlegendentry{$h = 0.20741$}

\addplot [color=mycolor2, line width=1.0pt, mark=x, mark options={solid, mycolor2}]
  table[row sep=crcr]{%
0.125	27351067.0110269\\
0.0625	5.70494038094131\\
0.03125	1.21533897451319\\
0.015625	0.23018650095115\\
0.0078125	0.0726821989929774\\
0.00390625	0.0703305582283983\\
0.001953125	0.073456965753248\\
0.0009765625	0.0743735974808927\\
};
%\addlegendentry{$h = 0.14394$}

\addplot [color=black!20!mycolor2, line width=1.0pt, mark size=1.3pt, mark=*, mark options={solid, black!20!mycolor2}]
  table[row sep=crcr]{%
0.125	168973369025.338\\
0.0625	3540430.1519551\\
0.03125	1.188092845754\\
0.015625	0.234890638841886\\
0.0078125	0.0504990251628197\\
0.00390625	0.034722074575103\\
0.001953125	0.036794374146334\\
0.0009765625	0.0376388626688615\\
};
%\addlegendentry{$h = 0.093568$}

\addplot [color=black!50!mycolor1, line width=1.0pt, mark=diamond, mark options={solid, black!50!mycolor1}]
  table[row sep=crcr]{%
0.125	413826742567648\\
0.0625	55588848559872.3\\
0.03125	8357.77005808154\\
0.015625	0.237499652010902\\
0.0078125	0.0478966340736215\\
0.00390625	0.0172881212457833\\
0.001953125	0.017050346186561\\
0.0009765625	0.0177637619058008\\
};
%\addlegendentry{$h = 0.067169$}

\addplot [color=black!60!mycolor2, line width=1.0pt, mark=+, mark options={solid, black!60!mycolor2}]
  table[row sep=crcr]{%
0.125	6.23277750229572e+17\\
0.0625	1.75228981063693e+20\\
0.03125	2.49536973473469e+17\\
0.015625	0.237329837617399\\
0.0078125	0.0495809838182056\\
0.00390625	0.0117440594778258\\
0.001953125	0.00813712440755406\\
0.0009765625	0.00855643944725433\\
};
%\addlegendentry{$h = 0.045276$}

\addplot [color=black!80!mycolor2, line width=1.0pt, mark=triangle, mark options={solid, black!80!mycolor2}]
  table[row sep=crcr]{%
0.125	1.26274071624917e+21\\
0.0625	6.21973708341087e+26\\
0.03125	6.46877350335312e+30\\
0.015625	2.5744445091027e+18\\
0.0078125	0.0505296702449439\\
0.00390625	0.0108935090995114\\
0.001953125	0.00424481268995336\\
0.0009765625	0.0041921634962945\\
};
%\addlegendentry{$h = 0.032228$}

\addplot [color=black, line width=1.0pt, mark=asterisk, mark options={solid, black}]
  table[row sep=crcr]{%
0.125	1.98313775485994e+24\\
0.0625	6.09633375520594e+32\\
0.03125	4.68673830873914e+42\\
0.015625	1.11779684075818e+42\\
0.0078125	1496.98544333247\\
0.00390625	0.0112566225498899\\
0.001953125	0.00284880200138384\\
0.0009765625	0.00203737961507957\\
};
%\addlegendentry{$h = 0.024661$}

\addplot [color=mycolor3, dashdotted, forget plot]
  table[row sep=crcr]{%
0.125	12.5\\
0.0625	3.125\\
0.03125	0.78125\\
0.015625	0.1953125\\
0.0078125	0.048828125\\
0.00390625	0.01220703125\\
0.001953125	0.0030517578125\\
0.0009765625	0.000762939453125\\
};
\node[left, align=right, inner sep=0, font=\color{mycolor3}]
at (axis cs:0.01,0.007) {order 2};
\end{axis}
\end{tikzpicture}%
	\caption{Errors with respect to the exact solution of the wave equation with strongly damped dynamic boundary conditions. The curves are compared with a dashed line representing second-order convergence. 
	}	
	\label{fig:3}
\end{figure}
%
%
%=============================================================================
%=========  Conclusion
%=============================================================================
\section{Conclusion}\label{sect:conclusion}
In this paper, we have introduced and analyzed a fully discrete bulk--surface splitting scheme of second order for wave systems with kinetic boundary conditions. The resulting multistep scheme is based on different delay formulae and can be interpreted as a perturbation of the classical $\bdf$-2 method. As verified by numerical experiments, the scheme is stable under the CFL condition $\tau \lesssim \sqrt{h}$ (or $\tau\lesssim h$ in the case of strong damping), which is less restrictive than the typical CFL condition for wave systems. 
% the differential--algebraic structure of the system equations.  
%
%
%======================================================================
%=========  Bibs
%======================================================================
\bibliographystyle{alpha}
\bibliography{bib_dynBC}

@article{EllR13,
  author = {Elliott, C. M. and Ranner, T.},
  title = {Finite element analysis for a coupled bulk--surface partial differential equation},
  journal = {IMA J. Numer. Anal.},  
  fjournal = {IMA Journal of Numerical Analysis},
  volume = {33},
  number = {2},
  pages = {377--402},
  year = {2013},
  doi = {10.1093/imanum/drs022},
  URL = {http://dx.doi.org/10.1093/imanum/drs022},
}

@article{Gal07,
    AUTHOR = {Gal, C. G.},
     TITLE = {Global well-posedness for the non-isothermal {C}ahn-{H}illiard equation with dynamic boundary conditions},
   JOURNAL = {Adv. Differential Equ.},
  FJOURNAL = {Advances in Differential Equations},
    VOLUME = {12},
      YEAR = {2007},
    NUMBER = {11},
     PAGES = {1241--1274},
      ISSN = {1079-9389},
}

@article{GarK20,
  author = {Garcke, H. and Knopf, P.},
  title = {Weak Solutions of the {C}ahn--{H}illiard System with Dynamic Boundary Conditions: {A} Gradient Flow Approach},
  journal = {SIAM J. Math. Anal.},
  fjournal = {SIAM Journal on Mathematical Analysis},
  volume = {52},
  number = {1},
  pages = {340--369},
  year = {2020},
  doi = {10.1137/19M1258840},
  URL = {https://doi.org/10.1137/19M1258840}, 
}

@book{GilT01,
    AUTHOR = {Gilbarg, D. and Trudinger, N. S.},
     TITLE = {Elliptic Partial Differential Equations of Second Order},
 PUBLISHER = {Springer-Verlag},
   ADDRESS = {Berlin},
      YEAR = {2001},
     PAGES = {xiv+517},
      ISBN = {3-540-41160-7},
}

@article{Gol06,
    AUTHOR = {Goldstein, G. R.},
     TITLE = {Derivation and physical interpretation of general boundary conditions},
   JOURNAL = {Adv. Differential Equ.},
  FJOURNAL = {Advances in Differential Equations},
    VOLUME = {11},
      YEAR = {2006},
    NUMBER = {4},
     PAGES = {457--480},
      ISSN = {1079-9389},
}

@article{GolMS11,
	title = {A {C}ahn–{H}illiard model in a domain with non-permeable walls},
	journal = {Physica D},
	fjournal = {Physica D: Nonlinear Phenomena},
	volume = {240},
	number = {8},
	pages = {754--766},
	year = {2011},
	author = {Goldstein, G. R. and Miranville, A. and Schimperna, G.}
}

@article{GraL14,
    AUTHOR = {Graber, P. J. and Lasiecka, I.},
     TITLE = {Analyticity and {G}evrey class regularity for a strongly damped wave equation with hyperbolic dynamic boundary conditions},
   JOURNAL = {Semigroup Forum},
    VOLUME = {88},
      YEAR = {2014},
    NUMBER = {2},
     PAGES = {333--365},
      ISSN = {0037-1912},
       DOI = {10.1007/s00233-013-9534-3},
       URL = {https://doi.org/10.1007/s00233-013-9534-3},
}

@book{HaiW96,
  address = {Berlin},
  author = {Hairer, E. and Wanner, G.},
  edition = {Second},
  publisher = {Springer-Verlag},
  title = {Solving Ordinary Differential Equations {II}: Stiff and Differential-Algebraic Problems},
  year = {1996}
}

@phdthesis{Hip17,
   author  = {Hipp, D.},
   title   = {A unified error analysis for spatial discretizations of wave-type equations with applications to dynamic boundary conditions},
   school  = {Karlsruher Institut f\"ur Technologie (KIT)},
   year    = {2017},
   type    = {Ph{D} thesis},
   url     = {https://publikationen.bibliothek.kit.edu/1000070952},
}

@article{KomZ90,
    AUTHOR = {Komornik, V. and Zuazua, E.},
     TITLE = {A direct method for the boundary stabilization of the wave equation},
   JOURNAL = {J. Math. Pures Appl.},
    VOLUME = {69},
      YEAR = {1990},
    NUMBER = {1},
     PAGES = {33--54},
      ISSN = {0021-7824},
}

@article{KovL17,
  author = {Kov{\'a}cs, B. and Lubich, C.},  
  title = {Numerical analysis of parabolic problems with dynamic boundary conditions},
  journal = {IMA J. Numer. Anal.},
  fjournal = {IMA Journal of Numerical Analysis},
  volume = {37},
  number = {1},
  pages = {1--39},
  year = {2017},
  doi = {10.1093/imanum/drw015},
  URL = {http://dx.doi.org/10.1093/imanum/drw015},
  eprint = {/oup/backfile/content_public/journal/imajna/37/1/10.1093_imanum_drw015/3/drw015.pdf}
}

@book{LamMT13,
  author = {Lamour, R. and M{\"a}rz, R. and Tischendorf, C.},
  pages = {xxviii+649},
  publisher = {Springer-Verlag},
  address = {Berlin, Heidelberg},
  title = {Differential-Algebraic Equations: A Projector Based Analysis},
  year = {2013},
  doi = {10.1007/978-3-642-27555-5},
  isbn = {978-3-642-27554-8; 978-3-642-27555-5},
  url = {http://dx.doi.org/10.1007/978-3-642-27555-5},
}

@incollection{Vit13,
    AUTHOR = {Vitillaro, E.},
     TITLE = {Strong solutions for the wave equation with a kinetic boundary condition},
 BOOKTITLE = {Recent trends in nonlinear partial differential equations {I}. {E}volution problems},
    SERIES = {Contemp. Math.},
    VOLUME = {594},
     PAGES = {295--307},
 PUBLISHER = {Amer. Math. Soc., Providence, RI},
      YEAR = {2013},
       DOI = {10.1090/conm/594/11793},
       URL = {https://doi.org/10.1090/conm/594/11793},
}

@article{Vit15,
  author     = {Vitillaro, E.},
  title      = {On the the wave equation with hyperbolic dynamical boundary conditions, interior and boundary damping and source},
  year       = {2017},
  journal    = {Arch. Ration. Mech. An.},  
  fjournal   = {Archive for Rational Mechanics and Analysis},
  volume     = {223},
  number     = {3},
  pages      = {1183--1237},
  url        = {https://doi.org/10.1007/s00205-016-1055-2}, 
  doi        = {10.1007/s00205-016-1055-2}
}

@article{VraS13a,
    AUTHOR = {Vr\'{a}bel', V. and Slodi\v{c}ka, M.},
     TITLE = {Nonlinear parabolic equation with a dynamical boundary condition of diffusive type},
   JOURNAL = {Appl. Math. Comput.},
  FJOURNAL = {Applied Mathematics and Computation},
    VOLUME = {222},
      YEAR = {2013},
     PAGES = {372--380},
      ISSN = {0096-3003},
       DOI = {10.1016/j.amc.2013.07.057},
       URL = {https://doi.org/10.1016/j.amc.2013.07.057},
}

@article{HocL20,
  author = {Hochbruck, M. and Leibold, J.},
  title = {Finite element discretization of semilinear acoustic wave equations with kinetic boundary conditions},
  fjournal = {ETNA - Electronic Transactions on Numerical Analysis},
  journal = {Electron. T. Numer. Ana.},
  year = {2020},
  volume = {53},
  number = {},
  pages = {522--540},
  doi = {https://doi.org/10.1553/etna_vol53s522},
}

@article{HocL21,
  author = {Hochbruck, M. and Leibold, J.},
  title = {An implicit--explicit time discretization scheme for second-order semilinear wave equations with application to dynamic boundary conditions},
  journal = {Numer. Math.},
  year = {2021},
  volume = {147},
  number = {4},
  pages = {869--899},
  doi = {https://doi.org/10.1007/s00211-021-01184-w},
}

@article{HipK21,
    author = {Hipp, D. and Kov{\'a}cs, B.},
    title = {Finite element error analysis of wave equations with dynamic boundary conditions: {$L^2$} estimates},
    journal = {IMA J. Numer. Anal.},
    volume = {41},
    number = {1},
    pages = {638--728},
    year = {2021},
    issn = {0272-4979},
    doi = {10.1093/imanum/drz073},
    url = {https://doi.org/10.1093/imanum/drz073},
}

@article{AltKZ22,
  author = {Altmann, R. and Kov{\'a}cs, B. and Zimmer, C.},  
  title = {Bulk--surface {L}ie splitting for parabolic problems with dynamic boundary conditions},
  journal = {IMA J. Numer. Anal.},  
  fjournal = {IMA Journal of Numerical Analysis},
  volume = {43},
  number = {2},
  pages = {950--975},
  year = {2022},
  issn = {0272-4979},
  doi = {10.1093/imanum/drac002},
  url = {https://doi.org/10.1093/imanum/drac002},
  eprint = {https://academic.oup.com/imajna/article-pdf/43/2/950/49629962/drac002.pdf},
}

@article{Alt23,
  author     = {Altmann, R.},
  title      = {Splitting schemes for the semi-linear wave equation with dynamic boundary conditions},
  journal = {Comput. Math. Appl.},
  fjournal = {Computers and Mathematics with Applications},
  volume = {151},
  pages = {12--20},
  year = {2023},
}

@article{AltV21,
  author = {Altmann, R. and Verf\"urth, B.},
  title = {A multiscale method for heterogeneous bulk-surface coupling},
  journal = {Multiscale Model. Simul.},
  volume = {19},
  number = {1},
  pages = {374--400},
  year = {2021}
}

@article{PerS04,
  author = {Persson, P.-O. and Strang, G.},
  title = {A Simple Mesh Generator in {MATLAB}},
  volume = {46},
  journal = {SIAM Rev.},
  year = {2004},
  pages = {329--345},
}

@article{HipHS18,
    author = {Hipp, D. and Hochbruck, M. and Stohrer, C.},
    title = {Unified error analysis for nonconforming space discretizations of wave-type equations},
    journal = {IMA J. Numer. Anal.},  
    fjournal = {IMA Journal of Numerical Analysis},
    volume = {39},
    number = {3},
    pages = {1206--1245},
    year = {2018},
    issn = {0272-4979},
    doi = {10.1093/imanum/dry036},
    url = {https://doi.org/10.1093/imanum/dry036},
    eprint = {https://academic.oup.com/imajna/article-pdf/39/3/1206/28955384/dry036.pdf},
}

@incollection {Dzi88,
	AUTHOR = {Dziuk, G.},
	TITLE = {Finite elements for the {B}eltrami operator on arbitrary surfaces},
	BOOKTITLE = {Partial differential equations and calculus of variations},
	PAGES = {142--155},
	PUBLISHER = {Springer},
	address = {Berlin},
	YEAR = {1988},
	DOI = {10.1007/BFb0082865},
	URL = {https://doi.org/10.1007/BFb0082865},
}

@article{AltZ24,
    author = {Altmann, R. and Zimmer, C.},
    title = {A second-order bulk--surface splitting for parabolic problems with dynamic boundary conditions},
    journal = {IMA J. Numer. Anal.},
    fjournal = {IMA Journal of Numerical Analysis},
    volume = {44},
    number = {4},
    pages = {2370--2393},
    year = {2024},
    issn = {0272-4979},
    doi = {10.1093/imanum/drad062},
    url = {https://doi.org/10.1093/imanum/drad062},
    eprint = {https://academic.oup.com/imajna/article-pdf/44/4/2370/58633662/drad062.pdf},
}

@phdthesis{Zim21phd,
	title={Temporal discretization of constrained partial differential equations},
	author={Zimmer, C.},
	year={2021},
	school={Technische Universität Berlin}
}

@book{ErnG21,
	author={Ern, A. and Guermond, J.-L.},
	title={Finite Elements I: Approximation and Interpolation}, 
	ISBN={9783030563417},
	ISSN={2196-9949},
	url={http://dx.doi.org/10.1007/978-3-030-56341-7}, 
	DOI={10.1007/978-3-030-56341-7},
	journal={Texts in Applied Mathematics},
	publisher={Springer International Publishing},
	address={Cham},
	year={2021}
}

@misc{Emm99ppt,
	title={Discrete versions of {G}ronwall's lemma and their application to the numerical analysis of parabolic problems},
	author={Emmrich, E.},
	year={1999},
	howpublished={Preprint No. 637, TU Berlin},
	institution={Technische Universität Berlin}
}

@inproceedings{VilLZM06,
	title = {Boundary control for a class of dissipative differential operators including diffusion systems},
	author = {Villegas, J. A. and Le Gorrec, Y. and Zwart, H. J. and Maschke, B. M.},
	year = {2006}, 
	publisher = {Kyoto University}, 
	pages = {297--304},
	editor = {Yamamoto, Y.},
	booktitle = {Proceedings of the 17th International Symposium on Mathematical Theory of Networks and Systems (MTNS)},
	address = {Japan},
}

@inproceedings{RamS04,
	title = {On interconnections of infinite-dimensional port-{H}amiltonian systems},
	author = {Pasumarthy, R. and van der Schaft, A.},
	year = {2004},
	isbn = {9056825178},
	publisher = {Katholieke Universiteit Leuven},
	booktitle = {Proceedings of the 16th International Symposium on Mathematical Theory of Networks and Systems},
}

@phdthesis{Mor24phd,
	title = {Modeling and Numerical Treatment of Port-{{Hamiltonian}} Descriptor Systems},
	author = {Morandin, R.},
	year = {2024},
	school={Technische Universität Berlin},
	Doi		= {10.14279/depositonce-19826}
	}
%
%
%======================================================================
%=========  Appendix A
%======================================================================
\appendix
\section{Technical lemmas}
\label{app:technicalLemma}
We start with results for the difference operator $E u = u - u_\tau$.
\begin{lemma}
\label{lem:Ek_Lp_error}
	Let $u\in W^{k,p}(-k\tau,0;\R^n)$ with $k\geq 1$ and $p\in[1,\infty]$. Then it holds that
	\begin{equation}
		\norm{E^ku(0)} \leq \tau^{k-\frac{1}{p}}\norm{u^{(k)}}_{L^p}.
	\end{equation}
%	where $\frac{1}{p'}=1-\frac{1}{p}$.
\end{lemma}
\begin{proof}
	Let $\frac{1}{p'}=1-\frac{1}{p}$.
	We assume for simplicity $n=1$. For $n>1$ the result then follows immediately. %from
%	\[
%	\norm{E^ku(0)} = \pset*{\sum_{i=1}^n \abs{E^ku_i(0)}^2}^{\frac{1}{2}} \leq \tau^{k-1+\frac{1}{p'}}\pset*{\sum_{i=1}^n \norm{u_i^{(k)}}_{L^p}^2}^{\frac{1}{2}}
%	= \tau^{k-1+\frac{1}{p'}} \norm{u^{(k)}}_{L^p}^2.
%	\]
	For $p=\infty$, the inequality immediately follows by noting that
	\begin{multline*}
		\abs[\big]{E^ku(0)} = \abs[\big]{E^{k-1}\big(u(0)-u(-\tau)\big)} = \abs*{\int_{-\tau}^0 E^{k-1}\dot u(t_1)\dt_1} \\
		= \ldots = \abs*{\iint_{[-\tau,0]^k}u^{(k)}(t_1+\ldots+t_k)\, \mathrm{d}(t_1\times\ldots\times t_k)}
		\leq \tau^k\norm{u^{(k)}}_{L^\infty}.
	\end{multline*}
	Let us now assume that $p<\infty$.
	The inequality for $k=1$ immediately follows from Hölder's inequality, since
	\[
	\abs{Eu(0)}
	= \abs*{ \int_{-\tau}^0\dot u(t) \dt }
	\leq \norm{\dot u}_{L^1}
	\leq \norm{1}_{L^{p'}}\norm{\dot u}_{L^p}
	= \tau^{\frac{1}{p'}}\norm{\dot u}_{L^p}.
	\]
	We now prove the statement for $k\geq 2$ by induction.
	Let us assume that the theorem statement holds for $k-1$. Noting that $Eu=u-u_\tau\in W^{k-1,p}(-(k-1)\tau,0;\R)$ with $(Eu)^{(k-1)}=Eu^{(k-1)}$, we deduce that
	\[
	\abs[\big]{E^ku(0)}
	= \abs[\big]{E^{k-1}Eu(0)}
	\leq \tau^{k-2+\frac{1}{p'}}\norm[\big]{Eu^{(k-1)}}_{L^p}
	\]
	and
%	\begin{align*}
%		\norm[\big]{Eu^{(k-1)}}_{L^p}
%		&= \pset*{\int_{-(k-1)\tau}^0\abs*{ u^{(k-1)}(t) - u^{(k-1)}(t-\tau) }^p \dt}^{\frac{1}{p}} \\
%		&\leq \pset*{\int_{-(k-1)\tau}^0 \abs*{ \int_{-\tau}^0 u^{(k)}(t+s)\ds }^p\dt}^{\frac{1}{p}} \\
%		&\leq \tau^{\frac{1}{p'}}\pset*{\int_{-(k-1)\tau}^0 \int_{-\tau}^0 \abs*{ u^{(k)}(t+s) }^p \ds\dt}^{\frac{1}{p}},
%	\end{align*}
	\begin{align*}
		\norm[\big]{Eu^{(k-1)}}^p_{L^p}
		\leq \int_{-(k-1)\tau}^0 \abs*{ \int_{-\tau}^0 u^{(k)}(t+s)\ds }^p\dt 
		\leq \tau^{\frac{1}{p'}} \int_{-(k-1)\tau}^0 \int_{-\tau}^0 \abs*{ u^{(k)}(t+s) }^p \ds\dt,
	\end{align*}
	where we applied again Hölder's inequality.
	Then, by applying the change of variables $r=s+t$ and defining the measurable set
	\[
	K_r = \big\{ s \in [-\tau,0] \mid s-r \in [0,k\tau] \big\}
	\]
	and its Lebesgue measure $f(r)=\lambda(K_r)$, which fulfills $0\leq f(r)\leq\tau$, we can write
	\begin{multline*}
		\pset*{\int_{-(k-1)\tau}^0 \int_{-\tau}^0 \abs*{ u^{(k)}(t+s) }^p \ds\dt}^{\frac{1}{p}}
		= \pset*{\int_{-k\tau}^0 \int_{K_r} \abs*{ u^{(k)}(r) }^p \ds\dr}^{\frac{1}{p}} \\
		= \pset*{\int_{-k\tau}^0 \abs*{ u^{(k)}(r) }^p f(r) \dr}^{\frac{1}{p}}
		\leq \norm{f}_{L^{\infty}}^{\frac{1}{p}}\norm{u^{(k)}}_{L^p}
		= \tau^{\frac{1}{p}}\norm{u^{(k)}}_{L^p}.
	\end{multline*}
	Putting everything together, we obtain $\abs{E^ku(0)} \leq \tau^{k-1+\frac{1}{p'}}\norm{u^{(k)}}_{L^p}$.
\end{proof}
Next, we recall the approximation properties of the $\bdf$ formulae.
\begin{lemma}
\label{lem:BDF1_L2_error}
	Let $u\in W^{3,p}(-2\tau,0;\R)$ with $p\in[1,\infty]$. Then it holds that
	\begin{equation}
		\abs[\big]{\bdf_1 u(0) - \dot u(0)} \lesssim \tau^{2-\frac{1}{p}}\norm{u^{(3)}}_{L^p}.
	\end{equation}
\end{lemma}
\begin{proof}
	By the integral form of Taylor's remainder:
	\begin{align*}
		u(-\tau) &= u(0) - \tau\dot u(0) + \frac{\tau^2}{2}\ddot u(0) - \frac{1}{2}\int_{-\tau}^0(\tau+t)^2u^{(3)}(t) \dt, \\
		u(-2\tau) &= u(0) - 2\tau\dot u(0) + 2\tau^2\ddot u(0) - \frac{1}{2}\int_{-2\tau}^0(2\tau+t)^2u^{(4)}(t) \dt.
	\end{align*}
	Then
	\begin{multline*}
		\abs{\bdf_1u(0) - \dot u(0)}
		%		= \frac{1}{\tau^2}(E^2+E^3)u(0) - \ddot u(0) \\
		= \abs*{\frac{1}{\tau}\pset*{ \frac{3}{2}u(0) - 2u(-\tau) + \frac{1}{2}u(-2\tau) } - \dot u(0)} \\
		= \frac{1}{2\tau}\abs*{ \int_{-\tau}^0(\tau+t)^2u^{(3)}(t) \dt - \frac{1}{4}\int_{-2\tau}^0(2\tau+t)^2u^{(3)}(t) \dt}
		\\
		\lesssim \tau\int_{-2\tau}^0\abs{u^{(3)}}\dt
		\lesssim \tau^{2-\frac{1}{p}}\norm{u^{(3)}}_{L^p},
	\end{multline*}
	where the last step follows from Hölder's inequality.
\end{proof}

\begin{lemma}
\label{lem:BDF2_L2_error}
	Let $u\in W^{4,p}(-3\tau,0;\R)$ with $p\in[1,\infty]$. Then it holds that
	\begin{equation}
		\abs[\big]{\bdf_2 u(0) - \ddot u(0)} \lesssim \tau^{2-\frac{1}{p}}\norm{u^{(4)}}_{L^p}.
	\end{equation}
\end{lemma}
\begin{proof}
%	By the integral form of Taylor's remainder:
%	%
%	\begin{align*}
%		u(-\tau) &= u(0) - \tau\dot u(0) + \frac{\tau^2}{2}\ddot u(0) - \frac{\tau^3}{6}\dddot u(0) + \frac{1}{6}\int_{-\tau}^0(\tau+t)^3u^{(4)}(t) \dt, \\
%		u(-2\tau) &= u(0) - 2\tau\dot u(0) + 2\tau^2\ddot u(0) - \frac{4}{3}\tau^3\dddot u(0) + \frac{1}{6}\int_{-2\tau}^0(2\tau+t)^3u^{(4)}(t) \dt, \\
%		u(-3\tau) &= u(0) - 3\tau\dot u(0) + \frac{9}{2}\tau^2\ddot u(0) - \frac{9}{2}\tau^3\dddot u(0) + \frac{1}{6}\int_{-3\tau}^0(3\tau+t)^3u^{(4)}(t) \dt.
%	\end{align*}
%	%
%	Then
%	%
%	\begin{multline*}
%		\abs{\bdf_2u(0) - \ddot u(0)}
%%		= \frac{1}{\tau^2}(E^2+E^3)u(0) - \ddot u(0) \\
%		= \abs*{\frac{1}{\tau^2}\pset[\big]{ 2u(0) - 5u(-\tau) + 4u(-2\tau) - u(-3\tau) } - \ddot u(0)} \\
%		= \frac{1}{6\tau^2}\abs*{ -5\int_{-\tau}^0(\tau+t)^3u^{(4)}(t) \dt + 4\int_{-2\tau}^0(2\tau+t)^3u^{(4)}(t) \dt - \int_{-3\tau}^0(3\tau+t)^3u^{(4)}(t) \dt }
%		\\
%		\lesssim \tau\int_{-3\tau}^0\abs{u^{(4)}}\dt
%		\lesssim \tau^{2-\frac{1}{p}}\norm{u^{(4)}}_{L^p},
%	\end{multline*}
%	%
%	where the last step follows from Hölder's inequality.
	The proof is analogous to the one of \Cref{lem:BDF1_L2_error}.
\end{proof}
More generally, we denote by $\bdf_k=\tau^{-k}(E^k+\tfrac{k}{2}E^{k+1})$, $k\geq 0$, the second-order $\bdf$ formula for the $k$-th derivative, $k\geq 0$. %We omit the proof that this formula holds.
We now show that by testing $\bdf_{k+1}$ with $\bdf_k$ we obtain an expression for discrete sequences related to the formula $\tfrac{1}{2}\ddt\norm{u^{(k)}}^2=\aset{u^{(k)},u^{(k+1)}}$ in the continuous setting.
\begin{lemma}
\label{lem:bdf-bdf_formula}
	The formula
	\begin{multline}
		\tau^{2k+1}\aset{\bdf_kv^n,\bdf_{k+1}v^n} \\
		= \tfrac{(k+1)(k+2)}{8}\norm{E^{k+2}v^n}^2
		+ E\pset*{
			\tfrac{1}{2(k+2)}\norm{E^kv^n}^2
			+ \tfrac{k+1}{2(k+2)} \norm*{ E^kv + \tfrac{k+2}{2}E^{k+1}v^n }^2
		}
	\end{multline}
	holds for all $k\geq 0$, $n\geq k+1$, and every sequence $v^n$.
\end{lemma}

\begin{proof}
	By applying discrete integration by parts, we obtain
	\begin{align*}
		&\tau^{2k+1}\aset{\bdf_kv^n,\bdf_{k+1}v^n}
		= \aset[\big]{(E^k+\tfrac{k}{2}E^{k+1})v^n , E^{k+1}+\tfrac{k+1}{2}E^{k+2}v^n}
		\\
		&\ = \aset{E^kv^n,E^{k+1}v^n}
		+ \tfrac{k}{2}\norm{E^{k+1}v^n}^2
		+ \tfrac{k+1}{2}\aset{E^kv^n,E^{k+2}v^n}
		+ \tfrac{k(k+1)}{4}\aset{E^{k+1}v^n,E^{k+2}v^n}
		\\
		&\ = \aset{E^kv^n,E^{k+1}v^n}
		- \tfrac{1}{2}\norm{E^{k+1}v^n}^2
		+ \tfrac{(k+1)(k+2)}{4}\aset{E^{k+1}v^n,E^{k+2}v^n}
		+ \tfrac{k+1}{2}E\aset{E^kv^n,E^{k+1}v^n}
		\\
		&\ = \tfrac{(k+1)(k+2)}{8}\norm{E^{k+2}v^n}^2
		+ E\pset*{
			\tfrac{1}{2}\norm{E^kv^n}^2
			+ \tfrac{(k+1)(k+2)}{8}\norm{E^{k+1}v^n}^2
			+ \tfrac{k+1}{2}\aset{E^kv^n,E^{k+1}v^n}
		}.
	\end{align*}
	Then, by substituting
	\begin{align*}
		&\tfrac{k+1}{2}\,\aset{E^kv^n,E^{k+1}v^n}
		= 2\, \aset*{ \sqrt{\tfrac{k+1}{2(k+2)}}E^kv , \sqrt{\tfrac{(k+1)(k+2)}{8}}E^{k+1}v^n }
		\\
		&\ = \norm*{ \sqrt{\tfrac{k+1}{2(k+2)}}E^kv + \sqrt{\tfrac{(k+1)(k+2)}{8}}E^{k+1}v^n }^2
		- \tfrac{k+1}{2(k+2)}\, \norm{E^kv}^2
		- \tfrac{(k+1)(k+2)}{8}\, \norm{E^{k+1}v^n}^2
		\\
		&\ = \tfrac{k+1}{2(k+2)}\, \norm*{ E^kv + \tfrac{k+2}{2}E^{k+1}v^n }^2
		- \tfrac{k+1}{2(k+2)}\, \norm{E^kv}^2
		- \tfrac{(k+1)(k+2)}{8}\, \norm{E^{k+1}v^n}^2,
	\end{align*}
	we obtain the claimed formula. 
%	\begin{multline*}
%		\tau^{2k+1}\aset{\bdf_kv^n,\bdf_{k+1}v^n} \\
%		= \tfrac{(k+1)(k+2)}{8}\norm{E^{k+2}v^n}^2
%		+ E\pset*{
%			\tfrac{1}{2(k+2)}\norm{E^kv^n}^2
%			+ \tfrac{k+1}{2(k+2)} \norm*{ E^kv + \tfrac{k+2}{2}E^{k+1}v^n }^2
%		},
%	\end{multline*}
%	concluding the proof.
\end{proof}

We conclude this Appendix with a special discrete version of Grönwall's lemma.

\begin{lemma}\label{lem:discreteGronwallMinimal}
	Let $a_n,b_n\geq 0$ be two sequences %of nonnegative numbers 
	with $b_n$ nondecreasing. Further, let $\alpha\in(0,1)$ and $\kappa\in[0,1-\alpha]$ be constants. Suppose that
	$
		a_m \leq b_m + \kappa\sum_{n=0}^m a_n
	$
	for all $m\geq 0$. Then
	$
		a_m \leq e^{\frac{(m+1)\kappa}{\alpha}}b_m
	$
	holds for all $m\geq 0$.
\end{lemma}

\begin{proof}
By applying \cite[Prop.~4.1]{Emm99ppt}, we deduce $a_m\leq b_m(\frac{1}{1-\kappa})^{m+1}$. The claimed result then follows from
\[
	\frac{1}{1-\kappa} 
	= 1 + \frac{\kappa}{1-\kappa} 
	\leq 1 + \frac{\kappa}{\alpha} 
	\leq e^{\frac{\kappa}{\alpha}}.
	\qedhere
\]
%	\Cref{lem:discreteGronwall_final} with $\alpha_0=b_0,\ \alpha_n=Eb_n$ for $n\geq 1$, and $\kappa_n=\kappa\leq 1-\alpha<1$, we obtain
%	\[
%	a_m \leq \exp\pset*{ (m+1)\frac{\kappa}{1-\kappa} } b_m \leq e^{\frac{(m+1)\kappa}{\alpha}} b_m
%	\]
%	for all $m\geq 0$, concluding the proof.
\end{proof}
%
%
%
%======================================================================
%=========  Appendix B
%======================================================================
\section{On the coefficients of the semidiscretized extended model}
\label{app:semidiscExtModel}

As mentioned in \Cref{rem:semidiscExtModel}, the matrix $D_\Gamma\in\R^{N_\Gamma,N_\Gamma}$ obtained via the spatial discretization
%(see \Cref{app:advectionMatricesFEM} for implementation details)
does not satisfy $D_\Gamma\geq 0$ automatically, unless additional assumptions on the finite element method for $p$ or on the coefficients of the PDE are taken.
In presence of non-zero advection fields $\gamma_\Omega$ or $\gamma_\Gamma$, the condition \eqref{eq:adv_cond_surf}, which ensures the monotonicity of $\calR_\Gamma$, might fail in the discrete setting when $\Gamma_h\neq\Gamma$. In fact, the counterpart of $\calR_\Gamma$ for the semidiscretized problem would be
\[
\calR_\Gamma^h(p,q) = \delta_\Gamma\int_{\Gamma_h}\nabla_\Gamma p\cdot\nabla_{\Gamma_h} q\dx + \int_{\Gamma_h} \Big(\big(\alpha_\Gamma+\tfrac{1}{2}(\gamma_\Omega^h\cdot n_{\Gamma_h}-\nabla_{\Gamma_h}\cdot \gamma_\Gamma^h)\big)\,pq + \tfrac{1}{2}\nabla_{\Gamma_h}\cdot(\gamma_\Gamma^h pq)\dx\Big),
\]
where $\gamma_\Omega^h$ and $\gamma_\Gamma^h$ are appropriate approximations of the vector fields on $\Gamma$, obtained for example using lift operations or interpolation~\cite{HipK21}, and $n_{\Gamma_h}$ denotes the outward unit vector normal to $\Gamma_h$.
On the one hand, the integral $\int_{\Gamma_h}\nabla_{\Gamma_h}\cdot(\gamma_\Gamma^h pq)\dx$, which appears due to integration by parts but vanishes in the continuous setting, is in general nontrivial for polyhedral $\Gamma_h$.
On the other hand, the condition \eqref{eq:adv_cond_surf} does not necessarily guarantee $\alpha_\Gamma+\tfrac{1}{2}(\gamma_\Omega^h\cdot n_{\Gamma_h}-\nabla_{\Gamma_h}\cdot\gamma_\Gamma^h)\geq 0$.

One way out of this problem would be to use finite elements with curved boundaries. Another possibility is to replace the condition \eqref{eq:adv_cond_surf} by a stronger one, for example
\begin{equation}
	\alpha_\Gamma + \tfrac{1}{2}(\gamma_\Omega\cdot n - \nabla_\Gamma\cdot\gamma_\Gamma) 
	\geq c_0 \big(\norm{\gamma_\Gamma}_{L^\infty}+\norm{\calJ_{\gamma_\Omega}}_{H^1}\big),
\end{equation}
where $\calJ_{\gamma_\Omega}$ denotes the Jacobian of $\gamma_\Omega$, and $c_0>0$ is a sufficiently large constant, only depending on $\Gamma$. Then, if $P_k$-finite elements are used to approximate $p$ and the triangulation $\mathcal T_\Gamma$ is sufficiently fine, one can verify that $D_\Gamma+D_\Gamma^\top\geq 0$ holds (in fact, $D_\Gamma+D_\Gamma^\top>0$ as long as $\gamma_\Omega\neq 0$ or $\gamma_\Gamma\neq 0$). Different choices of finite elements could also help to ensure this condition.

The coefficient matrices depend on the coefficients of the extended model in the following way:
\begin{subequations}
	\begin{align}
		D_\Omega &= \delta_\Omega S_\Omega + J_\Omega(\gamma_\Omega) + W_\Omega(\alpha_\Omega,\gamma_\Omega), \\
		A_\Omega &= \beta_\Omega S_\Omega + \kappa_\Omega M_\Omega, \\
		D_\Gamma &= \delta_\Gamma S_\Gamma + J_\Gamma(\gamma_\Gamma) + W_\Gamma(\alpha_\Gamma,\gamma_\Gamma,\gamma_\Omega), \\
		A_\Gamma &= \beta_\Gamma S_\Gamma + \kappa_\Gamma M_\Gamma,
	\end{align}
\end{subequations}
where the matrices $S_\Omega,J_\Omega(\gamma_\Omega),W_\Omega(\alpha_\Omega,\gamma_\Omega)\in\R^{N_\Omega,N_\Omega}$ and $S_\Gamma,J_\Gamma(\gamma_\Gamma),W_\Gamma(\alpha_\Gamma,\gamma_\Gamma,\gamma_\Omega)\in\R^{N_\Gamma,N_\Gamma}$ are constructed using finite elements. In the following, we omit to write the explicit dependency of $J_\Omega$, $J_\Gamma$, $W_\Omega$, and $W_\Gamma$ from the system coefficients.
The matrices $S_\Omega$, $S_\Gamma$, $W_\Omega$, and $W_\Gamma$ are symmetric positive semidefinite, $J_\Omega$ and $J_\Gamma$ are skew-symmetric, and the bounds
\begin{subequations}
	\begin{align}
		\aset{v,S_\Omega u} 
		&\leq Ch^{-2}\norm{v}_{M_\Omega}\norm{u}_{M_\Omega}, \\
		%			\aset{q,p}_{S_\Gamma} &\leq C\frac{1}{h^2}\norm{q}_{M_\Gamma}\norm{p}_{M_\Gamma}, \\
		\aset{v,J_\Omega u} 
		&\leq C\, \norm{\gamma_\Omega}_{L^\infty}\big(\norm{v}_{M_\Omega}\norm{u}_{S_\Omega} + \norm{v}_{S_\Omega}\norm{u}_{M_\Omega}\big), \\
		\aset{v,J_\Omega u} 
		&\leq Ch^{-1}\norm{\gamma_\Omega}_{L^\infty}\norm{v}_{M_\Omega}\norm{u}_{M_\Omega} \\
		\aset{q,J_\Gamma p} 
		&\leq C\, \norm{\gamma_\Gamma}_{L^\infty}\big(\norm{q}_{M_\Gamma}\norm{p}_{S_\Gamma} + \norm{q}_{S_\Gamma}\norm{p}_{M_\Gamma}\big), \\
		\aset{v,W_\Omega u} 
		&\leq C \big(\alpha_\Omega+\norm{\nabla\cdot\gamma_\Omega}_{L^\infty}\big) \norm{v}_{M_\Omega}\norm{u}_{M_\Omega}, \\
		\aset{q,W_\Gamma p} 
		&\leq C \big(\alpha_\Gamma + \norm{\gamma_\Omega}_{L^\infty} + \norm{\gamma_\Gamma}_{W^{1,\infty}}\big) \norm{q}_{M_\Gamma}\norm{p}_{M_\Gamma}
	\end{align}
\end{subequations}
hold for all $u,v\in\R^{N_\Omega}$ and $p,q\in\R^{N_\Gamma}$, for some constant $C>0$ depending only on $\Gamma$.

\end{document}